\documentclass[a4paper,12pt]
{amsart}
\usepackage[utf8x]{inputenc}
\usepackage{amssymb,latexsym}
\usepackage{mathrsfs}
\usepackage{euscript}
\usepackage{hyperref}
\usepackage{amsmath}

\newtheorem{theorem}{Theorem}[section]
\newtheorem{definition}[theorem]{Definition}
\newtheorem{lemma}[theorem]{Lemma}
\newtheorem{corollary}[theorem]{Corollary}
\newtheorem{remark}[theorem]{Remark}
\newtheorem{proposition}[theorem]{Proposition}
\newtheorem{notation}[theorem]{Notation}

\def\cal{\mathcal}

\def\R{{\mathbb R}}

\def\<{\left<}
\def\>{\right>}

\def\S{{\cal S}}

\def\j{\jmath}

\def\BMO{{\rm BMO}}

\def\CMO{{\rm CMO}}

\def\({( \hspace{-0.335em}(}
\def\){) \hspace{-0.335em})}
\def\supp{{\rm supp\,}}

\def\fu2{\frac{n}{2}}

\def\l({\left(}
\def\r){\right)}
\def\be{\begin{enumerate}}
\def\ee{\end{enumerate}}

\def\sign{{\rm sign}}

\def \diam {{\rm diam}}

\def \dist {{\rm dist}}

\def \rdist {{\rm \, rdist}}
\def \ec {{\rm ec}}

\title{A characterization of compactness for singular integrals}
\author{Paco Villarroya}
\address{Centre for Mathematical Sciences, University of Lund, PO Box 118, 22100 Lund, Sweden}
\email{paco.villarroya@maths.lth.se}

\thanks{The last author has been partially supported by project MTM2011-23164}


\begin{document}

\maketitle

\begin{abstract}
We prove a $T(1)$ Theorem to completely characterize compactness of Calder\'on-Zygmund operators. 
The result provides sufficient and necessary conditions for the compactness of singular integral operators acting on $L^{p}(\mathbb R)$ with $1<p<\infty$.

\end{abstract}

\section{Introduction}
The theory of singular integrals started with the study of singular convolution operators such as the Hilbert and Riesz transforms. David and Journ\'e's famous $T(1)$ theorem \cite{DJ} marked significant progress in the theory, characterizing the boundedness of a larger class of singular integral operators, including operators of non-convolution type. In contrast with convolution operators, some of these operators are compact. Well known examples include certain Hankel operators \cite{Roch}, commutators of singular integrals and multiplication operators \cite{Uchi}, and layer potential operators \cite{Fabes}. The study of compactness of singular integral operators has been an area of active research, and continues to be so today with a variety of applications, many in the field of elliptic partial differential equations \cite{hmtimb}.  

The purpose of this paper is to develop a general theory of
compactness for a large class of singular integrals that, in
the spirit of David and Journ\'e, describes when a Calder\'on-Zygmund operator is compact
in terms of its action on special families of functions. We actually 
present a new $T(1)$ theorem which characterizes
the compactness of singular integral operators, in analogy with the
characterization of their boundedness given by the classical $T(1)$
theorem. More precisely, we show that a singular integral operator $T$
is compact on $L^p(\R)$ with $1<p<\infty $ 
if and only if its kernel satisfies the definition of what we call a compact 
Calder\'on-Zygmund kernel, the operator satisfies a new property of
\textit{weak compactness}, analogue to the classical \textit{weak boundedness},  
and the functions $T(1)$ and $T^*(1)$ belong to the space
$\CMO(\R)$, the appropriate substitute of $\BMO(\R)$.  

Our hypotheses
impose additional smoothness bounds on the kernel of $T$, beyond it being
of Calder\'on-Zygmund type. However, it is important to notice that these additional smoothness and decay 
conditions
are, on the one side, necessary and, on the other side, they can be of arbitrary size, and hence fully singular kernels are
within the scope of our theorem. 

%

The paper is structured as follows. In Section 2 we give necessary definitions and state our main result, Theorem \ref{Mainresult}.  
In Section 3 we prove the necessity of the hypotheses of Theorem \ref{Mainresult}: in Proposition 
\ref{necessityofacompactCZkernel} we show that compact Calder\'on-Zygmund operators are associated with compact Calder\'on-Zygmund kernels; 
in Proposition 
\ref{necessityofweakcompactness} we prove necessity of 
the weak compactness condition; and in Proposition 
\ref{necessity2} we show 
the membership of $T(1)$ and $T^{*}(1)$ in $\CMO(\mathbb R)$. 
All remaining sections are devoted to prove their sufficiency. In Section 4 we prove a fundamental lemma describing the action of the operator on bump functions. In Section 5 we demonstrate compactness on $L^{p}$, $1<p<\infty$, under special cancellation conditions. Finally, in Section 6 we construct the paraproducts which allow to show compactness in full generality. 

We would like to highlight here two surprising facts. The first one is the lack of use of the kernel decay of a standard 
Calder\'on-Zygmund kernels in our calculations. As proved in Lemmata \ref{smoothimpliesdecay1} and 
\ref{smoothimpliesdecay2}, 
this is because the kernel smoothness and the weak boundedness condition imply the decay condition. The second one is that, unlike other compact operators like Hilbert-Schimdt operators for instance,
compact Calder\'on-Zygmund operators are associated with kernels that satisfy precise pointwise decay estimates in the directions perpendicular and parallel to the diagonal 
(see Proposition \ref{necessityofacompactCZkernel}).

I express my gratitude in chronological order to Karl-Mikael Perfekt, Sandy Davie, Xavier Tolsa, Christoph Thiele, Tuomas Hyt\"onen, Fernando Soria,  Ignacio Uriarte and Andrei Stoica for their suggestions and comments which led to substantial improvements of this paper.  Further, I would like to especially thank Ja Young Kim for many interesting conversations. 

\section{Definitions and statement of the main result}
\subsection{Compact Calder\'on-Zygmund kernel}
\begin{definition}
We say that three bounded functions $L,S, D: [0,\infty )\rightarrow [0,\infty )$
constitute a set of admissible functions if 
the following limits hold
\begin{equation}\label{limits}
\lim_{x\rightarrow \infty }L(x)=\lim_{x\rightarrow 0}S(x)=\lim_{x \rightarrow \infty }D(x)=0
\end{equation}


\end{definition}


\begin{remark}\label{constants}
Since any fixed dilation of an admissible function $L_{\lambda}(x)=L(\lambda^{-1}x)$ is again admissible, we will often omit all universal constants appearing in the argument of these functions. 
\end{remark}

\begin{definition}\label{prodCZ}
Let $\Delta $ be the diagonal of $\mathbb R^{2}$.
Let $L,S,D$ be admissible functions. 

A function $K:(\R^{2} \setminus \Delta )\to \mathbb C$ is called a
compact Calder\'on-Zygmund kernel if it is bounded on compact sets of $\mathbb R^{2} \setminus \Delta$
and for some $0<\delta \leq 1$ and $C>0$, we have
\begin{multline*}
|K(t,x)-K(t',x')|
\le 
C\frac{(|t-t'|+|x-x'|)^\delta}{|t-x|^{1+\delta}}
L(|t-x|)S(|t-x|)
D(|t+x|)
\end{multline*}
whenever $2(|t-t'|+|x-x'|)<|t-x|$. 
\end{definition}


In order to simplify notation, we will often write 
$$
F(t,x)=L(|t-x|)S(|t-x|) D(|t+x|)
$$

To prove our main result, we follow the scheme of the original proof of the $T(1)$ Theorem. 
However, the way we organize these classical ideas and actually deliver our proof is relatively new. This is mostly evident in the surprising fact that we never use the decay condition of a standard Calder\'on-Zygmund kernel. And yet this does not imply that we are dealing with a larger class of kernels because, as we prove in the following technical lemma, the smoothness condition of the kernel implies the expected decay condition. 
Note that the proof in fact works for classical Calder\'on-Zygmund kernels as well. 

\begin{lemma}\label{smoothimpliesdecay1}
Let $K$ be a compact Calder\'on-Zygmund kernel as given in Definition \ref{prodCZ} such that 
$
\displaystyle{\lim_{|t-x|\to\infty} K(t,x) = 0}
$. 
Then
$$|K(t,x)|
\le \frac{1}{|t-x|}\tilde{F}(t,x)
$$
for all $(t,x)\in  \R^{2} \setminus \Delta $, with a possibly different function $\tilde{F}$ of the same type as $F$.
\end{lemma}
\proof
For every fixed $(t,x)\in  \R^{2} \setminus \Delta $ we denote by $B_{t,x}$ the set 
the points $(t',x')\in  \R^{2} \setminus \Delta $
such that $2(|t-t'|+|x-x'|)<|t-x|$, that is, the closed $l^{1}$-norm ball with centre 
$(t,x)$ and radius $|t-x|/2$.

Then, by the definition of a compact Calder\'on-Zygmund kernel, 
\begin{align*}
|K(t,x)| 
&\leq |K(t',x')|+ \frac{(|t-t'|+|x-x'|)^\delta}{|t-x|^{1+\delta}}F(t,x)\\
&\leq |K(t',x')|+ 2^{-\delta }\frac{1}{|t-x|}F(t,x)
\end{align*}
for all $(t',x')\in B_{t,x}$. Therefore 
$$
|K(t,x)|\leq \inf_{B_{t,x}} |K(t',x')|+\frac{C}{|t-x|}F(t,x)
$$


Now we fix $(t,x)\in  \R^{2} \setminus \Delta $. By symmetry we can assume that $t<x$. 
For fixed $0<\epsilon <1/3$, we consider the sequence 
$\{(t_{n},x_{n})\}_{n\geq 0}\subset \R^{2} \setminus \Delta $ defined by $t_{0}=t$, $x_{0}=x$ and
$$
t_{n}=t_{n-1}-(1-\epsilon )/4|t_{n-1}-x_{n-1}|
$$ 
$$
x_{n}=x_{n-1}+(1-\epsilon )/4|t_{n-1}-x_{n-1}|
$$
for all $n\geq 1$. Notice that this way 
$(t_{n}, x_{n})\in B_{t_{n-1}, x_{n-1}}$ and so, by iterating previous calculations,
we get
$$
|K(t,x)|\leq |K(t_{n},x_{n})|+C\sum_{k=0}^{n-1}\frac{1}{|t_{k}-x_{k}|}F(t_{k},x_{k})
$$

Moreover
$
|t_{n}-x_{n}|
=((3-\epsilon)/2)^{n}|t_{0}-x_{0}|>(4/3)^{n}|t-x|
$
while
$
|t_{n}+x_{n}|
=|t+x|
$. 
%
This way,  
$$
|K(t,x)|\leq \lim_{n\rightarrow \infty }|K(t_{n},x_{n})|+C\sum_{k=0}^{\infty}
\frac{1}{(4/3)^{k}|t-x|}F(t_{k},x_{k})
$$
with 
$$
F(t_{k},x_{k})
=L((3-\epsilon)/2)^{k}|t-x|)S((3-\epsilon)/2)^{k}|t-x|)D(|t+x|)
$$
This finally shows that
\begin{equation*}\label{decay}
|K(t,x)|\leq \frac{1}{|t-x|}\tilde{F}(t,x)
\end{equation*}
with $\tilde{F}(t,x)=C\sum_{k=0}^{\infty}
\frac{1}{(4/3)^{k}}F(t_{k},x_{k})$. 
Notice that an application of Lebesgue's Dominated Convergence Theorem shows that 
the function $\tilde{F}$ satisfies the same limit properties as $F$. 



\begin{remark}

We note that we do not need to impose the condition $\lim_{|t-x|\rightarrow \infty }K(t,x)=0$ in the Definition 
\ref{prodCZ} of a compact Calder\'on-Zygmund kernel because this limit is a consequence of the boundedness (or even weak boundedness) of the operator $T$ associated with the kernel $K$. We will show the proof of this fact at the end of subsection \ref{necessityofCCZkernel}.

\end{remark}

\begin{definition}\label{SchwartzN} 
For every $N\in \mathbb N$, $N\geq 1$, we define $\S_{N}(\R)$ to be the set of 
all functions 
$f\in {\mathcal C}^{N}(\mathbb R)$
such that 
$$
\| f\|_{m,n}=\sup_{x\in \mathbb R}|x|^{m}|f^{(n)}(x)|<\infty
$$ 
for all $m,n\in \mathbb N$ with $m,n\leq N$. Clearly, $\S_{N}(\R)$ 
equipped with the family of seminorms $\| \cdot \|_{m,n}$
is a Fr\'echet space. Then,   
we can also define its dual space $\S_{N}'(\R)$ equipped with the dual topology which turns out to be
a subspace of the space of tempered distributions. 
\end{definition}

\begin{definition}\label{intrep}
Let $T:\S_{N}(\R)\to \S_{N}'(\R)$ be a linear operator which is continuous with respect the topology of 
$\S_{N}(\R)$ for a fixed $N\geq 1$.

We say that 
$T$ is associated with a compact Calder\'on-Zygmund kernel $K$ if 
the action of $T(f)$ as a distribution satisfies the following 
integral representation
$$
\langle T(f), g\rangle =\int_{\R}\int_{\R} f(t)g(x) K(t,x)\, dt \, dx
$$
for all functions $f,g\in {\cal S_{N}(\R)}$
with disjoint compact supports.
\end{definition}

\begin{remark}
The hypothesis that $K$ is bounded on compact sets of $\mathbb R^{2} \setminus \Delta$ guarantees that the integral is absolutely convergent for functions with disjoint compact supports. 
\end{remark}
\subsection{Weak compactness condition}

\begin{definition}\label{defbump}
For $0<p\leq \infty $ and $N\in \mathbb N$, we say that a function $\phi \in {\mathcal S}_{N}(\mathbb R)$ is 
an $L^p(\mathbb R)$-normalized bump function adapted to $I$ with constant $C>0$ and order $N$,
if it satisfies 
$$|\phi^{(n)}(x)|\le C|I|^{-1/p-n}(1+|I|^{-1}|x-c(I)|)^{-N}, \ \ \ 0\leq  n \le N$$
for every interval $I\subset \mathbb R$, where we denote its centre by $c(I)$ and its length by $|I|$.
\end{definition}

The order of the bump functions
will always be denoted by $N$, even though its value might change from line to line.
We will often use the greek letters $\phi $, $\varphi $ for general bump functions while
we reserve the use of $\psi $ to denote bump functions with mean zero.
If not otherwise stated, we will usually assume that bump functions are $L^2(\mathbb R)$-normalized. 



In Proposition \ref{symmetricspecialcancellation}, we use the same concept for functions of two variables. By this we simply mean  
that the function satisfies the corresponding bounds of product tensor type:
$$|\partial_{1}^{n_{1}}\partial_{2}^{n_{2}}\phi(x_{1},x_{2})|\le C\prod_{i=1,2}|I_{i}|^{-1/2-n_{i}}(1+|I_{i}|^{-1}|x_{i}-c(I_{i})|)^{-N}, \ \ \ 0\leq  n \le N$$

A result we will often use is the following property of bump functions whose 
proof can be found in \cite{TLec}:

\begin{lemma}\label{decayofidentity} Let $I$, $J$ be intervals and let $\phi_I $, $\varphi_J$ be bump functions $L^2$-adapted to
$I$ and $J$ respectively with order $N$ and constant $C>0$. Then,
$$
|\langle \phi_I,\varphi_J\rangle|\leq C \left(\frac{\min(|I|,|J|)}{\max (|I|,|J|)}\right)^{1/2}
\left( \max(|I|,|J|)^{-1}{\rm diam}(I\cup J) \right)^{-N}
$$
Moreover, if $|J|\leq |I|$ and $\psi_J$ has mean zero then 
$$
|\langle \phi_I,\psi_J\rangle|\leq C \left(\frac{|J|}{|I|}\right)^{3/2}
\left( |I|^{-1}{\rm diam}(I\cup J) \right)^{-(N-1)}
$$
\end{lemma}

%

\begin{notation}\label{ecandrdist}
We introduce now some notation which will be frequently used throughout the paper.  
We denote by $\mathbb B=[-1/2,1/2]$ and $\mathbb B_{\lambda }=\lambda \mathbb B=[-\lambda/2,\lambda /2]$.
Given two intervals $I,J\subset \mathbb R$, 
we define $\langle I,J\rangle$ as the smallest interval containing $I\cup J$ and  
we denote its measure by $\diam(I\cup J)$. Notice that 
\begin{eqnarray*}
\diam(I\cup J) & \approx & |I|/2+|c(I)-c(J)|+|J|/2\\
& \approx & |I|+\dist(I,J)+|J|
\end{eqnarray*}
where $\dist(I,J)$ denotes the set distance between $I$ and $J$. Actually, 
$$
|I|/2+|c(I)-c(J)|+|J|/2\leq \diam(I\cup J)\leq 2(|I|/2+|c(I)-c(J)|+|J|/2)
$$ 
and
$$
2^{-1}(|I|+\dist(I,J)+|J|)\leq \diam(I\cup J)\leq |I|+\dist(I,J)+|J|
$$ 

We define
the relative distance between $I$ and $J$ by
$$\rdist(I,J)=\frac{\diam(I\cup J)}{\max(|I|,|J|)}$$ which is comparable $\max(1,n)$ where $n$ is the minumum number of times the larger interval needs to be shifted so that it contains the smaller one. Notice that 
\begin{eqnarray*}
\rdist(I,J) & \approx & 1+\max(|I|,|J|)^{-1}|c(I)-c(J)|\\
& \approx & 1+\max(|I|,|J|)^{-1}\dist(I,J)
\end{eqnarray*}

Finally, we define the eccentricity $I$ and $J$ as 
$$
\ec(I,J)=\frac{\min(|I|,|J|)}{\max (|I|,|J|)}
$$

\end{notation}

\begin{definition}\label{WB}
A linear operator $T : \S_{N}(\R) \to \S_{N}'(\R)$ with $N\geq 1$ satisfies the weak compactness condition, if 
there exist admissible functions $L, S, D$ 
such that: for every $\epsilon>0$ there 
exists $M\in \mathbb N$ so that for any interval $I$ and every pair $\phi_I, \varphi_I$ of
$L^2$-normalized bump functions adapted to $I$ with constant $C>0$ and order $N$, we have
\begin{equation}\label{weakcompactnessformula}
|\langle T(\phi_I),\varphi_I)\rangle |\lesssim C(L(2^{-M}|I|)S(2^{M}|I|)
D(M^{-1}\rdist(I,\mathbb B_{2^{M}}))+\epsilon )
\end{equation}
where the implicit constant only depends on the operator $T$. 

\end{definition}

\begin{remark}
In the main results of the paper, like Theorem \ref{Mainresult}, Proposition 
\ref{symmetricspecialcancellation} or Theorem \ref{L2bounds}, when we say that $T$ satisfies the weak compactness condition, we mean that there is an integer $N\geq 1$ sufficiently large depending on the operator or its kernel so that the operator can be defined $T : \S_{N}(\R) \to \S_{N}'(\R)$, it is continuous with respect the topology in 
$\S_{N}(\R) $ and it satisfies Definition \ref{WB} for that value of $N$. 
\end{remark}

There exist other equivalent formulations of the concept of weak compactness. Maybe the simplest one would 
be stating 
that $|\langle T(\phi_I),\varphi_I)\rangle |\leq C$ and 
$$
\lim_{N\rightarrow 0}\sup_{I\notin {\mathcal I}_{M}}|\langle T(\phi_I),\varphi_I)\rangle |=0
$$
where the definition of ${\mathcal I}_{M}$ is given in Definition \ref{Imdef}. 

However, we decided to keep the more detailed formulation in Definition \ref{WB} because 
it explicitly distinguishes the contribution of the operator given by the term $\epsilon>0$ in (\ref{weakcompactnessformula}) from 
the contribution of the lagom intervals (see the definition of the concept in Definition \ref{Imdef}) by means of the admissible functions (which are independent of $\epsilon $). Notice that (\ref{weakcompactnessformula}) is exactly the expression we find in Proposition 
\ref{necessityofweakcompactness}, where we prove the necessity of the weak compactness condition, with admissible functions which are also independent of the operator. 

\vskip10pt
From now on, we will denote
$$
F_{K}(I)=L_{K}(|I|)S_{K}(|I|)D_{K}(\rdist(I,\mathbb B))
$$
and 
$$
F_{W}(I;M)
=L_{W}(2^{-M}|I|)S_{W}(2^{M}|I|)D_{W}(M^{-1}\rdist(I,\mathbb B_{2^{M}})) 
$$
where $L_{K}$, $S_{K}$ and $D_{K}$ are the functions appearing in the definition of a compact 
Calder\'on-Zygmund kernel,
while $L_{W}$, $S_{W}$, $D_{W}$ and the constant $M$ are as in the definition of the weak compactness condition. Note 
that the value $M = M_{T,\epsilon}$ depends not only on $T$ but also on $\epsilon$. 

We will also denote $F(I;M)=F_{K}(I)+F_{W}(I;M)$,
$$
F_{K}(I_{1},\cdots, I_{n})=\big( \sum_{i=1}^{n}L_{K}(|I_{i}|)\big) \big( \sum_{i=1}^{n}S_{K}(|I_{i}|)\big)
\big(\sum_{i=1}^{n}D_{K}(\rdist(I_{i},\mathbb B))\big)
$$
\begin{eqnarray*}
\hspace{-1.5cm}
F_{W}(I_{1},\cdots, I_{n};M)&=&\big(\sum_{i=1}^{n}L_{W}(2^{-M}|I_{i}|)\big) \big(\sum_{i=1}^{n}S_{W}(2^{M}|I_{i}|)\big)\\
&&\big(\sum_{i=1}^{n}D_{W}(M^{-1}\rdist(I_{i},\mathbb B_{2^{M}}))\big)
\end{eqnarray*}
and 
$
F(I_{1},\cdots, I_{n};M)=F_{K}(I_{1},\cdots, I_{n})+F_{W}(I_{1},\cdots, I_{n};M)
$.

\subsection{Characterization of compactness}
To prove our results on compact singular integrals, we use the following 
characterization of compact operators in a Banach space with a Schauder basis (see \cite{Fab}).

\begin{theorem}\label{charofcompact}
Suppose that $\{e_{n}\}_{n\in \mathbb N}$ is a Schauder basis of a Banach space $E$. For each positive integer $k$, let $P_{k}$ be the canonical projection,
$$
P_{k}(\sum_{n\in \mathbb N}\alpha_{n}e_{n})=\sum_{n\leq k}\alpha_{n}e_{n}
$$
Then, a bounded linear operator $T:E \to E$ is compact if and only if $P_{k}\circ T$ converges to $T$ in operator norm.
\end{theorem}

\begin{definition} \label{Imdef}
For every $M\in \mathbb N$, let ${\cal I}_{M}$ be the family of intervals such that 
$2^{-M}\leq |I|\leq 2^{M}$ and 
$\rdist(I,\mathbb B_{2^{M}})\leq M$. Let ${\mathcal D}$ be the family of dyadic intervals of the real line. We also define ${\cal D}_{M}$ as the intersection of ${\cal I}_{M}$ with ${\mathcal D}$. 

For every fixed $M$, we will call the intervals in ${\cal I}_{M}$ and ${\cal D}_{M}$ as lagom intervals and dyadic lagom intervals respectively.
\end{definition}

\begin{remark}
Notice that $I\in {\cal D}_{M}$ implies that   $2^{-M}(2^{M}+|c(I)|)\leq M$ and then
$|c(I)|\leq (M-1)2^{M}$.  
Therefore, $I\subset \mathbb B_{M2^{M}}$ with
$2^{-M}\leq |I|$.  

On the other hand, $I\notin {\cal D}_{M}$ implies either $|I|>2^{M}$ or 
$|I|<2^{-M}$ or $2^{-M}\leq |I|\leq 2^{M}$ with $|c(I)|>(M-1)2^{M}$.
\end{remark}

Let $E$ be one of the following Banach spaces: the Lebesgue space $L^{p}(\mathbb R)$, $1<p<\infty $, the Hardy space $H^{1}(\mathbb R)$, or the space $\CMO(\R)$, to be introduced later as the closure in $\BMO(\mathbb R)$ of continuous functions vanishing at infinity. In each case, $E$ is equipped with smooth wavelet bases which are also Schauder bases (see \cite{HerWeiss} and Lemma \ref{lem:cmochar}). Moreover, in all cases we have at our disposal smooth and compactly supported wavelet bases. 
 
\begin{definition}\label{lagom}
Let $E$ be one of the previously mentioned Banach spaces.  
Let $(\psi_{I})_{I\in {\mathcal D}}$ be a wavelet basis of $E$. 
Then, for every $M\in \mathbb N$,
we define the lagom projection operator $P_{M}$ by 
$$
P_{M}(f)=\sum_{I\in {\cal D}_{M}}\langle f,\psi_{I}\rangle \psi_{I}
$$
where $\langle f,\psi_{I}\rangle =\int_{\mathbb R}f(x)\overline{\psi(x)}dx$.

We also define the orthogonal lagom projection operator as $P_{M}^{\perp }(f)=f-P_{M}(f)$. 
\end{definition}

\begin{remark}
Without explicit mention, we will let the wavelet basis defining $P_M$ vary from proof to proof to suit our technical needs.

We also note the use of the same notation for the action of $T(f)$ as a distribution
and the inner product.
We hope that this will not cause confusion. 
\end{remark}

It is easy to see that both $P_M$ and $P_{M}^{\perp }$ are self-adjoint operators.

We observe that in $E$, the equality
\begin{equation}\label{ortho}
P_{M}^{\perp}(f)=\sum_{I\in {\cal D}_{M}^{c}}\langle f,\psi_{I}\rangle \psi_{I} 
\end{equation}
is to be interpreted in its Schauder basis sense,
$$
\lim_{M'\rightarrow \infty }\| P_{M}^{\perp}(f)
-\sum_{I\in {\cal D}_{M'}\backslash {\cal D}_{M}}\langle f,\psi_{I}\rangle \psi_{I}  \|_{E}=0
$$
Note that according to Theorem \ref{charofcompact}, an operator $T: E \to E$, is compact if and only if 
$$
\lim_{M\rightarrow \infty }\| P_{M}^{\perp}\circ T\|=0
$$
where $\| \cdot \|$ is the operator norm.


\subsection{The space $\CMO$}
We provide now the definition 
of the space to which the function $T(1)$ must belong if $T$ is compact.
%
%
%
%

\begin{definition}
We define $\CMO(\mathbb R)$ as the closure in $\BMO(\mathbb R)$ of the space of continuous functions vanishing at infinity. 
\end{definition}

The next lemma gives two characterizations of $\CMO(\mathbb R)$: the first in terms of the average deviation from the mean (see \cite{Roch} for a proof), and the second in terms of a  
wavelet decomposition (see  \cite{LTW}).

\begin{lemma} \label{lem:cmochar}
i) $f\in \CMO(\mathbb R)$ if and only if 
$f\in \BMO(\mathbb R)$ and
\begin{equation}\label{CMO}
\lim_{M\rightarrow \infty } \sup_{I\notin {\mathcal I}_{M}} \frac{1}{|I|}\int_{I}\Big|f(x)-\frac{1}{|I|}\int_{I}f(y)dy\Big|dx =0\\
\end{equation}

ii)
$f\in \CMO(\mathbb R)$ if and only if 
$f\in \BMO(\mathbb R)$ and
\begin{equation}\label{CMO2}
\lim_{M\rightarrow \infty} 
\sup_{\Omega \subset \mathbb R} \Big(\frac{1}{|\Omega |}
\sum_{\tiny \begin{array}{c}I\notin {\cal D}_{M}\\ I\subset \Omega \end{array}}|\langle f,\psi_{I}\rangle |^{2}\Big)^{1/2}=0
\end{equation}
\end{lemma}
We will mainly be using the latter formulation.

\subsection{Main result}
We now give the statement of our main result, leaving for forthcoming sections the exact meaning of $T(1)$ and $T^*(1)$. 

\begin{theorem}\label{Mainresult}
Let
$T$ be a linear operator associated with a standard Calder\'on-Zygmund kernel. 

Then, $T$ extends to a compact operator on $L^p(\mathbb R)$ for all $p$ with $1<p<\infty$ if and only if $T$ 
is associated with a compact Calder\'on-Zygmund kernel and satisfies
the weak compactness condition
and the cancellation conditions
$T(1), T^{*}(1) \in \CMO(\mathbb R)$. 
\end{theorem}

\begin{remark} The proof of Theorem \ref{Mainresult}, carried out mainly in Theorem \ref{L2bounds} and Proposition \ref{paraproducts1}, shows more than it is explicitly stated. It actually proves that the operator $T$ can be approximated by finite range operators $P_{M}\circ T$ which also Calder\'on-Zygmund operators.  

Moreover, we prove that for every $\epsilon > 0$ there is an $M_{0}\in \mathbb N$ such that for every $M>M_{0}$, 
$$
\| P_{M}^{\perp}\circ T\|_{p\rightarrow p}\lesssim C_{p}\Big(
C_{\delta}(C_{K}+C_{W})\big( C(\sup_{I\notin {\mathcal D}_{M}}F(I;M_0) + \epsilon)
+M^{-\delta}\|F\|_{\infty }\big)
$$
$$
+\| (P_{M}^{\perp}\circ T)(1)\|_{\BMO}+\| (P_{M}^{\perp}\circ T^{*})(1)\|_{\BMO}\Big)
$$
where the implicit constant is universal, $C_{p}$ is the constant of boundedness of the square function on $L^{p}(\mathbb R)$,  $C_{\delta}$ a constant depending on $\delta $, the exponent in the Definition \ref{prodCZ} of a compact Calder\'on-Zygmund kernel, and 
$C_{K}$ and $C_{W}$ are the constants appearing in the definitions of a compact Calder\'on-Zygmund kernel and the weak compactness condition, respectively.
\end{remark}

\section{Necessity of the hypotheses}

We prove the necessity of the hypotheses of Theorem \ref{Mainresult}:
the compact Calder\'on-Zygmund kernel, 
the weak compactness condition and the membership of $T(1)$ and $T^{*}(1)$ in $\CMO(\mathbb R)$. We note that, in analogy with the classical theory, necessity of the weak compactness condition holds for any compact operator on $L^{p}(\mathbb R)$ whose adjoint also defines a bounded operator on $L^p$. On the other hand, the necessity of the space $\CMO(\mathbb R)$ can only be shown for operators associated with compact Calder\'on-Zygmund kernels. 

\subsection{Necessity of a compact Calder\'on-Zygmund kernel}\label{necessityofCCZkernel}
We show that all Calder\'on-Zygmund operators that extend compactly on $L^{p}(\mathbb R)$ are associated with a compact Calder\'on-Zygmund kernel as defined in 
Definition \ref{prodCZ}.
We start the proof with a technical lemma. 

\begin{lemma}Let $\Phi$ be a positive smooth function such that it is 
supported and $L^{\infty}$-adapted to $[-1/2,1/2]$ and 
$\int \Phi (x)dx=1$.

Let $K$ be a standard Calder\'on-Zygmund kernel with constant $C>0$ and parameter $0<\delta\leq 1$ 
and let $0<\epsilon <C/2^{\delta }$. 
Let $(t,x)\in \mathbb R^{2}\backslash \Delta $
and $t'\in \mathbb R$ such that $3|t-t'|<|t-x|$.

Then, for every $0<\lambda_{i}<(C^{-1}\epsilon)^{1/\delta} |t-t'|$, we have
\begin{equation}\label{approximation'}
\Big|\int \int {\mathcal T}_{t}{\mathcal D}_{\lambda_{1}}^{1}\Phi(u)
{\mathcal T}_{x}{\mathcal D}_{\lambda_{2}}^{1}\Phi(y)K(u,y)dudy-K(t,x)\Big|
<\epsilon \frac{|t-t'|^{\delta }}{|t-x|^{1+\delta }}
\end{equation}
\begin{align}\label{approximation'2}
|\langle T({\mathcal T}_{t}{\mathcal D}_{\lambda_{1} }^{1}\Phi)-T({\mathcal T}_{t'}{\mathcal D}_{\lambda_{1}}^{1}\Phi),
{\mathcal T}_{x}{\mathcal D}_{\lambda_{2} }^{1}\Phi\rangle |
\lesssim C\frac{|t-t'|^{\delta }}{|t-x|^{1+\delta }}
\end{align}

\end{lemma}
\proof 
For every $0<\lambda_{i} <|t-x|/2$ 
and $u,v \in \supp {\mathcal D}_{\lambda}^{1}\Phi $, 
we have that  
$2(|u|+|y|)<2(\lambda_{1}+{\lambda_{2})/2}
<|t-x|$ and so, 
\begin{align*}
\Big|\int \int {\mathcal T}_{t}&{\mathcal D}_{\lambda}^{1}\Phi(u)
{\mathcal T}_{x}{\mathcal D}_{\lambda}^{1}\Phi(y)K(u,y)dudy-K(t,x)\Big|
\\
&=\Big|\int \int {\mathcal D}_{\lambda}^{1}\Phi(u){\mathcal D}_{\lambda}^{1}\Phi(y)(K(u+t,y+x)-K(t,x))dudy\Big|
\\
&\leq C\int \int {\mathcal D}_{\lambda}^{1}\Phi(u){\mathcal D}_{\lambda}^{1}\Phi(y)
\frac{(|u|+|y|)^{\delta }}{|t-x|^{1+\delta}}dudy
\leq C\frac{((\lambda_{1}+\lambda_{2})/2)^{\delta }}{|t-x|^{1+\delta }}
\end{align*}
This implies that, for $0<\lambda_{i}<(C^{-1}\epsilon)^{1/\delta} |t-t'|$,
we get \eqref{approximation'}.

On the other hand, for $3|t-t'|<|t-x|$, $0<\lambda_{1},\lambda_{2}<(C^{-1}\epsilon)^{1/\delta }|t-t'|$ and 
$u\in \supp {\mathcal D}_{\lambda_{1}}^{1}\Phi $ ,$y \in \supp {\mathcal D}_{\lambda_{2}}^{1}\Phi $, 
we have 
\begin{align*}
|u+t-(y+x)|&\geq |t-x|-|u|-|y|\geq |t-x|-(\lambda_{1}+\lambda_{2})/2 
\\
&\geq (3-(C^{-1}\epsilon)^{1/\delta })|t-t'|>
2|t-t'|
\end{align*} 
by the choice of $\epsilon $.
Then, $\supp {\mathcal T}_{t}{\mathcal D}_{\lambda_{1} }^{1}\Phi$ and 
$\supp {\mathcal T}_{t'}{\mathcal D}_{\lambda_{1} }^{1}\Phi$ are disjoint with 
$\supp {\mathcal T}_{x}{\mathcal D}_{\lambda_{2} }^{1}\Phi$ and hence, 
\begin{align*}
|\langle &T({\mathcal T}_{t}{\mathcal D}_{\lambda_{1} }^{1}\Phi)-T({\mathcal T}_{t'}{\mathcal D}_{\lambda_{1}}^{1}\Phi),
{\mathcal T}_{x}{\mathcal D}_{\lambda_{2} }^{1}\Phi\rangle |
\\
&= \Big|\int \int (
{\mathcal T}_{t}{\mathcal D}_{\lambda_{1} }^{1}\Phi(u)-{\mathcal T}_{t'}{\mathcal D}_{\lambda_{1} }^{1}\Phi(u))
{\mathcal T}_{x}{\mathcal D}_{\lambda_{2} }^{1}\Phi(y)K(u,y)dudy\Big|
\\
&\leq \int \int {\mathcal D}_{\lambda_{1} }^{1}\Phi(u){\mathcal D}_{\lambda_{2} }^{1}\Phi(y)
|K(u+t,y+x)-K(u+t',y+x)|dudy
\\
&\leq C\int \int {\mathcal D}_{\lambda_{1} }^{1}\Phi(u){\mathcal D}_{\lambda_{2} }^{1}\Phi(y)
\frac{|t-t'|^{\delta }}{|u+t-(y+x)|^{1+\delta}}dudy
\lesssim C\frac{|t-t'|^{\delta }}{|t-x|^{1+\delta }}
\end{align*}
The last inequality holds because $\lambda_{1},\lambda_{2}<|t-x|/3$ and so, 
$|u+t-(y+x)|\geq |t-x|-(\lambda_{1}+\lambda_{2})/2 \geq |t-x|/2$.

Now we prove the main result. It is well known that all Hilbert-Schmidt operators are compact on $L^{2}(\mathbb R)$ and that, obviously, the fact that the kernel $K$ belongs to $L^{2}(\mathbb R^{2})$ does not imply any pointwise decay. However, if $K$ is a standard Calder\'on-Zygmund and so, it decays pointwise in the direction perpendicular to the diagonal then, 
Proposition \ref{necessityofacompactCZkernel} tells us that 
compactness implies also pointwise decay in the direction parallel to the diagonal.
\begin{proposition}\label{necessityofacompactCZkernel}
Let $T$ be a linear operator associated with a standard Calder\'on-Zygmund kernel $K$ with parameter $0<\delta \leq 1$. Let $1<p<\infty $. If $T$ extends compactly on 
$L^{p}(\mathbb R)$, then $K$ is a compact Calder\'on-Zygmund kernel.
\end{proposition}

\proof By symmetry, in order to prove that $K$ is a compact Calder\'on-Zygmund kernel, it suffices to show that there is $0<\delta' <\delta $ such that 
$$
|K(t,x)-K(t',x)|
\lesssim \frac{|t-t'|^{\delta' }}{|t-x|^{1+\delta'}}L(|t-x|)S(|t-t'|)D\Big(1+\frac{|t+x|}{1+|t-x|}\Big)
$$
when $2|t-t'|<|t-x|$.
Actually, we are going to prove that 
the expression 
\begin{equation}\label{supremum}
\sup_{\tiny \begin{array}{c}t'\in \mathbb R \\ 0<2|t-t'|<|t-x|\end{array}}\hspace{-.5cm}
\frac{|t-x|^{1+\delta'}}{|t-t'|^{\delta' }}|K(t,x)-K(t',x)|
\end{equation}
tends to zero when $|t-x|$ tends to infinity or when $|t-t'|$ tends to zero or when $|t+x|$ tends to infinity while 
$|t-x|$ is bounded above and below. 
For every $0<\epsilon <1$ there exists $t'=t'(t,x,\epsilon )$ such that $2|t-t'|<|t-x|$ and the expression in \eqref{supremum} is bounded by
$$
\frac{|t-x|^{1+\delta'}}{|t-t'|^{\delta' }}|K(t,x)-K(t',x)|+\epsilon
$$
Then, we can always assume that $t'$ depends on $t, x$ and we just need to prove that 
\begin{equation}\label{nonsupremum}
C(t,x)=\frac{|t-x|^{1+\delta'}}{|t-t'|^{\delta' }}|K(t,x)-K(t',x)|
\end{equation}
tends to zero in the three stated cases for any $0<\delta'<\delta$. We note that we can assume, 
without loss of generality, that $t'\neq t$ even if the supremum is attained when $t'=t$.
We also note that 
$C(t,x)$ and the function defined by \eqref{supremum} are bounded functions.

Since $K$ is a standard Calder\'on-Zygmund kernel with constant $C>0$ and parameter $0<\delta\leq 1$, 
we have
$$
C(t,x)
\leq C\frac{|t-t'|^{\delta-\delta' }}{|t-x|^{\delta-\delta'}}
$$
for $2|t-t'|<|t-x|$. 

Therefore, since $\delta'<\delta $, we trivially get $\lim_{n\rightarrow \infty }C(t_{n}, x_{n})=0$
for every sequence $(t_{n},x_{n})_{n\in \mathbb N}$ and $t_{n}'=t'(t_{n},x_{n})$ as described before, such that 
\begin{equation}\label{quotient}
\lim_{n\rightarrow \infty }\frac{|t_{n}-t_{n}'|}{|t_{n}-x_{n}|}=0
\end{equation}
and one of the following limits hold: 
\begin{enumerate}
\item$\lim_{n\rightarrow \infty }|t_{n}-x_{n}|=\infty$,
\item or 
$\lim_{n\rightarrow \infty }|t_{n}-t_{n}'|=0$,
\item or $\lim_{n\rightarrow \infty }|t_{n}+x_{n}|=\infty $ with $d\leq |t_{n}-x_{n}|\leq 2d$ for some $d>0$
\end{enumerate} 
Hence, we only need to calculate the limit $\lim_{n\rightarrow \infty }C(t_{n}, x_{n})$
when (\ref{quotient}) does not hold. In this case we can assume, taking a subsequence of $(t_{n},x_{n})_{n\in \mathbb N}$ if necessary, that 
there is $\tilde{C}>2$ so that 
\begin{equation}\label{bothequivalent}
\tilde{C}^{-1}|t_{n}-x_{n}|\leq |t_{n}-t_{n}'|\leq |t_{n}-x_{n}|/2
\end{equation}
for all $n\in \mathbb N$. 
In this situation, 
the second limit $\lim_{n\rightarrow \infty }|t_{n}-t_{n}'|=0$ can be changed without loss of generality by 
$\lim_{n\rightarrow \infty }|t_{n}-x_{n}|=0$. 

Let $(t_{n},x_{n})_{n\in \mathbb N}$ and $t_{n}'=t'(t_{n},x_{n})$ satisfying \eqref{bothequivalent} and one of the three previous limits. 
Then, 
$$
C(t_{n},x_{n})\leq \tilde{C}^{\delta'}|t_{n}-x_{n}||K(t_{n},x_{n})-K(t_{n}',x_{n})|
$$
Let $0<\epsilon <C/2^{\delta }$ be fixed and arbitrarily small. We fix
a sequence $(\lambda_{n})_{n\in \mathbb N}$ so that 
$(C^{-1}\epsilon)^{1/\delta}|t_{n}-t_{n}'|/2\leq \lambda_{n}\leq (C^{-1}\epsilon)^{1/\delta}|t_{n}-t_{n}'|$. 
By \eqref{approximation'} and the second inequality in \eqref{bothequivalent}, we get
\begin{align*}
C(t_{n},x_{n})\leq \tilde{C}^{\delta'}|t_{n}-x_{n}|
\Big|\int \int {\mathcal T}_{t_{n}}{\mathcal D}_{\lambda_{n} }^{1}\Phi(u)
{\mathcal T}_{x_{n}}{\mathcal D}_{\lambda_{n} }^{1}\Phi(y)K(u,y)dudy
\\
-\int \int {\mathcal T}_{t_{n}'}{\mathcal D}_{\lambda_{n} }^{1}\Phi(u)
{\mathcal T}_{x_{n}}{\mathcal D}_{\lambda_{n} }^{1}\Phi(y)K(u,y)dudy \Big|
+2\tilde{C}^{\delta'}\epsilon 2^{-\delta}
\end{align*}
Moreover, $|u-y|\geq |t_{n}-x_{n}|-|u-t_{n}|-|y-x_{n}|\geq |t_{n}-x_{n}|-\lambda_{n} \geq |t_{n}-x_{n}|/2>0$ and similarly in the second integral changing $t_{n}$ by $t_{n}'$. Then, 
${\mathcal T}_{t_{n}}{\mathcal D}_{\lambda_{n} }^{1}\Phi$ and
${\mathcal T}_{x_{n}}{\mathcal D}_{\lambda_{n} }^{1}\Phi$ have disjoint support and, 
by the integral representation, we can write
\begin{align*}
C(t_{n},x_{n})\leq \tilde{C}^{\delta'}|t_{n}-x_{n}|
|\langle T({\mathcal T}_{t_{n}}{\mathcal D}_{\lambda_{n} }^{1}\Phi)-T({\mathcal T}_{t_{n}'}{\mathcal D}_{\lambda_{n} }^{1}\Phi),
{\mathcal T}_{x_{n}}{\mathcal D}_{\lambda_{n} }^{1}\Phi\rangle |+2\tilde{C}^{\delta'}\epsilon 
\end{align*}

We remind that $T$ is compact on 
$L^{p}(\mathbb R)$ with $1<p<\infty $. Then, 
let $f_{n}=|t_{n}-x_{n}|^{1/p'}({\mathcal T}_{t_{n}}{\mathcal D}_{\lambda_{n}}^{1}\Phi-{\mathcal T}_{t_{n}'}{\mathcal D}_{\lambda_{n}}^{1}\Phi)$ and
$g_{n}=|t_{n}-x_{n}|^{1/p}{\mathcal T}_{x_{n}}{\mathcal D}_{\lambda_{n}}^{1}\Phi$
for every $n\in \mathbb N$.
This way, 
$$
C(t_{n},x_{n})
\leq \tilde{C}^{\delta' }
|\langle T(f_{n}),g_{n}\rangle |+2\tilde{C}^{\delta'}\epsilon 
$$


By \eqref{bothequivalent} and the choice of $\lambda_{n}$, we have
\begin{align*}
\| f_{n}\|_{L^{p}(\mathbb R)}&\leq |t_{n}-x_{n}|^{\frac{1}{p'}}2\lambda_{n}^{-\frac{1}{p'}}
\\
&\leq |t_{n}-x_{n}|^{\frac{1}{p'}}2((C^{-1}\epsilon)^{\frac{1}{\delta}} |t_{n}-t_{n}'|/2)^{-\frac{1}{p'}}
\\
&\leq |t_{n}-x_{n}|^{\frac{1}{p'}}2^{1+\frac{1}{p'}}(C\epsilon^{-1})^{\frac{1}{p'\delta}}
(\tilde{C}^{-1} |t_{n}-x_{n}|)^{-\frac{1}{p'}}
\\
&= \tilde{C}^{\frac{1}{p'}}2^{1+\frac{1}{p'}}(C\epsilon^{-1})^{\frac{1}{p'\delta}}
\end{align*}
Hence, $(f_{n})_{n\in \mathbb N}$ is a bounded sequence and, by compactness of $T$ on 
$L^{p}(\mathbb R)$, there is a convergent subsequence $(T(f_{n_{k}}))_{k\in \mathbb N}$. 
Then, for 
$0<\tilde{\epsilon}<\epsilon^{1+\frac{1}{p\delta }}\tilde{C}^{-(\delta' +\frac{1}{p})}2^{-\frac{1}{p}}
C^{-\frac{1}{p\delta}}$, 
there is $k_{\epsilon }\in \mathbb N$ such that 
for all $k,l>k_{\epsilon }$, we have 
\begin{equation*}\label{uscomp}
\| T(f_{n_{k}})-T(f_{n_{l}})\|_{L^{p}(\mathbb R)}
<\tilde{\epsilon } 
\end{equation*}
With this, we write
\begin{align*}
\nonumber
&C(t_{n_{k}},x_{n_{k}})
\leq \tilde{C}^{\delta' }
|\langle T(f_{n_{k}}),g_{n_{k}}\rangle |+2\tilde{C}^{\delta'}\epsilon
\\
&\leq \tilde{C}^{\delta' }
(|\langle T(f_{n_{k}})-T(f_{n_{l}}),g_{n_{k}}\rangle |+|\langle T(f_{n_{l}}),g_{n_{k}}\rangle |)
+2\tilde{C}^{\delta'}\epsilon
\end{align*}
for all $k,l>k_{\epsilon }$.
Then, the first term 
can be bounded by
\begin{align*}
\tilde{C}^{\delta' }\| T(f_{n_{k}})-T(f_{n_{l}})\|_{L^{p}(\mathbb R)}\| g_{n_{k}}\|_{L^{p'}(\mathbb R)}
&\leq  \tilde{C}^{\delta' }\tilde{\epsilon}
|x_{n_{k}}-t_{n_{k}}|^{\frac{1}{p}}\lambda_{n_{k}}^{-\frac{1}{p}}
\\
&\leq  \tilde{C}^{\delta' }\tilde{\epsilon}
\tilde{C}^{\frac{1}{p}}2^{\frac{1}{p}}(C\epsilon^{-1})^{\frac{1}{p\delta}}
<\epsilon
\end{align*}
by the choice of $\tilde{\epsilon }$. 

On the other hand, we rewrite the second term 
as
\begin{align}\label{secterm}
\nonumber
\tilde{C}^{\delta' }|t_{n_{l}}-x_{n_{l}}|^{\frac{1}{p'}}&|t_{n_{k}}-x_{n_{k}}|^{\frac{1}{p}}
\\
|\langle T&({\mathcal T}_{t_{n_{l}}}{\mathcal D}_{\lambda_{n_{l}} }^{1}\Phi)-T({\mathcal T}_{t_{n_{l}}'}{\mathcal D}_{\lambda_{n_{l}} }^{1}\Phi),
{\mathcal T}_{x_{n_{k}}}{\mathcal D}_{\lambda_{n_{k}} }^{1}\Phi\rangle |
\end{align}
and we bound this expression
in different ways accordingly with the properties of the sequence 
$(t_{n},x_{n})_{n\in \mathbb N}$. But first, we remind that since $T$ is compact on $L^{p}(\mathbb R)$, it is a Calder\'on-Zygmund operator bounded 
on $L^{p}(\mathbb R)$ and, therefore, by the classical theory, $T$ is bounded on $L^{q}(\mathbb R)$
for any $1<q<\infty$. 

1) We first consider the case when $(t_{n},x_{n})_{n\in \mathbb N}$ satisfies \eqref{bothequivalent} and 
$\lim_{n\rightarrow \infty }|t_{n}-x_{n}|=\infty $. 

We choose $q>p$. Then, by denoting $C_{T,q}$ the constant of boundedness of $T$ on $L^{q}(\mathbb R)$, we can bound the last factor in \eqref{secterm} by 
\begin{align*}
C_{T,q}(\| {\mathcal T}_{t_{n_{l}}}&{\mathcal D}_{\lambda_{n_{l}} }^{1}\Phi\|_{L^{q}(\mathbb R)}+
\| {\mathcal T}_{t_{n_{l}}'}{\mathcal D}_{\lambda_{n_{l}} }^{1}\Phi\|_{L^{q}(\mathbb R)})
\| {\mathcal T}_{x_{n_{k}}}{\mathcal D}_{\lambda_{n_{k}} }^{1}\Phi\|_{L^{q'}(\mathbb R)}
\\
&
\leq  2C_{T,q}
\lambda_{n_{l}}^{-\frac{1}{q'}}\lambda_{n_{k}}^{-\frac{1}{q}}
\\
&\leq 2C_{T,q}
((C^{-1}\epsilon)^{\frac{1}{\delta }}|t_{n_{l}}-t_{n_{l}}'|/2)^{-\frac{1}{q'}}
((C^{-1}\epsilon)^{\frac{1}{\delta }}|t_{n_{k}}-t_{n_{k}}'|/2)^{-\frac{1}{q}}
\\
&\leq C_{T,q}(C\epsilon^{-1})^{\frac{1}{\delta }}
(\tilde{C}^{-1}|t_{n_{l}}-x_{n_{l}}|)^{-\frac{1}{q'}}
(\tilde{C}^{-1}|t_{n_{k}}-x_{n_{k}}|)^{-\frac{1}{q}}
\end{align*}
This way, \eqref{secterm} is bounded by
$$
\tilde{C}^{1+\delta' }C_{T,q}(C\epsilon^{-1})^{\frac{1}{\delta }}
\frac{1}{|t_{n_{l}}-x_{n_{l}}|^{\frac{1}{q'}-\frac{1}{p'}}}|t_{n_{k}}-x_{n_{k}}|^{\frac{1}{p}-\frac{1}{q}}
$$

Now, since $\frac{1}{q'}-\frac{1}{p'}\geq 0$, we fix $k>k_{\epsilon }$ and choose $l>k$ such that
$|t_{n_{l}}-x_{n_{l}}|^{\frac{1}{q'}-\frac{1}{p'}}>\epsilon^{-1}\tilde{C}^{1+\delta' }C_{T,q}(C\epsilon^{-1})^{\frac{1}{\delta }}|t_{n_{k}}-x_{n_{k}}|^{\frac{1}{p}-\frac{1}{q}}$. This implies that last expression and \eqref{secterm} are bounded by $\epsilon $.

2) The second case we study is when $(t_{n},x_{n})_{n\in \mathbb N}$ satisfies \eqref{bothequivalent} and 
$\lim_{n\rightarrow \infty }|t_{n}-x_{n}|=0$. 
This time, we choose $q<p$. As before, by boundedness of $T$ on $L^{q}(\mathbb R)$, we can bound \eqref{secterm} by
$$
\tilde{C}^{1+\delta' }C_{T,q}(C\epsilon^{-1})^{\frac{1}{\delta }}
|t_{n_{l}}-x_{n_{l}}|^{\frac{1}{p'}-\frac{1}{q'}}\frac{1}{|t_{n_{k}}-x_{n_{k}}|^{\frac{1}{q}-\frac{1}{p}}}
$$

Now, since $\frac{1}{p'}-\frac{1}{q'}\geq 0$, we fix $k>k_{\epsilon }$ and choose $l>k$ such that
$|t_{n_{l}}-x_{n_{l}}|^{\frac{1}{p'}-\frac{1}{q'}}<\epsilon (\tilde{C}^{1+\delta' }C_{T,q}(C\epsilon^{-1})^{\frac{1}{\delta }}|t_{n_{k}}-x_{n_{k}}|^{-(\frac{1}{q}-\frac{1}{p})})^{-1}$. This implies than \eqref{secterm} is 
again bounded by $\epsilon $.

3) Finally, we consider the case when $(t_{n},x_{n})_{n\in \mathbb N}$ satisfies  \eqref{bothequivalent} and
$\lim_{n\rightarrow \infty }|t_{n}+x_{n}|=\infty $ with
$d\leq |t_{n}-x_{n}|\leq 2d$ for some $d>0$.
These hypotheses imply 
that $\lim_{n\rightarrow \infty }|t_{n}|=
\infty $ and so,
$\lim_{l\rightarrow \infty }|t_{n_{l}}-x_{n_{k}}|=
\infty $ for any fixed $n_{k}$. 

Now, we fix $k>k_{\epsilon}$ and choose $l>k$ such that 
$3d< |t_{n_{l}}-x_{n_{k}}|$
and $\tilde{C}^{\delta' }C2d^{1+\delta }
|t_{n_{l}}-x_{n_{k}}|^{-1}<\epsilon$. 
The first condition implies that $3|t_{n_{l}}-t_{n_{l}}'|<3|t_{n_{l}}-x_{n_{l}}|/2<3d<|t_{n_{l}}-x_{n_{k}}|$. 

On the other hand, we note that all previous reasoning holds with the choice 
$\tilde{C}^{-1}(C^{-1}\epsilon)^{1/\delta }|t_{n_{l}}-t'_{n_{l}}|/2\leq \lambda_{n}\leq \tilde{C}^{-1}(C^{-1}\epsilon)^{1/\delta }|t_{n_{l}}-t'_{n_{l}}|$. 
Notice that the bounds are strictly smaller. This way, we obviously have 
$
\lambda_{n_{l}}\leq (C^{-1}\epsilon)^{1/\delta }|t_{n_{l}}-t'_{n_{l}}|
$
but also
\begin{align*}
\lambda_{n_{k}}&\leq \tilde{C}^{-1}(C^{-1}\epsilon)^{1/\delta }|t_{n_{k}}-t'_{n_{k}}|
\leq \tilde{C}^{-1}(C^{-1}\epsilon)^{1/\delta }|t_{n_{k}}-x_{n_{k}}|/2
\\
&\leq \tilde{C}^{-1}(C^{-1}\epsilon)^{1/\delta }d
\leq \tilde{C}^{-1}(C^{-1}\epsilon)^{1/\delta }|t_{n_{l}}-x_{n_{l}}|
\\
&\leq (C^{-1}\epsilon)^{1/\delta }|t_{n_{l}}-t'_{n_{l}}|
\end{align*}
applying \eqref{bothequivalent} in the last inequality. Then, 
we can apply \eqref{approximation'2} and so, bound \eqref{secterm} by a constant times
$$
\tilde{C}^{\delta' }|t_{n_{l}}-x_{n_{l}}|^{\frac{1}{p'}}|t_{n_{k}}-x_{n_{k}}|^{\frac{1}{p}}\frac{C|t_{n_{l}}-t_{n_{l}}'|^{\delta }}{|t_{n_{l}}-x_{n_{k}}|^{1+\delta}}
\leq \tilde{C}^{\delta' }2d\frac{Cd^{\delta}}{|t_{n_{l}}-x_{n_{k}}|^{1+\delta}}
\leq \epsilon
$$
by the choice of $l$. Notice that in the second last inequality we have used that 
$|t_{n_{l}}-t_{n_{l}}'|<|t_{n_{l}}-x_{n_{l}}|/2\leq d$.
\vskip5pt
We end this subsection proving that in the classical setting, the smoothness of the kernel and the weak boundedness condition
implies the off-diagonal decay of the kernel. We briefly remind the very well known definitions:

\begin{definition}A function $K:(\R^{2} \setminus \Delta )\to \mathbb C$ satisfies the smoothness condition of a 
standard Calder\'on-Zygmund kernel if 
for some $0<\delta \leq 1$ and $C>0$, we have
\begin{equation}\label{defsmoothCZ}
|K(t,x)-K(t',x')|
\le 
C\frac{(|t-t'|+|x-x'|)^\delta}{|t-x|^{1+\delta}}
\end{equation}
whenever $2(|t-t'|+|x-x'|)<|t-x|$. And $K$ satisfies the decay condition of a 
standard Calder\'on-Zygmund kernel if there is some $C>0$ such that 
\begin{equation}\label{defdecayCZ}
|K(t,x)|\le C {\displaystyle \frac{1}{|t-x|}} \
\end{equation} 

Finally, a linear operator $T : \S_{N}(\R) \to \S_{N}'(\R)$ satisfies the weak boundedness condition if 
for any interval $I$ and every pair $\phi_I, \varphi_I$ of
$L^2$-normalized bump functions adapted to $I$ with constant $C>0$ and order $N$, we have
$$
|\langle T(\phi_I),\varphi_I)\rangle |\lesssim C
$$
where the implicit constant only depends on the operator $T$.

\end{definition}

\begin{lemma}\label{smoothimpliesdecay2}
Let $K$ satisfying the smoothness condition (\ref{defsmoothCZ}) and such that the operator $T$ associated with $K$ is weakly bounded. Then, $K$ satisfies the decay condition (\ref{defdecayCZ}). 
\end{lemma}
\proof As in the previous result, let $\Phi$ be a positive smooth function such that it is 
supported and $L^{\infty}$-adapted to $[-1/2,1/2]$ and satisfies $\int \Phi (x)dx=1$.


For every $(t,x)\in \mathbb R^{2}\backslash \Delta $, let 
$f={\mathcal T}_{t}{\mathcal D}_{\lambda }^{1}\Phi$ and 
$g={\mathcal T}_{x}{\mathcal D}_{\lambda }^{1}\Phi$ with 
$\lambda =|t-x|/2$. We denote the interval $I=[\min(t,x),\max(t,x)]$. Since $\lambda^{1/2}f$ and $\lambda^{1/2}g$ are 
both $L^{2}$-adapted to $I$ with the same constant and order, by the weak boundedness condition we have
$$
|\langle T(f),g\rangle|\leq C\lambda^{-1}=C|t-x|^{-1}
$$

On the other hand, since $2(|u|+|y|)<2\lambda=|t-x|$
and $K$ satisfies the smoothness condition, we get
$$
|\langle T(f),g\rangle -K(t,x)|\leq 
\Big|\int \int {\mathcal D}_{\lambda}^{1}\Phi(u){\mathcal D}_{\lambda}^{1}\Phi(y)(K(u+t,y+x)-K(t,x))dudy\Big|
$$
$$
\leq C\int \int {\mathcal D}_{\lambda }^{1}\Phi(u){\mathcal D}_{\lambda}^{1}\Phi(y)
\frac{(|u|+|y|)^{\delta }}{|t-x|^{1+\delta}}dudy
\leq C\lambda^{\delta }|t-x|^{-(1+\delta )}\leq C|t-x|^{-1}
$$
Therefore, 
$$
|K(t,x)|\leq C|t-x|^{-1}
$$

\subsection{Weak compactness condition} We prove now that every linear operator compact on $L^p(\mathbb R)$ and bounded on $L^{p'}(\mathbb R)$, with $1 < p < \infty$ and $\frac{1}{p} + \frac{1}{p'}=1$, satisfies 
the weak compactness condition. 

\begin{proposition}\label{necessityofweakcompactness} Let $1<p<\infty $. 
Let $T$ be a bounded operator on $L^{p}(\mathbb R)$ and on $L^{p'}(\mathbb R)$. 
Then, 
for every $M\in \mathbb N$ and every $N\in \mathbb N$, $T$ satisfies that 
for all intervals $I\subset \mathbb R$ and all bump functions $\phi_I,\varphi_I$ $L^{2}$-adapted to $I$
with constant $C>0$ and order $N$, we have
\begin{align*}
|\langle T(\phi_I),\varphi_I\rangle |
&\leq  C_{p,T}C \Big(1+\frac{|I|}{2^{M}}\Big)^{-\alpha}\Big(1+\frac{2^{-M}}{|I|}\Big)^{-\alpha}
\Big(1+\frac{\rdist(I,\mathbb B_{2^{M}})}{M}\Big)^{-N}\\
&+C\| P_{M}^{\perp}\circ T\|
\end{align*}
with $\alpha=|1/2-1/p|+1/2$.
\end{proposition}

\begin{remark} We have that 
$C_{p,T} \lesssim \max(\| T\|_{p\rightarrow p}, \| T\|_{p'\rightarrow p'})C_{p}$, where 
$C_{p}$ is the constant of boundedness on $L^{p}(\mathbb R)$ of the lagom projection operator $P_{M}$
(which in turn depends on the boundedness constant on $L^{p}(\mathbb R)$ of the square function). 
\end{remark}

\begin{corollary}
Let $T$ be an operator compact on $L^{p}(\mathbb R)$ and bounded on $L^{p'}(\mathbb R)$. Then, $T$ satisfies the weak compactness condition for all $N>1$.
We note that, in this case, we can take 
$L(x)=(1+x)^{-\alpha}$, $S(x)=(1+x^{-1})^{-\alpha}$ and 
$D(x)=(1+x)^{-1}$.
\end{corollary}

\begin{proof}
We start with the estimate
$$
|\langle T(\phi_I),\varphi_I\rangle |
\leq |\langle (P_{M}\circ T)(\phi_I),\varphi_I\rangle |
+|\langle (T-P_{M}\circ T)(\phi_I),\varphi_I\rangle |
$$
whose second term can be quickly bounded by 
$$
\| T-P_{M}\circ T\| \| \phi_I\|_{p} \|\varphi_I\|_{p'}
\leq  C\| P_{M}^{\perp}\circ T\||I|^{1/p-1/2}|I|^{1/p'-1/2}= C \| P_{M}^{\perp}\circ T\|
$$

The control of the first term requires more work. If $I\in {\cal D}_{M+2}$, by boundedness of $T$ and $P_{M}$
on $L^{p}(\mathbb R)$, we can bound it by
$$
|\langle (P_{M}\circ T)(\phi_I),\varphi_I\rangle |
\leq  C_{p}\| T\|\| \phi_I\|_{p}\| \varphi_I\|_{p'}\leq C_{p}C\| T\|
$$
This estimate is compatible with the statement since $2^{-(M+2)}\leq |I|\leq 2^{(M+2)}$ and 
$\rdist(I,\mathbb B_{2^{M+2}})<M+2$.

If $I\notin {\cal D}_{M+2}$, we denote by $(\psi_{J})_{J\in {\cal D}}$ the wavelet basis defining $P_{M}$. 
Then, the first term can be bounded by
\begin{equation}\label{sumofdualpair}
|\langle T(\phi_I),P_{M}(\varphi_I)\rangle |
\leq \sum_{J\in {\cal D}_{M}} |\langle T(\phi_I),\psi_J\rangle | |\langle \varphi_I,\psi_{J}\rangle |
\end{equation}
Now, by boundedness of $T$ on $L^{p}(\mathbb R)$ and $L^{p'}(\mathbb R)$, we can bound the first factor inside the sum by 
\begin{align*}
|\langle T(\phi_I),&\psi_J\rangle | \leq \min(\| T\|_{p\rightarrow p} \| \phi_I\|_{p} \|\psi_J\|_{p'},
\| T\|_{p'\rightarrow p'} \| \phi_I\|_{p'} \|\psi_J\|_{p})
\\
&\leq C\max(\| T\|_{p,p},\| T\|_{p',p'} ) 
\min (|I|^{\frac{1}{p}-\frac{1}{2}}|J|^{\frac{1}{p'}-\frac{1}{2}}, |I|^{\frac{1}{p'}-\frac{1}{2}}|J|^{\frac{1}{p}-\frac{1}{2}}) 
\\
&\lesssim \min \Big(\big(\frac{|I|}{|J|}\big)^{\frac{1}{p}-\frac{1}{2}},\big(\frac{|J|}{|I|}\big)^{\frac{1}{p}-\frac{1}{2}}\Big) 
=\Big(\frac{\min(|I|,|J|)}{\max(|I|,|J|)}\Big)^{|\frac{1}{2}-\frac{1}{p}|}
\end{align*}
We denote $\theta_{p}=|1/2-1/p|$. 

To bound the second factor inside the sum at the right hand side of (\ref{sumofdualpair}) we consider three cases: 1)
$|I|>2^{M+2}$, 2) $|I|<2^{-(M+2)}$ and 3) $2^{-(M+2)}<|I|<2^{M+2}$ with 
$\rdist (I, \mathbb B_{2^{M+2}})>M+2$. 

{\bf 1)} Since $|I|>2^{M}\geq |J|$, by the mean zero of $\psi_{J}$ we know from Lemma \ref{decayofidentity} that
$$
|\langle \varphi_I,\psi_{J}\rangle |
\leq C \Big(\frac{|J|}{|I|}\Big)^{\frac{3}{2}}(1+|I|^{-1}|c(I)-c(J)|)^{-N}
$$
which leads to 
bound (\ref{sumofdualpair}) by
\begin{equation}\label{sumparametrized1}
|\langle T(\phi_I),P_M(\varphi_I)\rangle |
\lesssim \sum_{J\in {\mathcal D}_{M}} \Big(\frac{|J|}{|I|}\Big)^{\theta_p+\frac{3}{2}}(1+|I|^{-1}|c(I)-c(J)|)^{-N}
\end{equation}

Now, since $J\in {\mathcal D}_{M}$ implies $J\subset \mathbb B_{M2^{M}}$ with $2^{-M}\leq |J|\leq 2^{M}$, we can parametrize these intervals 
by their size:  
$|J|=2^{-k}2^{M}$ with $0\leq k\leq 2M$, in such a way that for every fixed $k$ there are 
$M2^{M}/|J|=M2^{k}$ intervals. Moreover, for all $J\in {\cal D}_{M}$, the value of $1+|I|^{-1}|c(I)-c(J)|$ can be bounded from below to obtain
$$
1+|I|^{-1}|c(I)-c(J)|\gtrsim 1+M^{-1}\rdist (I, \mathbb B_{2^{M}})
$$
as we show:
\begin{enumerate}
\item[a)] when 
$\rdist (I, \mathbb B_{M2^{M}})>2$, we have 
$\dist (I, \mathbb B_{M2^{M}})>M2^{M}$ and so, 
$
\dist (I, \mathbb B_{2^{M}})\leq  \dist (I, \mathbb B_{M2^{M}})+M2^{M}\leq 2\dist (I, \mathbb B_{M2^{M}})
$. 
This implies
\begin{align*}
\hskip 25pt |c(I)-c&(J)|\geq |I|/2+\dist(I,\mathbb B_{M2^{M}})
\\
&\geq |I|/4+2^{M}+\dist(I,\mathbb B_{2^{M}})/2
\\
&
\geq 1/4(|I|+2^{M}+\dist(I,\mathbb B_{2^{M}}))
\geq 1/4\diam(I,\mathbb B_{2^{M}})
\end{align*}
and so
\begin{align*}
1+|I|^{-1}|c(I)-c(J)|&\geq 1+1/4|I|^{-1}\diam(I,\mathbb B_{2^{M}})
\\
&\geq  1/4(1+\rdist(I,{\mathbb B}_{2^{M}}))
\end{align*}

\item[b)] when $\rdist (I, \mathbb B_{M2^{M}})\leq 2$, we have
\begin{align*}
\diam (I, &\mathbb B_{2^{M}})\leq \diam (I, \mathbb B_{M2^{M}})
\\
=&\rdist (I, \mathbb B_{M2^{M}})\max(|I|,M2^{M})
\leq 2M\max(|I|,2^{M}) 
\end{align*}
which implies
$\rdist (I, \mathbb B_{2^{M}})\leq 2M$. 
Therefore,
$$
1+|I|^{-1}|c(I)-c(J)|\geq  1\geq 1/3(1+M^{-1}\rdist (I, \mathbb B_{2^{M}}))
$$
\end{enumerate}

Finally, we also notice that for every fixed $k$ and for every 
$n\geq 4^{-1}(1+M^{-1}\rdist (I, \mathbb B_{2^{M}}))$ there are less than 
$\frac{|I|}{|J|}=\frac{|I|}{2^{-k}2^{M}}$ intervals $J\in {\mathcal D}_{M}$ such that 
$n\leq 1+|I|^{-1}|c(I)-c(J)|< n+1$. But, as said before, there also less the 
$2^{k}M$ of such intervals.

As a result, we can parametrize and bound the sum in (\ref{sumparametrized1}) in the following way
$$
|\langle T(\phi_I),P_M(\varphi_I)\rangle |
\lesssim \hspace{-.2cm}\sum_{0\leq k\leq 2M}\hspace{-.3cm}
 \sum_{\tiny \begin{array}{c}J\in {\mathcal D}_{M}\\ |J|=2^{-k}2^{M}\end{array}}
\hspace{-.3cm}\Big(\frac{2^{-k}2^{M}}{|I|}\Big)^{\theta_p+\frac{3}{2}}\Big(1+\frac{|c(I)-c(J)|}{|I|}\Big)^{-N}
$$
$$
\lesssim \Big(\frac{2^{M}}{|I|}\Big)^{\theta_p+\frac{3}{2}}\hspace{-.3cm}\sum_{0\leq k\leq 2M} 
2^{-k(\theta_p+\frac{3}{2})}\min (\frac{|I|}{2^{-k}2^{M}},2^{k}M)
\hspace{-.5cm}\sum_{n\geq \frac{1}{4}(1+M^{-1}\rdist (I, \mathbb B_{2^{M}}))}n^{-N}
$$
$$
\lesssim \Big(\frac{2^{M}}{|I|}\Big)^{\theta_p+\frac{3}{2}}\min (\frac{|I|}{2^{M}},M)
\Big(1+\frac{\rdist(I,\mathbb B_{2^{M}})}{M}\Big)^{-N}
\sum_{0\leq k\leq 2M} 2^{-k(\theta_p+\frac{1}{2})} 
$$
$$
\lesssim \Big(\frac{2^{M}}{|I|}\Big)^{\theta_p+\frac{3}{2}}\frac{|I|}{2^{M}}
\Big(1+\frac{\rdist(I,\mathbb B_{2^{M}})}{M}\Big)^{-N}
$$

%

Then, we finally obtain 
\begin{equation}\label{1}
|\langle T(\phi_I),P_M(\varphi_I)\rangle |
\lesssim \Big(\frac{2^{M}}{|I|}\Big)^{\theta_p+\frac{1}{2}}
\Big(1+\frac{\rdist(I,\mathbb B_{2^{M}})}{M}\Big)^{-N}
\end{equation}


{\bf 2)} In this case, since $|I|<2^{-M}\leq |J|$, we have by Lemma \ref{decayofidentity} 
$$
|\langle \varphi_I,\psi_{J}\rangle |
\leq C \Big(\frac{|I|}{|J|}\Big)^{\frac{1}{2}}\Big(1+\frac{|c(I)-c(J)|}{|J|}\Big)^{-N}
$$
which leads to bound (\ref{sumofdualpair}) by 
\begin{equation}\label{sumparametrized2}
|\langle T(\phi_I),P_M(\varphi_I)\rangle |
\lesssim \sum_{J\in {\mathcal D}_{M}} \Big(\frac{|I|}{|J|}\Big)^{\theta_p+\frac{1}{2}}\Big(1+\frac{|c(I)-c(J)|}{|J|}\Big)^{-N}
\end{equation}

As before, we parametrize the intervals $J\in {\mathcal D}_{M}$ by their size 
$|J|=2^{-k}2^{M}$ with $0\leq k\leq 2M$, all of them satisfying  
$J\subset \mathbb B_{M2^{M}}$. We work now to bound $1+|J|^{-1}|c(I)-c(J)|$ from above and below. 
Since $I\notin {\cal D}_{M+2}$, we have 
$$
\dist(I,\mathbb B_{M2^{M}})\leq |c(I)-c(J)|
\leq  |c(I)|+|c(J)|
$$
$$
\leq  |c(I)|+M2^{M}\leq 2(|I|/2+|c(I)|+M2^{M-1})
\leq 2\diam(I,\mathbb B_{M2^{M}})
$$
and so, for every $k$
$$
2^{k}
2^{-M}\dist(I,\mathbb B_{M2^{M}})
\leq |J|^{-1}|c(I)-c(J)|\leq 2^{k+1}2^{-M}\diam(I,\mathbb B_{M2^{M}})
$$

Moreover, since $|I|\leq |J|$, for fixed $k$, the $M2^{k}$ intervals $J\in \mathcal D_{M}$ can be parametrized in such a 
way that $1+|J|^{-1}|c(I)-c(J)|=j$ with all $j$  pairwise different and such that
\begin{equation}\label{index}
1+2^{k}2^{-M}\dist(I,\mathbb B_{M2^{M}})\leq j
\leq 1+2^{k}M\rdist(I,\mathbb B_{M2^{M}})
\end{equation}


We now separate into two subcases in a similar way as we did before but by using the distance instead of the relative distance: $\dist(I,\mathbb B_{M2^{M}})>2|\mathbb B_{M2^{M}}|$ and 
$\dist(I,\mathbb B_{M2^{M}})\leq 2|\mathbb B_{M2^{M}}|$. 
In the first case, we have 
$$
\diam(I,\mathbb B_{2^{M}})\leq \diam(I,\mathbb B_{M2^{M}})\leq |I|+\dist(I,\mathbb B_{M2^{M}})+M2^{M}
$$
$$
\leq 2M2^{M}+\dist(I,\mathbb B_{M2^{M}})\leq 2\dist(I,\mathbb B_{M2^{M}})
$$
Then, by (\ref{index}),
$$
j\geq 1+2^{k}2^{-M}\dist(I,\mathbb B_{M2^{M}})
\geq 1+2^{k-1}2^{-M}\diam(I,\mathbb B_{2^{M}})
$$
$$
\geq 2^{k-1}\rdist(I,\mathbb B_{2^{M}})
\geq 2^{k-2}(1+\rdist(I,\mathbb B_{2^{M}}))
$$ 
Therefore, we parametrize and bound the terms in the sum of (\ref{sumparametrized2})
that correspond to this case as follows:
$$
|\langle T(\phi_I),P_M(\varphi_I)\rangle |
\lesssim \sum_{0\leq k\leq 2M} \hspace{-.3cm}\sum_{\tiny \begin{array}{c}J\in {\mathcal D}_{M}\\ |J|=2^{-k}2^{M}\end{array}}
 \hspace{-.3cm} \Big(\frac{|I|}{2^{-k}2^{M}}\Big)^{\theta_p+\frac{1}{2}}\Big(1+\frac{|c(I)-c(J)|}{|J|}\Big)^{-N}
$$
$$
\leq \Big(\frac{|I|}{2^{M}}\Big)^{\theta_p+\frac{1}{2}}\sum_{0\leq k\leq 2M} 2^{k(\theta_p+\frac{1}{2})}
\sum_{j\geq 2^{k-2}(1+\rdist(I,\mathbb B_{2^{M}}))}j^{-N}
$$
$$
\lesssim \Big(\frac{|I|}{2^{M}}\Big)^{\theta_p+\frac{1}{2}}
\sum_{0\leq k\leq 2M} 2^{-k(N-(\theta_p+\frac{1}{2}))}(1+\rdist (I,\mathbb B_{2^{M}}))^{-N}
$$
$$
\lesssim \Big(\frac{|I|}{2^{M}}\Big)^{\theta_p+\frac{1}{2}}(1+\rdist (I,\mathbb B_{2^{M}}))^{-N}
$$
$$
= \Big(\frac{|I|}{2^{-M}}\Big)^{\theta_p+\frac{1}{2}}2^{-2M(\theta_p+\frac{1}{2})}(1+\rdist (I,\mathbb B_{2^{M}}))^{-N}
$$
which is better than the stated bound. 

On the other hand, when $\dist(I,\mathbb B_{M2^{M}})\leq 2|\mathbb B_{M2^{M}}|$, we have that 
$$
\diam(I,\mathbb B_{M2^{M}})\leq |I|+\dist(I,\mathbb B_{M2^{M}})+M2^{M}
$$
$$
\leq 2M2^{M}+\dist(I,\mathbb B_{M2^{M}})\leq 4M2^{M}
$$
and so 
$|\mathbb B_{M2^{M}}|\leq \diam(I,\mathbb B_{M2^{M}})\leq 4|\mathbb B_{M2^{M}}|$, that is, 
$1\leq \rdist(I,\mathbb B_{M2^{M}})\leq 4$.
Then, by (\ref{index}), we have 
$1\leq j\leq 1+2^{k+2}M$.
Moreover, 
$$
\rdist(I,\mathbb B_{2^{M}})=|\mathbb B_{2^{M}}|^{-1}\diam(I,\mathbb B_{2^{M}})
$$
$$
\leq M|\mathbb B_{M2^{M}}|^{-1}\diam(I,\mathbb B_{M2^{M}})
\leq 4M
$$
and so, $1+M^{-1}\rdist(I,\mathbb B_{2^{M}})\leq 5$.

Thereby, we can now bound the the terms in the sum of (\ref{sumparametrized2})
that correspond to this case as follows:
$$
|\langle T(\phi_I),P_M(\varphi_I)\rangle |
\lesssim \sum_{0\leq k\leq 2M}  \hspace{-.1cm}\sum_{\tiny \begin{array}{c}J\in {\mathcal D}_{M}\\ |J|=2^{-k}2^{M}\end{array}}
 \hspace{-.3cm}\Big(\frac{|I|}{2^{-k}2^{M}}\Big)^{\theta_p+\frac{1}{2}}\Big(1+\frac{|c(I)-c(J)|}{|J|}\Big)^{-N}
$$
$$
\leq \Big(\frac{|I|}{2^{M}}\Big)^{\theta_p+\frac{1}{2}}\sum_{0\leq k\leq 2M} 2^{k(\theta_p+\frac{1}{2})}
\sum_{1\leq j\leq 1+2^{k+2}M}j^{-N}
$$
$$
\lesssim \Big(\frac{|I|}{2^{M}}\Big)^{\theta_p+\frac{1}{2}}
\sum_{0\leq k\leq 2M} 2^{k(\theta_p+\frac{1}{2})}
\lesssim \Big(\frac{|I|}{2^{M}}\Big)^{\theta_p+\frac{1}{2}}2^{2M(\theta_p+\frac{1}{2})}
$$
$$
\lesssim  \Big(\frac{|I|}{2^{-M}}\Big)^{\theta_p+\frac{1}{2}}(1+M^{-1}\rdist (I,\mathbb B_{2^{M}}))^{-N}
$$

Therefore, in both subcases we obtain
\begin{equation}\label{2}
|\langle T(\phi_I),P_M(\varphi_I)\rangle |
\lesssim \Big(\frac{|I|}{2^{-M}}\Big)^{\theta_p+\frac{1}{2}}(1+M^{-1}\rdist (I,\mathbb B_{2^{M}}))^{-N}
\end{equation}

{\bf 3)} To deal with case $2^{-(M+2)}<|I|<2^{M+2}$ and $\rdist (I,\mathbb B_{2^{M+2}})>M+2$,  
 we will mix the techniques of the two previous cases: 
for those $J\in {\cal D}_{M}$ such that $|J|\leq |I|$ we will reason as we did to deduce (\ref{1}), while for those 
$J\in {\cal D}_{M}$ such that  $|J|\geq |I|$ we will reason in a similar way as we did to obtain (\ref{2}). 


In the subcase $|I|\geq |J|$, we first prove the bound 
\begin{equation}\label{lowbound}
1+|I|^{-1}|c(I)-c(J)|\gtrsim 1+ \rdist (I,\mathbb B_{2^{M}})
\end{equation}

We start by checking that $\rdist (I,\mathbb B_{2^{M+2}})>M+2$ implies that 
 $I\cap \mathbb B_{M2^{M}}=\emptyset$. 
Since 
$$
|I|+\dist(I,\mathbb B_{2^{M+2}})+2^{M+2}\geq \diam(I,\mathbb B_{2^{M+2}})
$$
$$
=\rdist (I,\mathbb B_{2^{M+2}})2^{M+2}>(M+2)2^{M+2}
$$ 
and $|I|\leq 2^{M+2}$, we have $\dist(I,\mathbb B_{2^{M+2}})>M2^{M+2}$. Therefore, 
we have $I\cap \mathbb B_{M2^{M+2}}=\emptyset$. 

This, together with $J\subset \mathbb B_{M2^{M}}$, implies 
$$
|c(I)-c(J)|\geq \dist (I,\mathbb B_{M2^{M}})
\geq \dist (I,\mathbb B_{2^{M}})-\dist (\mathbb B_{2^{M}},\mathbb B_{M2^{M}}^{c})
$$
that is,
\begin{equation}\label{dife}
|c(I)-c(J)|\geq \dist (I,\mathbb B_{2^{M}})-2^{M}(M-1)
\end{equation}

Moreover, 
since $\dist (I,\mathbb B_{2^{M+2}})\leq \dist (I,\mathbb B_{2^{M}})$, we have
$$
(M+2)2^{M+2}\leq \diam (I,\mathbb B_{2^{M+2}})
$$
$$
\leq |I|+\dist (I,\mathbb B_{2^{M+2}})+2^{M+2}
\leq \dist (I,\mathbb B_{2^{M}})+2^{M+3}
$$
and so, 
\begin{equation}\label{M}
M2^{M}\leq 1/4\dist (I,\mathbb B_{2^{M}})
\end{equation}
With this and (\ref{dife}), we get
$
|c(I)-c(J)|\geq 3/4\dist (I,\mathbb B_{2^{M}})
$.

Furthermore, since (\ref{M}) implies $|I|\leq 2^{M+2}\leq M^{-1}\dist (I,\mathbb B_{2^{M}})$, 
we also have 
$$
\diam (I,\mathbb B_{2^{M}})\leq \dist (I,\mathbb B_{2^{M}})+|I|+2^{M}
$$
$$
\leq (1+2M^{-1})\dist (I,\mathbb B_{2^{M}})
\leq 3\dist (I,\mathbb B_{2^{M}})
$$
which implies
$
|c(I)-c(J)|\geq 1/4\diam(I,\mathbb B_{2^{M}})
$.
Finally then,
$$
1+|I|^{-1}|c(I)-c(J)|\geq 1/4( 1+|I|^{-1}\diam (I,\mathbb B_{2^{M}}))
$$
$$
\geq 1/4(1+ \max(|I|,2^{M})^{-1}\diam (I,\mathbb B_{2^{M}}))
=1/4(1+\rdist (I,\mathbb B_{2^{M}}))
$$
as we wanted. 

Now, similar as before, we notice that for every fixed $k$ and for every 
$n\geq 4^{-1}(1+ \rdist(I, \mathbb B_{2^{M}}))$ there are at most 
$|I|/|J|=\frac{|I|}{2^{-k}2^{M}}$ intervals $J\in {\mathcal D}_{M}$ such that 
$n\leq 1+|I|^{-1}|c(I)-c(J)|< n+1$. With this, 
the part under consideration of the sum in (\ref{sumparametrized1}) can be 
parametrized and estimated in the 
following way:
$$
\sum_{0\leq k\leq 2M} \sum_{\tiny \begin{array}{c}J\in {\cal D}_{M}\\ |J|=2^{-k}2^{M}\leq |I|\end{array}}
\Big(\frac{2^{-k}2^{M}}{|I|}\Big)^{\theta_{p}+\frac{3}{2}}(1+|I|^{-1}|c(I)-c(J)|)^{-N}
$$
$$
\lesssim \Big(\frac{2^{M}}{|I|}\Big)^{\theta_{p}+\frac{3}{2}}
\sum_{\max(0,M+\log|I|^{-1})\leq k\leq 2M} 2^{-k(\theta_{p}+\frac{3}{2})} 
\frac{|I|}{2^{-k}2^{M}}\sum_{n\geq \frac{1}{4}(1+\rdist(I, \mathbb B_{2^{M}}))}\hspace{-.5cm}n^{-N}
$$
$$
\lesssim \Big(\frac{2^{M}}{|I|}\Big)^{\theta_{p}+\frac{1}{2}}
\sum_{\max(0,M+\log|I|^{-1})\leq k} 2^{-k(\theta_{p}+\frac{1}{2})}(1+\rdist(I, \mathbb B_{2^{M}}))^{-N}
$$
$$
\lesssim \Big(\frac{2^{M}}{|I|}\Big)^{\theta_{p}+\frac{1}{2}}
(1+\rdist(I, \mathbb B_{2^{M}}))^{-N}
\min(\Big(\frac{|I|}{2^{M}}\Big)^{\theta_{p}+\frac{1}{2}},1)
$$
\begin{equation}\label{thirdcase0}
\lesssim (1+\rdist(I, \mathbb B_{2^{M}}))^{-N}
\end{equation}

For the second subcase, $|I|\leq |J|\leq 2^{M}$, we parametrize and bound the relevant subsum of (\ref{sumparametrized2}) as
\begin{equation}\label{twocasestogether}
\Big(\frac{|I|}{2^{M}}\Big)^{\theta_{p}+\frac{1}{2}}
\sum_{0\leq k\leq 2M}\hspace{-.5cm}
\sum_{\tiny \begin{array}{c}J\in {\cal D}_{M}\\ |J|=2^{-k}2^{M}\geq |I|\end{array}}
\hspace{-.3cm}2^{k(\theta_{p}+\frac{1}{2})}\Big(1+\frac{|c(I)-c(J)|}{|J|}\Big)^{-N}
\end{equation}
We now decompose the sums in the same way we did to obtain \eqref{2}:
$$
\sum_{0\leq k\leq \min (2M,M+\log|I|^{-1})}\hspace{-.5cm} 2^{k(\theta_{p}+\frac{1}{2})}
\Big(\sum_{j\geq 2^{k-2}(1+\rdist(I, \mathbb B_{2^{M}}))}j^{-N}
+\sum_{1\leq j\leq 1+2^{k+2}M}j^{-N}\Big)
$$
$$
\lesssim \hspace{-.3cm}\sum_{0\leq k\leq M+\log|I|^{-1}}\hspace{-.4cm} 2^{k(\theta_{p}+\frac{1}{2})}
\Big(2^{-Nk}(1+\rdist (I,\mathbb B_{2^{M}}))^{-N}+\Big(1+\frac{\rdist (I,\mathbb B_{2^{M}})}{M}\Big)^{-N}\Big)
$$
$$
\lesssim 
\Big(\hspace{-.5cm}\sum_{0\leq k\leq M+\log|I|^{-1}}\hspace{-.5cm} 2^{-k(N-(\theta_{p}+\frac{1}{2}))}\\
+\hspace{-.5cm}
\sum_{0\leq k\leq M+\log|I|^{-1}}2^{k(\theta_{p}+\frac{1}{2})}\Big)\Big(1+\frac{\rdist (I,\mathbb B_{2^{M}})}{M}\Big)^{-N}
$$
$$
\lesssim 
\Big(1+\Big(\frac{2^{M}}{|I|}\Big)^{\theta_{p}+\frac{1}{2}}\Big)
\Big(1+\frac{\rdist (I,\mathbb B_{2^{M}})}{M}\Big)^{-N}
$$
\begin{equation}\label{thirdcase}
\lesssim 
\Big(\frac{2^{M}}{|I|}\Big)^{\theta_{p}+\frac{1}{2}}
\Big(1+\frac{\rdist (I,\mathbb B_{2^{M}})}{M}\Big)^{-N}
\end{equation}
since $|I|\leq 2^{M}$.
Notice here that the splitting of the sum into two parts is merely a notational device to express the division into two further sub-subcases along the lines of the subcases considered in the case 2).  In particular, note that in the considerations from which the sum $\sum_{1 \leq j \leq 1 + 2^{k+2}M} j^{-N}$ arises, we showed that $\rdist (I,\mathbb B_{2^{M}})/M \lesssim 1$ and therefore, 
$$
\sum_{1 \leq j \leq 1 + 2^{k+2}M} j^{-N}\lesssim 1\lesssim \Big(1+\frac{\rdist (I,\mathbb B_{2^{M}})}{M}\Big)^{-N}
$$

Then, with (\ref{thirdcase}), we can bound (\ref{twocasestogether}) by 
$$
\Big(\frac{|I|}{2^{M}}\Big)^{\theta_{p}+\frac{1}{2}}
\Big(\frac{2^{M}}{|I|}\Big)^{\theta_{p}+\frac{1}{2}}
\Big(1+\frac{\rdist (I,\mathbb B_{2^{M}})}{M}\Big)^{-N}
=\Big(1+\frac{\rdist (I,\mathbb B_{2^{M}})}{M}\Big)^{-N}
$$

Putting \eqref{thirdcase0} and \eqref{thirdcase} together, we get in the third case that
\begin{equation}\label{3}
|\langle T(\phi_I),P_M(\varphi_I)\rangle | \lesssim \Big(1+\frac{\rdist (I,\mathbb B_{2^{M}})}{M}\Big)^{-N}
\end{equation}


Finally, with (\ref{1}), (\ref{2}), (\ref{3}) and using $\min (a,b)\approx 1/(a^{-1}+b^{-1})$, we have 
\begin{equation}\label{4}
|\langle T(\phi_I),P_M(\varphi_I)\rangle |
\end{equation}
$$
\lesssim
\min(1,\Big(\frac{2^{M}}{|I|}\Big))^{\theta_p+\frac{1}{2}}
\min(1,\Big(\frac{|I|}{2^{-M}}\Big))^{\theta_p+\frac{1}{2}}
\Big(1+\frac{\rdist (I,\mathbb B_{2^{M}})}{M}\Big)^{-N}
$$
$$
\lesssim
\Big(1+\frac{|I|}{2^{M}}\Big)^{-(\theta_p+\frac{1}{2})}
\Big(1+\frac{2^{-M}}{|I|}\Big)^{-(\theta_p+\frac{1}{2})}
\Big(1+\frac{\rdist (I,\mathbb B_{2^{M}})}{M}\Big)^{-N}
$$
finishing the proof.

\end{proof}

\subsection{Interlude on compact Calder\'on-Zygmund kernels.}\label{kernelandadmissible} 


Before proving the necessity of the third hypothesis, we devote some time to justify the assumption of some extra properties of the compact Calder\'on-Zygmund kernels and the admissible functions appearing in their definition. 

We will start by assuming that the admissible functions satisfy that 
$L$ and $D$ are monotone non-increasing while $S$ is monotone non-decreasing. 
This is possible because if $L$, $S$ and $D$ are admissible and as in Definition \ref{prodCZ}, then the functions 
$$
L_{1}(x)=\sup_{y\in [x,\infty )}L(y)
\hskip20pt
S_{1}(x)=\sup_{y\in [0,x]}S(y)
\hskip20pt
D_{1}(x)=\sup_{y\in [x,\infty )}D(y)
$$
are admissible, give the same kernel bounds as in Definition \ref{prodCZ} and satisfy the monotonicity requirements. 

More importantly, in all forthcoming proofs, 
we will not be using the smoothness hypothesis of the kernel as it is stated in Definition \ref{prodCZ}. Instead, 
we will adopt the alternative formulation provided by the following lemma:

\begin{lemma}\label{compactCZ}
Let $\Delta $ be the diagonal of $\mathbb R^{2}$.
Let $K:(\R^{2} \setminus \Delta )\to \mathbb C$ be a
compact Calder\'on-Zygmund kernel. Then, 
there exists $0 < \delta' < 1$ and admissible functions $L_{1},S_{1},D_{1}$ such that
\begin{multline}\label{compCZ}
|K(t,x)-K(t',x')|
\le C 
\frac{(|t-t'|+|x-x'|)^{\delta'}}{|t-x|^{1+\delta' }}\\
L_{1}(|t-x|)S_{1}(|t-t'|+|x-x'|)
D_{1}\Big(1+\frac{|t+x|}{1+|t-x|}\Big)
\end{multline}
whenever $2(|t-t'|+|x-x'|)<|t-x|$. 
\end{lemma}
\proof
We need to show that 
the condition in Definition \ref{prodCZ} implies this new condition (\ref{compCZ}), with any $\delta' < \delta$. 
To see this, pick $0 < \varepsilon < \delta $, let $\delta' = \delta - \varepsilon$, and 
\begin{equation*}
F(x,t,x',t') = \frac{(|t-t'|+|x-x'|)^{\varepsilon}}{|t-x|^\varepsilon}L(|t-x|)S(|t-x|)
D(|x+t|),
\end{equation*}
defined for tuples $(x,t,x',t')$ such that $2(|t-t'|+|x-x'|) < |t-x|$. 
Then, by Definition \ref{prodCZ}, 
$$
|K(t,x)-K(t',x')|
\le C 
\frac{(|t-t'|+|x-x'|)^{\delta'}}{|t-x|^{1+\delta' }}
F(x,t,x',t')
$$
with $F$ bounded and such that $\lim_{|t-x|\to \infty} F = 0$. 

We prove now that also $\lim_{|t-t'|+|x-x'|\to 0} F = 0$. 
Suppose there exists a sequence $(x_n,t_n,x_n',t_n')$ such that 
$2(|x_n-x_n'|+|t_n-t_n'|) < |x_n-t_n|$ and $\lim_{n\to\infty} |x_n-x_n'|+|t_n-t_n'| = 0$, 
but $\inf_n F(x_n,t_n,x_n',t_n') > 0$. 

If there would exist a constant $C > 0$ such that
\begin{equation*}
|x_n-x_n'|+|t_n-t_n'| \geq C|x_n-t_n|
\end{equation*}
then, by monotonicity of $S$, we would have 
$$
S(|x_{n}-t_{n}|)\leq S(C^{-1}(|x_{n}-x_{n'}|+|t_{n}-t_{n'}|))
$$
The upper bound tends to zero and so, also $F(x_{n},t_{n},x_{n'},t_{n})$ would tend to zero, giving a contradiction.

Hence, we can assume
\begin{equation*}
\varliminf_n \frac{|x_n-x_n'|+|t_n-t_n'|}{|x_n-t_n|} = 0,
\end{equation*}
which immediately gives $\varliminf_{n} F(x_n,t_n,x_n',t_n') = 0$, also a contradiction.

%

It is also clear that $\lim_{\frac{|x+t|}{1+|t-x|}\to \infty} F = 0$, in view of the inequality $\frac{|x+t|}{1+|t-x|}\leq |x+t|$ and the limit satisfied by $D$.


Finally, we define for $y \geq 0$,
\begin{align*}
L_{1}(y) &= \sup_{|t-x| \geq y} F(x,t,x',t')^{1/3}, \\
S_{1}(y) &= \sup_{|t-t'|+|x-x'| \leq y} F(x,t,x',t')^{1/3}, \\ 
D_{1}(y) &= \sup_{1+\frac{|x+t|}{1+|t-x|} \geq y} F(x,t,x',t')^{1/3},
\end{align*}
This way, $L_{1}$,  $S_{1}$ and $D_{1}$ constitute a set of admissible functions
and the corresponding smoothness condition holds since 
$$
L_{1}(|t-x|)S_{1}(|t-t'|+|x-x'|)D_{1}\Big(1+\frac{|x+t|}{1+|t-x|}\Big)\geq F(x,t,x',t')
$$
finishing the proof.

\vskip5pt
From now on we will assume that a compact Calder\'on-Zygmund kernel is associated with a $\delta'$ (which we will denote again by $\delta$) and admissible, monotone functions $L_1$, $S_1$ and $D_1$ (denoted $L, S$ and $D$) satisfying the inequality (\ref{compCZ}).
One reason for this change of notation is that, 
since 
$
|x|\leq (1+|t-x|+|x+t|)/2
$,
we have 
$$
1+\frac{|x|}{1+|t-x|}\leq \frac{3}{2}\Big(1+\frac{|x+t|}{1+|t-x|}\Big)
$$
and by the monotonicity of $D_{1}$,   
$$
D_{1}\Big(\displaystyle{1+\frac{|x+t|}{1+|t-x|}}\Big)
\leq D_{1}\Big(\displaystyle{\frac{2}{3}\Big(1+\frac{|x|}{1+|t-x|}\Big)}\Big).
$$
This fact will be crucially used in our proofs. 

Finally, we note that,  
under the condition $2(|t-t'|+|x-x'|)<|t-x|$,
$$
1+\frac{\min(|x+t|,|x'+t'|)}{1+|t-x|}
\approx 1+\frac{|x+t|}{1+|t-x|}.
$$



%


\subsection{Necessity of $T(1) \in \CMO$.}\label{T(1)} 
Now, in order to prove necessity of the third condition, we need to give a rigorous definition of $T(1)$. Since $T$ is a associated with a standard Calder\'on-Zygmund kernel, the actual definition of $T(1)$ is already known. However, in our proofs we need stronger estimates than the ones provided by the classical theory. 
This is the reason why we state and proof Lemma \ref{definecmo}, which will allow us to define $T(1)$ 
as a functional in the dual of the space of smooth functions with compact support and mean zero. 


Once this is done, 
the hypothesis that $T(1)\in \BMO(\mathbb R)$ means that the inequality
$$
|\langle T(1), f\rangle |\leq C\| f\|_{H^1(\mathbb R)}
$$
holds for a dense subspace of $H^{1}(\mathbb R)$. In particular, we will verify the estimate for all smooth functions $f$ with compact support and mean zero, which are dense in $H^{1}(\mathbb R)$. 
Furthermore, the hypothesis 
$T(1)\in \CMO(\mathbb R)$ means that
$$
 \lim_{M\rightarrow \infty} |\langle P_{M}^{\perp}(T(1)),f\rangle |=0
$$ 
holds uniformly in a dense subset of the unit ball of $H^{1}(\mathbb R)$. The necessity of $T(1) \in \CMO$ when $T$ is a compact operator will be proven in Proposition \ref{necessity2}.



For all $a\in \mathbb R$ and $\lambda >0$, $p>0$,   
we define the translation operator as $ {\cal T}_{a} (f)(x)=f(x-a)$ and the dilation operators as
${\cal D}_{\lambda }(f)(x)=f(x/\lambda)$ and ${\cal D}_{\lambda }^{p}(f)(x)=\lambda^{-1/p}f(x/\lambda)$. 
Let also $\Phi\in {\cal S}(\R)$ be a smooth cut-off function such that 
$\Phi(x)=1$ for $|x|\leq 1$, $0<\Phi(x)<1$ for $1<|x|<2$ 
and $\Phi(x)=0$ for $|x|>2$. 

In order to give meaning to $T(1)$ (and also to $T^{*}(1)$), we
use the following technical lemma:

\begin{lemma}\label{definecmo}
Let $T$ be a linear operator associated with a compact Calder\'on-Zygmund kernel $K$ with parameter $\delta >0$ (as in Section \ref{kernelandadmissible}).
Let $I\subset \mathbb R$ be an interval and let $f\in \S(\R)$ have compact support in $I$ and mean zero.
Then, the limit
$${\mathcal L}(f)=\lim_{k\to \infty} \langle T({\cal T}_{a} {\mathcal D}_{2^{k}|I|}\Phi ),f\rangle $$
exists 
and it is independent of the translation parameter $a\in \mathbb R$ and the cut-off function $\Phi$.

Moreover, for all $k\in \mathbb N$ such that 
$2^{k}\geq 1+|I|^{-1}|a-c(I)|$, we have the error bound
\begin{multline*}
|{\mathcal L}(f)- \langle T({\cal T}_{a} {\cal D}_{2^{k}|I|}\Phi ),f\rangle |
\le C2^{-\delta k} (1+|I|^{-1}|a-c(I)|)^{\delta}\\
\sum_{k'=0}^{\infty }2^{-\delta k'}
L_K(2^{k'+k}|I|)S_K(2^{k'+k}|I|)D_K(1+(1+2^{k'+k}|I|)^{-1}|a|)
\| f\|_{L^1(\mathbb R)}
\end{multline*}
where the constant $C$ depends only on $\Phi$ and $T$.


\end{lemma}

\begin{remark}
Notice that when $|a|\geq |c(I)|$ we have the bound
$$
|{\mathcal L}(f)- \langle T({\cal T}_{a} {\cal D}_{2^{k}|I|}\Phi ),f\rangle |
\leq C 2^{-k\delta }
\sum_{k'=0}^{\infty }2^{-\delta k'}F_{K}(2^{k'+k}I)
\| f\|_{L^1(\mathbb R)}
$$

This is due to the fact that when $|a|\geq |c(I)|$ we have 
$$
1+(1+2^{j}|I|)^{-1}|a|
\geq 1+(1+2^{j}|I|)^{-1}|c(I)|
$$
$$
\geq (1+2^{j}|I|)^{-1}(1+2^{j}|I|+|c(I)|)
\geq 1/2\max(2^{j}|I|,1)^{-1}\diam(2^{j}I\cup \mathbb B)
$$
$$
=1/2\rdist(2^{j}I,\mathbb B)
$$
and, since $D$ is non-increasing
$$
L(2^{j}|I|)S(2^{j}|I|)D(1+(1+2^{j}|I|)^{-1}|a|)
$$
$$
\lesssim L(2^{j}|I|)S(2^{j}|I|)D(\rdist(2^{j}I,\mathbb B))
=F_{K}(2^{j}I)
$$
\end{remark}

Notice that, by Lebesgue's Dominated Convergence Theorem, the function 
$$
\tilde{F}(2^{k}I)=\sum_{k'=0}^{\infty }2^{-\delta k'}F_{K}(2^{k'+k}I)
$$
has the same limit behavior as $F$, that is, 
$$
\lim_{\tiny \begin{array}{c}M\rightarrow \infty \\ 2^{k}I\in {\mathcal D}_{M}^{c}\end{array}}\tilde{F}(2^{k}I)=0
$$
and we can furthermore assume analogous monotonicity properties. 
Then, with some abuse of notation to maintain its symmetry, we will denote $\tilde{F}(2^{k}I)$ again by $F(2^{k}I)$. 

With this convention the inequality we will frequently use in later proofs is written as
\begin{equation}\label{error}
|{\mathcal L}(f)- \langle T({\cal T}_{a} {\cal D}_{2^{k}|I|}\Phi ),f\rangle |
\leq C 2^{-k\delta }
F_{K}(2^{k}I)
\| f\|_{L^1(\mathbb R)}
\end{equation}


\proof

In the proof, we drop the subindex $K$ in the notation of $F$ and the admissible functions. 

For $k\in \mathbb N$ with $2^{k}\geq 1+|I|^{-1}|a-c(I)|$
, we introduce the smooth cut-off
$\Psi_k={\cal T}_{a}D_{2^{k+2}|I|}\Phi-{\cal T}_{a}D_{2^{k+1}|I|}\Phi$. We aim to estimate
$|\langle T(\Psi_{k}), f\rangle |$.

In view of the supports of $\Psi_k$ and $f$, we may restrict ourselves to $t$ and $x$ satisfying $2^{k+1}|I|<|t-a|<2^{k+3}|I|$ and $|x-c(I)|\leq |I|/2$. Then,
\begin{align*}
|x-a|&\leq |x-c(I)|+|c(I)-a|\leq |I|/2+|c(I)-a|
\\
&\leq |I|(1+|I|^{-1}|c(I)-a|)<|I|2^{k}<2^{-1}|t-a|
\end{align*}

Therefore, the supports of $\Psi_k$ and $f$
are disjoint and so, we can use the kernel
representation and write
\begin{align*}
\langle T(\Psi_k), f\rangle 
=\int \Psi_{k}(t)f(x)(K(t,x)-K(t,a))\, dtdx
\end{align*}
due to the zero mean of $f$.
Now, since $2|x-a|<|t-a|$, we can use the definition of a compact Calder\'on-Zygmund kernel to obtain
\begin{align*}
|\langle T(\Psi_k), &f\rangle |\leq \int |\Psi_{k}(t)||f(x)| \frac{|x-a|^{\delta }}{|t-a|^{1+\delta}}
\\
&L(|t-a|)S(|x-a|)
D(1+(1+|t-a|)^{-1}|a|)
\, dtdx
\end{align*}
Notice that we are using the conventions of Section \ref{kernelandadmissible}.

Since $|x-c(I)|<|I|/2$, $2^{k+1}|I|<|t-a|<2^{k+3}|I|$ and 
$|x-a|\leq |I|+|a-c(I)|$, we have
by the monotonicity properties of $L$, $S$ and $D$:
\begin{align*}
|\langle T(\Psi_k), f\rangle |
&\lesssim L(2^{k+3}|I|)S(2^{k+1}|I|)D(1+(1+2^{k+3}|I|)^{-1}|a|)
\\
\int&_{|x-c(I)|<\frac{|I|}{2}}\int_{|t-a|<2^{k+3}|I|} |f(x)| 
\frac{|I|^{\delta}(1+|I|^{-1}|a-c(I)|)^{\delta }}{2^{k(1+\delta)}|I|^{(1+\delta )}}
\, dtdx
\\
&\lesssim 2^{-k\delta }(1+|I|^{-1}|a-c(I)|)^{\delta }
\\
L(&2^{k}|I|)S(2^{k}|I|)D(1+(1+2^{k}|I|)^{-1}|a|)
\| f\|_{L^1(\mathbb R)}
\end{align*}
where we used that $1+(1+2^{k+3}|I|)^{-1}|a|>2^{-3}(1+(1+2^{k}|I|)^{-1}|a|)$ and Remark \ref{constants} 
to hide all constants in the argument of the admissible functions. 

This way, the trivial bound 
$$
|\langle T(\Psi_k), f\rangle |\lesssim 2^{-k\delta }(1+|I|^{-1}|a-c(I)|)^{\delta }\| f\|_{L^1(\mathbb R)}
$$ 
proves that the sequence $(\langle T({\cal T}_{a}D_{2^{k}|I|}\Phi ),f\rangle )_{k>0}$ is Cauchy and thus, the existence of the limit, which we will momentarily denote by ${\mathcal L}_{a}(f)$.

Finally, the explicit rate of convergence stated in
the lemma follows by summing a geometric series. For 
every $k\in \mathbb N$,
\begin{align*}
|{\mathcal L}_{a}(f)&-\langle T({\cal T}_{a}{\mathcal D}_{2^{k}|I|}\Phi ),f\rangle |
\\
&\leq \lim_{m\rightarrow \infty }|{\mathcal L}_{a}(f)-\langle T({\cal T}_{a}{\mathcal D}_{2^{m}|I|}\Phi ),f\rangle |
+\sum_{k'=k}^{\infty }|\langle T(\Psi_{k'} ),f\rangle |
\\
&\lesssim \sum_{k'=k}^{\infty }2^{-k'\delta }
L(2^{k'}|I|)S(2^{k'}|I|)D(1+(1+2^{k'}|I|)^{-1}|a|)
\| f\|_{L^1(\mathbb R)}
\end{align*}
which is the stated bound.

\vskip10pt
We show now that ${\mathcal L}_{a}$ is independent of the translation parameter. 
Let $a,b\in \mathbb R$
and 
$\psi_{k,a,b}={\cal T}_{a}D_{2^{k}|I|}\Phi-{\cal T}_{b}D_{2^{k}|I|}\Phi $. Then, for $k$ large enough so that 
$|a-b|<2^{k}|I|$, the support of 
$\psi_{k,a,b}$ is contained in $(2^{k-1}+1)|I|<|t-(a+b)/2|\leq (2^{k+1}+1/2)|I|$. Whence, we can repeat previous reasoning to show that also 
$|\langle T(\psi_{k,a,b}), f\rangle |\lesssim
2^{-k\delta }\| f\|_{L^1(\mathbb R)}$.

This way,  
\begin{align*}
|{\mathcal L}_{a}(f)-{\mathcal L}_{b}(f)|
&\leq |{\mathcal L}_{a}(f)-\langle T({\cal T}_{a}{\mathcal D}_{2^{k}|I|}\Phi ),f\rangle |
\\
&+|\langle T({\cal T}_{a}{\mathcal D}_{2^{k}|I|}\Phi ),f\rangle - \langle T({\cal T}_{b}{\mathcal D}_{2^{k}|I|}\Phi ),f\rangle |
\\
&+|{\mathcal L}_{b}(f)-\langle T({\cal T}_{b}{\mathcal D}_{2^{k}|I|}\Phi ),f\rangle |
\lesssim 3\cdot 2^{-k\delta }\| f\|_{L^1(\mathbb R)}
\end{align*}

Moreover, a similar argument works to prove that ${\mathcal L}$ is also independent of the chosen cut-off function $\Phi $.
This is because if we take two such smooth cut-off functions $\Phi$ and $\tilde{\Phi}$, then 
${\cal T}_{a}D_{2^{k}|I|}\Phi-{\cal T}_{a}D_{2^{k}|I|}\tilde{\Phi} $ has support included in 
$2^{k}|I|<|t-a|\leq 2^{k+1}|I|$.

%
%


\vskip10pt
We state the following technical lemma whose proof can be found in 
\cite{TLec} and \cite{OV}, Lemma 3.2.

\begin{lemma}\label{lowoscillation}
Let $I$ be some interval and f be an integrable function supported
in $I$ with mean zero. For each dyadic interval $J$ let $\phi_{J}$ be a bump function adapted
to $J$ with constant $C>0$ and order $N$. 

Then, for all dyadic intervals $J$ such that 
$|I|\leq |J|$, we have
$$
|\langle f, \phi_{J} \rangle |
\leq C\| f\|_{1}\frac{|I|}{|J|^{3/2}}\Big( 1+\frac{|c(I)-c(J)|}{|J|}\Big)^{-(N-1)}
$$
\end{lemma}

We can prove now the necessity of the third condition. 
\begin{proposition}\label{necessity2}
Let $T$ be a linear operator associated with a standard Calder\'on-Zygmund kernel.
If $T$ can be extended to a compact operator on $L^{p}(\mathbb R)$ for $1<p<\infty $ then, $T(1),T^{*}(1)\in \CMO(\mathbb R)$.
\end{proposition}
\proof
We will only show the result for $T(1)$. The argument for $T^*(1)$ is analogous.

We start by proving membership in $\BMO(\mathbb R)$. Although the argument is classical, we include it here to clarify the calculations at the end of the proof. We show that $\mathcal L$, defined in  Lemma \ref{definecmo}, is a bounded linear functional on $H^{1}(\mathbb R)$. Since linearity is trivial, we prove 
its continuity on $H^{1}(\mathbb R)$. 
By standard arguments, it is enough to prove the result for $p'$-atoms. Moreover, due to density, we can also assume the atoms to be smooth. 
Then, let $f$ be a smooth atom in 
$H^{1}(\mathbb R)$ supported in an
interval $I$ with $\| f\|_{L^{p'}(\mathbb R)}\lesssim |I|^{-1/p}$ and mean zero.

For all  $k\in \mathbb N$, let $\Psi_{k}={\cal T}_{c(I)}{\mathcal D}_{2^{k+1}|I|}\Phi-{\cal T}_{c(I)}{\mathcal D}_{2^{k}|I|}\Phi $ be as in the proof of Lemma \ref{definecmo}. This way, for any $k\in \mathbb N$
we have
\begin{align*}
|{\mathcal L}(f)|&\leq |\langle T({\cal T}_{c(I)}{\mathcal D}_{|I|}\Phi ),f\rangle |
+\sum_{k'=0}^{k-1}|\langle T(\Psi_{k'} ),f\rangle |
\\
&+|{\mathcal L}(f)-\langle T({\cal T}_{c(I)}{\mathcal D}_{2^{k}|I|}\Phi ),f\rangle |
\end{align*}

By using boundedness of $T$ on $L^{p}(\mathbb R)$
we can bound the first term by
$$
\|T\| \| {\mathcal T}_{c(I)}D_{|I|}\Phi \|_{L^{p}(\mathbb R)}
\| f\|_{L^{p'}(\mathbb R)}
\lesssim |I|^{1/p}\| f\|_{L^{p'}(\mathbb R)}
\lesssim 1
$$
From the proof of Lemma \ref{definecmo}, rather than from the result itself, 
the second term can be bounded by a constant times
$$
\sum_{k'=0}^{k-1}2^{-k'\delta }F_{K}(2^{k'}I)\| f\|_{L^{1}(\mathbb R)}
\lesssim \sum_{k'=0}^{k-1}2^{-k'\delta }|I|^{1/p}\| f\|_{L^{p'}(\mathbb R)}
\lesssim 1
$$
We note that since $T$ is compact, the kernel is actually a compact Calder\'on-Zygmund kernel
wit parameter $\delta$. 
Applying the result of Lemma \ref{definecmo}, we bound the last term by a constant times
$$
2^{-k\delta }F_{K}(2^{k}I)\| f\|_{L^{1}(\mathbb R)}
\lesssim 1
$$
 
These three estimates
show
the bound
$
|{\mathcal L}(f)|\lesssim 1
$ 
for every atom $f$. This proves, as claimed, that 
${\mathcal L}$ 
defines
a bounded linear functional on 
$H^1(\mathbb R)$. Hence, 
by the $H^1$-$\rm BMO$ duality, the functional ${\mathcal L}$ is represented
by a $\rm BMO(\mathbb R)$ function denoted by  $T(1)$, that is, 
${\mathcal L}(f)=\langle T(1),f\rangle $.


\vskip10pt
In order to prove membership in $\CMO(\mathbb R)$, we need to show that 
there is a lagom projection operator $P_{M}$ such that 
$$
\lim_{M\rightarrow \infty}\langle P_{M}^{\perp}(T(1)),f\rangle =0
$$
uniformly 
for all $f\in {\cal S}(\mathbb R)$ in the ball of $H^{1}(\mathbb R)$ with mean zero and support in a dyadic interval $I$. 
For this, we consider the projection operator $P_{M}$ constructed with a wavelet basis 
$(\psi_{J})_{J\in {\mathcal D}}$ of $H^{1}(\mathbb R)$ such that every function $\psi_{J}$ is smooth, $L^{2}$-adapted and compactly supported in $J$.  Since 
$$
P_{M}(f)
=\sum_{J\in \mathcal D_{M}}\langle f,\psi_{J}\rangle \psi_{J}
=\sum_{J\in \mathcal D_{M}}|J|^{1/2}\langle f,\psi_{J}\rangle |J|^{-1/2}\psi_{J}
$$
is defined by a finite linear combination of $1$-atoms $|J|^{-1/2}\psi_{J}$ , we have that $P_{M}(f)\in H^{1}(\mathbb R)$. Moreover, from the equality 
$f=P_{M}(f)+P_{M}^{\perp}(f)$, we also deduce that $P_{M}^{\perp}(f)\in H^{1}(\mathbb R)$.

%

Now, we can prove the desired convergence. For every $0<\epsilon $, we take $k\in \mathbb N$ so that 
$2^{-k\delta }<\epsilon $. 
Then, due to compactness of $T$, 
we can take  
$M>0$ (depending on $T$, $\delta $, $p$, $\epsilon $ and $I$) so that 
$\| P_{M}^{\perp}\circ T\|_{p\rightarrow p} 2^{k/p}<\epsilon $ and $2^{M}>|I|>2^{-M}$, $I\subset {\mathbb B}_{2^{M}}$. 

We first note that, by the previous discussion of ${\cal L}$ and the fact that $P_M^\perp$ is self-adjoint in the $L^2$-pairing, we have
$$
\langle P_{M}^{\perp}(T(1)),f\rangle 
=\langle T(1),P_{M}^{\perp}(f)\rangle ={\mathcal L}(P_{M}^{\perp}(f))
$$
The second and last expressions
are meaningful because we already know that $T(1)\in \BMO(\mathbb R)$ and $P_{M}^{\perp}(f)\in H^{1}(\mathbb R)$. 
However, $P_{M}^{\perp}(f)$ is not supported on $I$ but on the larger set 
$I \cup (\cup_{J\in {\mathcal D}_{M}:J\cap I\neq \emptyset} J)$. 
Therefore, 
we need to do some extra work. 

By the comments after Definition \ref{lagom} of a lagom projection operator, we can write
$$
P_{M}^{\perp}(f)=\sum_{I\in {\cal D}_{M}^{c}}\langle f,\psi_{I}\rangle \psi_{I} 
$$
in a Schauder basis sense. Then, by linearity and continuity of ${\cal L}$ on $H^{1}(\mathbb R)$, we have
$$
{\mathcal L}(P_{M}^{\perp}(f))
=\sum_{J\in  {\mathcal D}_{M}^{c}}\langle f,\psi_{J}\rangle {\cal L}(\psi_{J})
$$

Since $f$ and $\psi_{J}$ have compact support, the non-null terms in the sum arise from those intervals $J$ satisfying 
$J\cap I\neq \emptyset$.
Therefore, 
\begin{align}\label{ortoT1}
\nonumber
{\mathcal L}(P_{M}^{\perp}(f))
&=\sum_{\tiny \begin{array}{l}J\in  {\mathcal D}_{M}^{c}\\ J\subset I\end{array}}\langle f,\psi_{J}\rangle {\cal L}(\psi_{J})
+\sum_{\tiny \begin{array}{l}J\in  {\mathcal D}_{M}^{c}\\ I\subset J\end{array}}\langle f,\psi_{J}\rangle {\cal L}(\psi_{J})
\\
&={\mathcal L}(f_{I})
+\sum_{\tiny \begin{array}{l}J\in  {\mathcal D}_{M}^{c}\\ I\subset J\end{array}}\langle f,\psi_{J}\rangle {\cal L}(\psi_{J})
\end{align}
where $f_{I}={\mathcal P}_{I}(P_{M}^{\perp}(f))$ and ${\mathcal P}_{I}$ denotes the classical projection operator. We note that $f_{I}\in H^{1}(\mathbb R)$ with support on $I$ and, due to boundedness of the projections, 
$\| f_{I}\|_{H^{1}}\lesssim \|f\|_{H^{1}}\leq 1$ and $\| f_{I}\|_{L^{p'}}\lesssim \|f\|_{L^{p'}}\lesssim |I|^{-1/p}$. Then,  we proceed to bound the first term as we did before:
$$
{\cal L}(f_{I})=\langle (P_{M}^{\perp}\circ T)({\mathcal T}_{c(I)}{\mathcal D}_{2^{k}|I|}\Phi),f_{I}\rangle 
+({\cal L}(f_{I})-\langle T({\mathcal T}_{c(I)}{\mathcal D}_{2^{k}|I|}\Phi),P_{M}^{\perp}(f_{I})\rangle )
$$
Since $P_{M}^{\perp}(f_{I})=f_{I}$, the second term in the sum can be dealt with by Lemma \ref{definecmo}
$$
|{\mathcal L}(f_{I})-\langle T({\mathcal T}_{c(I)}{\mathcal D}_{2^{k}|I|}\Phi),f_{I}\rangle |
\lesssim 2^{-k\delta }\| f_{I}\|_{L^1(\mathbb R)}
\lesssim 2^{-k\delta }<\epsilon 
$$
Meanwhile, the first term can be bounded by
$$
\| P_{M}^{\perp}\circ T\| \|{\mathcal T}_{c(I)}{\mathcal D}_{2^{k}|I|}\Phi \|_{L^p(\mathbb R)}
\| f_{I}\|_{L^{p'}(\mathbb R)}
\lesssim \| P_{M}^{\perp}\circ T\|  2^{\frac{k}{p}}|I|^{\frac{1}{p}}|I|^{-\frac{1}{p}}
<\epsilon 
$$

Now, we bound the second term in (\ref{ortoT1}) by a reiteration of a previous argument. 
We first parametrize the intervals $J$ such that $I\subset J$  by their size, 
$|J_{j}|=2^{j}|I|$ with $j\geq 0$. Then, we write this second term as 
\begin{align*}
\sum_{j\geq 0}\langle f,\psi_{J_{j}}&\rangle {\mathcal L}(\psi_{J_{j}})
=\sum_{j\geq 0}\langle f,\psi_{J_{j}}\rangle 
\langle (P_{M}^{\perp}\circ T)({\mathcal T}_{c(I)}{\mathcal D}_{2^{k+j}|I|}\Phi),\psi_{J_{j}}\rangle 
\\
&+\sum_{j\geq 0}\langle f,\psi_{J_{j}}\rangle 
({\mathcal L}(\psi_{J_{j}})-\langle T({\mathcal T}_{c(I)}{\mathcal D}_{2^{k+j}|I|}\Phi),P_{M}^{\perp}(\psi_{J_{j}})\rangle )
\end{align*}
Since $J_{j}\in {\mathcal D}_{M}^{c}$, we have 
$\psi_{J_{j}}=P_{M}^{\perp}(\psi_{J_{j}})$  and so, 
by Lemma \ref{definecmo}, we can bound the modulus  
of the second term in previous expression 
by 
\begin{align*}
\sum_{j\geq 0} |\langle f,&\psi_{J_{j}}\rangle |
|{\mathcal L}(\psi_{J_{j}})-\langle T({\mathcal T}_{c(I)}{\mathcal D}_{2^{k+j}|I|}\Phi),\psi_{J_{j}}\rangle |
\\
&\lesssim \sum_{j\geq 0} \| f\|_{L^1(\mathbb R)} \|\psi_{J_{j}}\|_{L^{\infty }(\mathbb R)}
2^{-(k+j)\delta }
\| \psi_{J_{j}}\|_{L^1(\mathbb R)}
\\
&\leq 2^{-k\delta }\| f\|_{L^1(\mathbb R)} 
\sum_{j\geq 0}|J_{j}|^{-1/2}
2^{-j\delta }|J_{j}|^{1/2}
\lesssim 2^{-k\delta }<\epsilon 
\end{align*}


Meanwhile, we bound the modulus of the first term by 
\begin{align*}
&\sum_{J\in {\mathcal D}_{M}^{c}}|\langle f,\psi_{J_{j}}\rangle | 
\| P_{M}^{\perp}\circ T\| \|{\mathcal T}_{c(I)}{\mathcal D}_{2^{k+j}|I|}\Phi \|_{L^p(\mathbb R)}\| \psi_{J_{j}}\|_{L^{p'}(\mathbb R)}
\\
&\lesssim \| P_{M}^{\perp}\circ T\| \sum_{j\geq 0}|\langle f,\psi_{J_{j}}\rangle | 2^{\frac{k}{p}}
|2^{j}I|^{\frac{1}{p}}|J_{j}|^{-\frac{1}{2}+\frac{1}{p'}}
\leq \epsilon \sum_{j\geq 0}|\langle f,\psi_{J_{j}}\rangle | 
|J_{j}|^{\frac{1}{2}} 
\end{align*}
Now, we use Lemma \ref{lowoscillation} to bound last expression by a constant times
\begin{align*}
\epsilon \sum_{j\geq 0}\| f\|_{1}
\frac{|I|}{|J_{j}|^{\frac{3}{2}}}|J_{j}|^{\frac{1}{2}} 
= \epsilon \| f\|_{1}\sum_{j\geq 0}\frac{|I|}{2^{j}|I|} 
\lesssim \epsilon
\end{align*}

\section{The operator acting on bump functions}\label{bump}

In this section we study the action of $T$ over bump functions. In particular, 
we seek good estimates of the dual pair $\langle T(\psi_{I}),\psi_{J}\rangle $ in terms of the space and
frequency localization of the bump functions $\psi_{I},\psi_{J}$.

Before starting, we will state and prove five lemmata about localization properties of bump functions. These results will be frequently used in
Proposition \ref{symmetricspecialcancellation}, the main result of this section. 

In these results, we use a smooth cut-off function $\Phi\in {\cal S}(\R)$ such that $\Phi(x)=1$ for $|x|\leq 1$, 
$0<\Phi(x)<1$ for $1<|x|<2$,   
and $\Phi(x)=0$ for $|x|>2$. Then, 
$\Phi_{I}={\mathcal T}_{c(I)}{\mathcal D}_{|I|}\Phi$ denotes an $L^\infty$-normalized function adapted to $I$ such that 
$\Phi_{I}=1$ in $2I$ and $\Phi_{I}=0$ in $(4I)^{c}$. We also recall that for all $\lambda >0$, $\lambda J$ denotes 
the interval with the same center as $J$ and length $\lambda |J|$.
We also introduce the notation $w_{I}(x)=1+|I|^{-1}|x-c(I)|$.

\begin{lemma}\label{bump5}
Let $\phi$ be a bump function adapted to $I$ with constant $C$ and order $N$. Let 
$\phi_{in}=\psi \cdot \Phi_{\lambda J}$ with $|J|=|I|$ and $\lambda >0$. Then, 
$\phi_{in}$ is adapted to $I$ with constant $C(1+\lambda^{-1})^{N}$ and order $N$. 
\end{lemma}
\proof For all $0\leq n\leq N$, 
$$
|\psi_{in}^{(n)}(x)|
\leq \sum_{k=0}^{n}\Big(\begin{array}{c}n\\k\end{array}\Big)
|\psi^{(k)}(x)||\Phi_{\lambda J}^{(n-k)}(x)|
$$
$$
\leq \sum_{k=0}^{n}\Big(\begin{array}{c}n\\k\end{array}\Big)
C|I|^{-(\frac{1}{2}+k)}w_{I}(x)^{-N}|\lambda I|^{-(n-k)}w_{\lambda J}(x)^{-N}
$$
$$
\leq C(1+\lambda^{-1})^{N}
|I|^{-(\frac{1}{2}+n)}w_{I}(x)^{-N}
$$
since $w_{\lambda J}(x)\geq 1$.

\begin{lemma}\label{bump1}
Let $I,J$ be such that $|J|\leq |I|$. Let $N\in \mathbb N$, $0<\theta <1$, 
$R\geq 1$ and
$\lambda \geq R^{-1}w_{J}(c(I))^{\theta }$.

Let $\phi_{J}$ be an $L^2$-normalized bump function adapted to $J$ with constant $C$ and order $N$.
Then, $\phi_{out}=\phi_{J}(1-\Phi_{\lambda J})$ is an $L^2$-normalized bump function adapted to $I$ with
order $[N/4]$ and constant comparable to
$$
CR^{\frac{N}{4}}\Big(\frac{|J|}{|I|}\Big)^{-\frac{1}{2}}\lambda^{-\frac{N}{2}}
$$
\end{lemma}

\begin{remark}\label{bump1remark} We apply the lemma with 
$\lambda =1/32(|J|^{-1}\diam (I\cup J))^{\theta }$, for which the result says that 
$\phi_{out}=\phi_{J}(1-\Phi_{\lambda J})$ is adapted to $I$ with 
order $[N/4]$ and constant comparable to
$$
C\Big(\frac{|J|}{|I|}\Big)^{\theta \frac{N}{2}-\frac{1}{2}}\rdist(I,J)^{-\theta \frac{N}{2}}
$$
We note that, when the remark is applied, the constant will depend on $N$ which, in turn, will be fixed.
\end{remark}
\proof
Due to the support of $1-\Phi_{\lambda J}$ we can assume $|x-c(J)|\geq\lambda |J|$ and so, 
$w_{J}(x)\geq \lambda $.
Then, since $\theta \leq 1$, we have
$$
w_{I}(x)
\leq |J|^{-1}|x-c(J)|+1+|J|^{-1}|c(I)-c(J)|
$$
$$
\leq R|J|^{-1}|x-c(J)|+(R\lambda)^{1/\theta}
\leq 2\cdot R^{1/\theta }w_{J}(x)^{1/\theta }
$$


We also note that, since 
$\Phi_{I}$ is $L^{\infty }$ -adapted to $I$, we have
$
|\Phi_{I}^{(k)}(x)|\leq C |I|^{-k}
$.
With all this, we have for all $0\leq k\leq n\leq N/4$, 
$$
|\phi_{J}^{(k)}(x)(1-\Phi_{\lambda J})^{(n-k)}(x)|
\leq C|J|^{-(\frac{1}{2}+k)}w_{J}(x)^{-N}|\lambda J|^{-(n-k)}\chi_{4\lambda J\backslash 2\lambda J}(x)
$$
$$
\leq C\Big(\frac{|J|}{|I|}\Big)^{-(\frac{1}{2}+n)}
|I|^{-(\frac{1}{2}+n)}
\frac{1}{(1+\lambda)^{\frac{3N}{4}}}
w_{J}(x)^{-\frac{N}{4}}
\lambda^{-(n-k)}
$$
$$
\leq C\Big(\frac{|J|}{|I|}\Big)^{-(\frac{1}{2}+n)}\lambda^{-\frac{N}{2}}
|I|^{-(\frac{1}{2}+n)}
\frac{1}{(1+\lambda)^{\frac{N}{4}}}
2^{\theta \frac{N}{4}}R^{\frac{N}{4}}w_{I}(x)^{-\theta \frac{N}{4}}\lambda^{-(n-k)}
$$
Therefore, 
$$
|(\phi_{J}(1-\Phi_{\lambda J}))^{(n)}(x)|
\leq \sum_{k=0}^{n}\Big(\begin{array}{c}n\\k\end{array}\Big)|\phi_{J}^{(k)}(x)||(1-\Phi_{\lambda J})^{(n-k)}(x)|
$$
$$
\leq C2^{\theta \frac{N}{4}}R^{\frac{N}{4}}\lambda^{-\frac{N}{2}}\Big(\frac{|J|}{|I|}\Big)^{-(\frac{1}{2}+n)}
|I|^{-(\frac{1}{2}+n)}w_{I}(x)^{-\theta \frac{N}{4}}
$$
where we used that, since $\lambda \geq R^{-1}$ and $n\leq N/4$, we have
$$
\frac{1}{(1+\lambda)^{\frac{N}{4}}}
\sum_{k=0}^{n}\Big(\begin{array}{c}n\\k\end{array}\Big)\lambda^{-(n-k)}
=\frac{(\lambda+1)^{n}}{(1+\lambda)^{\frac{N}{4}}}
\lambda^{-n}
\leq \lambda^{-n}
\leq R^{\frac{N}{4}}
$$


\begin{lemma}\label{bump3}
Let $I,J$ be such that $|J|\leq |I|$. Let $N \in \mathbb N$, $0<\theta \leq 1$, 
$R\geq 3$ and 
$\lambda \geq R^{-1}(|J|^{-1}\diam (I\cup J))^{\theta }$.


Then, $|\lambda J|^{-1/2}\Phi_{\lambda J}$ is an $L^2$-normalized bump function adapted to $I$ with 
order $[\theta N/4]$ and constant 
$$
R^{2N}\Big(\frac{|J|}{|I|}\Big)^{\theta \frac{N}{4}-\frac{1}{2}}\lambda^{\frac{N-1}{2}}
$$
\end{lemma}
\proof This time we may restrict ourselves to $|x-c(J)|\leq 2\lambda |J|$. Then, 
$
|x-c(I)|\leq |x-c(J)|+|c(I)-c(J)|\leq 2\lambda |J|+ |c(I)-c(J)|
$
and so,  
$$
w_{I}(x)\leq \frac{|J|}{|I|}2\lambda +w_{I}(c(J))
$$
Moreover, 
$
w_{I}(c(J))\leq 2|I|^{-1}\diam(I\cup J)
\leq 2\frac{|J|}{|I|}(R\lambda)^{1/\theta}
$
and then,
$$
w_{I}(x)\leq \frac{|J|}{|I|}R\lambda +2\frac{|J|}{|I|}(R\lambda)^{1/\theta}
\leq  3\frac{|J|}{|I|}(R\lambda)^{\frac{1}{\theta}}
\leq R^{1+\frac{1}{\theta}}\frac{|J|}{|I|}\lambda^{\frac{1}{\theta}}
$$
because $R\lambda \geq 1$ and $0<\theta \leq 1$.
Therefore, since $|\Phi_{I}^{(n)}(x)|\lesssim |I|^{-n}$ 
, we have for all $0\leq n\leq \theta N/4\leq N/4$, 
$$
|\lambda J|^{-\frac{1}{2}}|\Phi_{\lambda J}^{(n)}(x)|
\lesssim \lambda^{-\frac{1}{2}}|J|^{-\frac{1}{2}} 
|\lambda J|^{-n}
$$
$$
\leq \lambda^{-(\frac{1}{2}+n)}|J|^{-(\frac{1}{2}+n)} 
R^{(\theta +1)\frac{N}{2}}
\Big(\frac{|J|}{|I|}\Big)^{\theta \frac{N}{2}}\lambda^{\frac{N}{2}}w_{I}(x)^{-\theta \frac{N}{2}}
$$
$$
\leq R^{N}\Big(\frac{|J|}{|I|}\Big)^{\theta \frac{N}{2}-\frac{1}{2}-n} \lambda^{\frac{N-1}{2}-n}
|I|^{-(\frac{1}{2}+n)}w_{I}(x)^{-\theta \frac{N}{2}}
$$
$$
\leq R^{2N}\Big(\frac{|J|}{|I|}\Big)^{\theta \frac{N}{4}-\frac{1}{2}} \lambda^{\frac{N-1}{2}}|I|^{-(\frac{1}{2}+n)}
w_{I}(x)^{-\theta \frac{N}{4}}
$$

\begin{lemma}\label{bump6}
Let $J$ be an interval and $k\in \mathbb N$. We define 
$\phi_{J}(t)=|J|^{-1/2}({\mathcal T}_{c(J)}{\mathcal D}_{|J|}\Phi)(t)(t-c(J))^{k}$. 
Then, $\phi_{J} $ is adapted to $J$
with constant comparable to $2^{3k}k!|J|^{k}$ and order $k$. 
\end{lemma}
\proof We define $h(t)=(t-c(J))^{k}$. Since 
${\mathcal T}_{c(J)}{\mathcal D}_{|J|}\Phi $
is $L^{\infty }$-adapted to $J$, 
we have
$|({\mathcal T}_{c(J)}{\mathcal D}_{|J|}\Phi)^{(j)}(t)|\lesssim |J|^{-j}$. Moreover, 
due to its support 
we have $|t-c(J)|\leq 2|J|$ and thus, 
$w_{J}(t)\leq 3$. 
Then, for all $0\leq n\leq k$
$$
|\phi_{J}^{(n)}(t)|\leq \sum_{j=0}^{n}\Big(\begin{array}{c}n\\j\end{array}\Big)
|J|^{-1/2}|({\mathcal T}_{c(J)}{\mathcal D}_{|J|}\Phi)^{(j)}(t)||h^{(n-j)}(t)|
$$
$$
\leq \sum_{j=0}^{n}\Big(\begin{array}{c}n\\j\end{array}\Big)
|J|^{-\frac{1}{2}}|J|^{-j}\frac{k!}{(k-(n-j))!}|t-c(J)|^{k-(n-j)}
$$
$$
\leq k!|J|^{k}|J|^{-(\frac{1}{2}+n)}\sum_{j=0}^{n}\Big(\begin{array}{c}n\\j\end{array}\Big)
= 2^{n}k!|J|^{k}|J|^{-(\frac{1}{2}+n)}
$$
$$
\leq 2^{k}k!|J|^{k}|J|^{-(\frac{1}{2}+n)}w_{J}(t)^{-k}3^{k}
$$

\begin{lemma}\label{bump4}
Let $I,J$ be two intervals such that $c(I)=c(J)$. Let $\phi_{J}$ be a bump function adapted to $J$ with 
order $N$ and constant $C>0$. 

Then,  $\phi_{J}$ is a bump function adapted to $I$ with order $N$ and constant 
$C\ec(I,J)^{-(N+\frac{1}{2})}$.
\end{lemma}
\proof When $|J|\leq |I|$, we have for all $0\leq n\leq N$, 
$$
|\phi_{J}^{(n)}(x)|
\leq C|J|^{-(\frac{1}{2}+n)}\Big(1+\frac{|x-c(J)|}{|J|}\Big)^{-N}
$$
$$
\leq C\Big(\frac{|I|}{|J|}\Big)^{\frac{1}{2}+n}|I|^{-(\frac{1}{2}+n)}\Big(1+\frac{|x-c(I)|}{|I|}\Big)^{-N}
$$

On the other hand, when $|I|\leq |J|$, we have for all $0\leq n\leq N$, 
$$
|\phi_{J}^{(n)}(x)|
\leq C|J|^{-(\frac{1}{2}+n)}
\Big(1+\frac{|x-c(J)|}{|J|}\Big)^{-N}
$$
$$
= C\Big(\frac{|I|}{|J|}\Big)^{\frac{1}{2}+n}|I|^{-(\frac{1}{2}+n)}\Big(\frac{|I|}{|J|}\Big)^{-N}
\Big(\frac{|J|}{|I|}+\frac{|x-c(I)|}{|I|}\Big)^{-N}
$$
$$
\leq C\Big(\frac{|J|}{|I|}\Big)^{N-(\frac{1}{2}+n)}|I|^{-(\frac{1}{2}+n)}\Big(1+\frac{|x-c(I)|}{|I|}\Big)^{-N}
$$

%
%

\vskip10pt
We turn now to the main result of this section. 
Given two intervals $I$ and $J$, we denote $K_{min}=J$ and $K_{max}=I$ if $|J|\leq |I|$, while $K_{min}=I$ and $K_{max}=J$ otherwise.

\begin{proposition}\label{symmetricspecialcancellation}
Let $K$ be a compact Calder\'on-Zygmund kernel with parameter $\delta$. We choose $N\in \mathbb N$ such that $N>\max(2^{26},(2\delta)^{-2})$. 

Let $T:{\mathcal S}_{N}\rightarrow {\mathcal S}_{N}'$ be a linear operator associated with $K$ satisfying the weak compactness condition with parameter $N$ 
and the special cancellation conditions $T(1)=0$ and $T^{*}(1)=0$. 

For every $\frac{8}{N}<\theta <\frac{1}{4N^{1/2}}$, 
we define $\delta'=\delta-\theta (1+\delta )$. 

Then,  
there exists $C_{\delta'}>0$  such that
for every $\epsilon >0$, for all intervals $I, J$ with $|I|\ge |J|$ 
and 
all bump functions $\phi_{I}$, $\psi_{J}$, $L^{2}$-adapted to $I$ and $J$ respectively
with order $N$ and constant $C>0$ and 
mean zero, 
\begin{equation*}\label{twobump2}
|\langle T(\psi_{I}),\psi_{J}\rangle |\leq  C_{\delta'} C^{2} \, 
\frac{\ec(I,J)^{\frac{1}{2}+\delta'}}{\rdist(I,J)^{1+\delta'}}
\Big(F(I_{1},\ldots ,I_{6};M_{T,\epsilon })+\epsilon \Big)
\end{equation*}
where
$I_{1}=I$, $I_{2}=J$, $I_{3}=\langle I, J\rangle $, $I_{4}=\lambda_{1}\tilde{K}_{max}$, 
$I_{5}=\lambda_{2}\tilde{K}_{max}$,
$I_{6}=\lambda_{2}K_{min}$,
$\lambda_{1}=|K_{max}|^{-1}{\rm diam}(I\cup J)$, 
$\lambda_{2}=(|K_{min}|^{-1}{\rm diam}(I\cup J))^{\theta }$ and 
$\tilde{K}_{max}$ is the translate of $K_{max}$ with the same centre as $K_{min}$.
\end{proposition}

\begin{remark}
The function $F$ is as given in the comments after Definition \ref{WB}, but possibly defined by admissible functions satisfying Definitions \ref{WB} and \ref{compactCZ} that are larger than those originally given.
For simplicity of notation we will not distinguish between the function $F$ given after Definition \ref{WB} and the one given by the above proposition.

Notice that the value $M_{T,\epsilon }$ is given by the weak compactness property as explained in the remarks after Definition \ref{WB}.
\end{remark}

\begin{proof}

By symmetry, we assume $|J|\leq |I|$ and so, we only use the mean zero of 
$\psi_{J}$ and the condition $T(1)=0$. 
We remind $\Phi_{K}$ is a smooth function $L^\infty$-adapted to an interval $K$
and $w_{K} (x)=1+|K|^{-1}|x-c(K)|$. 


We note that the parameter $\theta $ fixed
in the statement satisfies $0<\theta <\min(2^{-15},\delta /(1+\delta ))$ and so, 
$0<\delta'<\delta $. Moreover, $N$ satisfies $\frac{8}{\theta }<N<\frac{1}{16\theta^{2}}$
(since $\frac{1}{16\theta^{2}}-\frac{8}{\theta}>1$, there is an integer in between).

Let $\psi (t,x)=\phi_{I}(t)\psi_{J}(x)$ which, by hypothesis, is adapted to $I\times J$ with constant $C^{2}$ and 
order $N$ and it has mean zero in the variable $x$. 
Then, we decompose $\psi $ in the following way:
\begin{align*}
&\psi =\psi_{out}+\psi_{in}
\\
&\psi_{in}(t,x) = (\psi(t,x) -a(t))\Phi_{\frac{1}{32}\lambda_2J}(x)
\end{align*}
where $\lambda_2=(|J|^{-1}{\rm diam}(I\cup J))^{\theta}$ and 
$a(t)$ is chosen so that $\psi_{in}$, and also $\psi_{out}$, have mean zero in the $x$ variable. 
We split again
\begin{align*}
&\psi_{in}=\psi_{in,out}+\psi_{in,in}
\\
&\psi_{in,in}(t,x)= \psi_{in}(t,x) \Phi_{\lambda_1\tilde{I}}(t)
\end{align*}
with $\tilde{I}$ the translate of $I$ centered at $c(J)$ and 
$\lambda_1=|I|^{-1}{\rm diam}(I\cup J)$.

We first explain why $\psi_{in,in}$ is also adapted to $I\times J$
with constant comparable to $C^{2}$. We know that $\psi_{in}$ is the sum of two bump functions. 
By Lemma \ref{bump5} applied to the variable $x$ 
we deduce that the first function is adapted to $I\times J$ with constant $C^{2}2^{N}$. 
By inequalities 
(\ref{a(t)}) and (\ref{a(t)2}) below, we deduce that the second term is also 
adapted to $I\times J$ with constant $C^{2}$. This shows that $\psi_{in}$ is adapted to $I\times J$ with constant $C^{2}2^{N}$.
Moreover, applying Lemma \ref{bump5} in the variable $t$, we finally have that $\psi_{in,in}$ is adapted to $I\times J$
with constant $C^{2}2^{2N}$. Since $N$ is fixed, from now being we will not keep a track of the dependence of the constants with respect $N$.

Notice that all functions previously defined are sums of at most two functions of tensor product type and have mean zero in the variable $x$.
For tensor product of bump functions
we denote 
$
\Lambda \psi_{1}\otimes \psi_{2})=\langle T(\psi_{1}),\psi_{2}\rangle
$. 
Then, we plan to bound $\Lambda (\psi_{out})$ by the weak compactness condition and the decay of the bump function;
$\Lambda(\psi_{in,in})$ by the weak compactness condition plus the special cancellation condition $T(1)=0$;
and $\Lambda(\psi_{in,out})$ by the integral representation and the smoothness property of a compact 
Calder\'on-Zygmund kernel.

\vskip10pt
{\bf a)} We start with 
\begin{equation}\label{psi1}
\psi_{out}(t,x)= \psi(t,x)(1-\Phi_{\frac{1}{32}\lambda_2J}(x))
+a(t)\Phi_{\frac{1}{32}\lambda_2J}(x)
\end{equation}
and we work to prove that each term is a bump function $L^{2}$-adapted to $I\times I$ with a gain in the constant. 

Since $\psi $ is adapted to $I\times J$, by Remark \ref{bump1remark}, we have that the first term
$\psi(t,x)(1-\Phi_{\frac{1}{32}\lambda_2J}(x))$ is adapted to $I\times I$ with constant  
$$
C \left(\frac{|J|}{|I|}\right)^{\theta \frac{N}{4}-\frac{1}{2}}\rdist(I,J)^{-\theta \frac{N}{2}}
\leq C \left(\frac{|J|}{|I|}\right)^{\frac{3}{2}}\rdist(I,J)^{-4}
$$
because $\theta $ and $N$ were chosen so that $\theta N>8$.

On the other hand, since
$\psi_{out}$ has mean zero in the variable $x$, we have for the second term that
$$
a(t)  \int \Phi_{\frac{1}{32}\lambda_{2}J}(x)\, dx
=-\int  \psi(t,x) (1-\Phi_{\frac{1}{32}\lambda_{2}J}(x))\, dx
$$
and, by the definition of $\Phi $ and $\psi $, 
$$
|a(t)| \frac{2}{32}\lambda_{2}|J|\leq \int_{|x-c(J)|\geq \frac{1}{32}\lambda_{2}|J|} |\psi(t,x)|dx
$$
$$
\leq C^{2}\phi_{I}(t)|J|^{-\frac{1}{2}}\hspace{-.2cm}
\int_{|x-c(J)|\geq \frac{1}{32}\lambda_{2}|J|}\hspace{-.2cm} w_{J}(x)^{-N}dx
\lesssim C^{2}\phi_{I}(t)|J|^{-\frac{1}{2}}\lambda_{2}^{-(N-1)}|J|
$$
where $\phi_{I}=|I|^{-1/2}w_{I}^{-N}$ is $L^{2}$-adapted to $I$. 
This implies
\begin{equation}\label{a(t)}
|a(t)|\lesssim C62|J|^{-\frac{1}{2}}\lambda_{2}^{-N}|\phi_{I}(t)|
=C^{2}|\lambda_{2}J|^{-\frac{1}{2}}\lambda_{2}^{-N+\frac{1}{2}}|\phi_{I}(t)|
\end{equation}
and therefore,
$$
|a(t)\Phi_{\frac{1}{32}\lambda_{2}J}(x)|\lesssim
C^{2}\lambda_{2}^{-N+\frac{1}{2}}|\phi_{I}(t)||\lambda_{2}J|^{-\frac{1}{2}}\Phi_{\lambda_{2}J}(x)
$$
By Lemma  \ref{bump3}, $|\lambda_{2}J|^{-1/2}\Phi_{\lambda_{2}J}$ is a bump function adapted to $I$ with constant comparable to $(|J|/|I|)^{\theta N/4-1/2}\lambda_{2}^{(N-1)/2}$. With this,  
we prove that $a(t)\Phi_{\frac{1}{32}\lambda_{2}J}(x)$ decays with 
a gain of constant of 
\begin{align*}
\Big( \frac{|J|}{|I|}&\Big)^{\theta \frac{N}{4}-\frac{1}{2}}\lambda_{2}^{-\frac{N}{2}}
=\Big( \frac{|J|}{|I|}\Big)^{\theta \frac{N}{4}-\frac{1}{2}}
\Big(\frac{{\rm diam}(I\cup J)}{|J|}\Big)^{-\theta \frac{N}{2}}
\\
&=\Big( \frac{|J|}{|I|}\Big)^{3\theta \frac{N}{4}-\frac{1}{2}}\rdist(I,J)^{-\theta \frac{N}{2}}
\leq \Big( \frac{|J|}{|I|}\Big)^{\frac{11}{2}}\rdist(I,J)^{-4}
\end{align*}
That is, we obtain the estimate
$$
|a(t)\Phi_{\frac{1}{32}\lambda_{2}J}(x)|\lesssim C^{2} \left(\frac{|J|}{|I|}\right)^{\frac{11}{2}}\rdist(I,J)^{-4}|\phi_I(t)|
|I|^{-\frac{1}{2}}w_{I}(x)^{-N}
$$

In order to estimate higher derivatives, we proceed in a similar way. 
We denote by $\partial_{1}$ and $\partial_{2}$ the operators of partial differentiation with respect to the variables 
$t$ and $x$, respectively. 
First, we notice that since
$\partial_{1}^{k}\psi_{out}$ also has mean zero in the variable $x$, we have  
$$
a^{(k)}(t)\int \Phi_{\frac{1}{32}\lambda_{2}J}(x)\, dx
=-\int  \partial_{1}^{k}\psi(t,x) (1-\Phi_{\frac{1}{32}\lambda_{2}J}(x))\, dx
$$
Therefore, we have as in (\ref{a(t)}),
\begin{equation}\label{a(t)2}
|a^{(k)}(t)
|\leq C\lambda_{2}^{-N+\frac{1}{2}}|\phi_{I}^{(k)}(t)||\lambda_{2}J|^{-\frac{1}{2}}
\end{equation}
with $\phi_{I}$ $L^{2}$-adapted to $I$. 
By Lemma \ref{bump3} again, we have that 
$$
|\lambda_{2}J|^{-\frac{1}{2}}|\Phi_{\frac{1}{32}\lambda_{2} J}^{(j)}(t)|
\lesssim \Big(\frac{|J|}{|I|}\Big)^{\theta \frac{N}{4}-\frac{1}{2}}\lambda_{2}^{\frac{N-1}{2}} |I|^{-\frac{1}{2}-(n-k)}w_{I}(t)^{-N}
$$
This way, 
we have for all $k,j\leq [\theta N]$, 
\begin{align*}
|a^{(k)}(t)\Phi_{\frac{1}{32}\lambda_{2}J}^{(j)}(x)|
&\leq C^{2}\lambda_{2}^{-N+\frac{1}{2}}
|I|^{-(\frac{1}{2}+k)}w_{I} (t)^{-N}
\\
&\Big(\frac{|J|}{|I|}\Big)^{\theta \frac{N}{4}-\frac{1}{2}}\lambda_{2}^{\frac{N}{2}-\frac{1}{2}}
|I|^{-(\frac{1}{2}+j)}w_{I} (x)^{-N}
\end{align*}
and we bound the constant as we did before:
$$
C^{2}\Big(\frac{|J|}{|I|}\Big)^{\theta \frac{N}{4}-\frac{1}{2}}\lambda_{2}^{-\frac{N}{2}}
\leq C^{2}\Big( \frac{|J|}{|I|}\Big)^{\frac{11}{2}}\rdist(I,J)^{-4}
$$

This proves that the second function in (\ref{psi1})
is adapted to
$I\times I$  with an acceptable gain in the constant. 

All this work shows that $\psi_{out}$ is adapted to $I\times I$ with the stated gain of constant 
and so, by the weak compactness condition, for every $\epsilon >0$ there is $M_{T,\epsilon }$ such that
$$
|\Lambda(\psi_{out})|
\lesssim C^{2}\left(\frac{|J|}{|I|}\right)^{3/2}
\rdist(I,J)^{-4}(F_{W}(I;M_{T,\epsilon })+\epsilon )
$$
which is better decay than the one stated.

\vskip10pt
{\bf b)} To work with $\psi_{in,in}$, we first argue that we can assume that 
$\partial_{1}^{k}\psi_{in,in}(c(J),x)=0$ for any $x$ and $0\leq k\leq N$.

The assumption comes from the substitution of $\psi_{in,in}(t,x)$ by
\begin{equation}\label{subtraction1}
\psi_{in,in}(t,x)-({\mathcal T}_{c(J)}{\mathcal D}_{\frac{4}{32}\lambda_{2}|\tilde{I}|}\Phi)(t)\cdot \psi_{in,in}(c(J),x)
\end{equation}
\begin{equation}\label{subtraction2}
-\sum_{k=1}^{N} ({\mathcal T}_{c(J)}{\mathcal D}_{|J|}\Phi)(t)\frac{1}{k!}(t-c(J))^{k}\cdot \partial_{1}^{k}\psi_{in,in}(c(J),x)
\end{equation}
We need to prove that the subtracted terms satisfy the stated bounds. 

We denote $\tilde{\psi}(x)=\psi_{in,in}(c(J),x)$ and 
$\tilde{\psi}_{k}(x)=\partial_{1}^{k}\psi_{in,in}(c(J),x)$. Since $\psi_{in,in} $ is adapted to $I\times J$ with constant  $C^{2}$, we have
$$
|\tilde{\psi}(x)|\leq C^{2}|I|^{-\frac{1}{2}}w_{I}(c(J))^{-N}
\varphi_{J}(x)
$$
where $\varphi_{J}=|J|^{-1/2}w_{J}^{-N}$ is $L^{2}$ adapted to $J$. Moreover, since 
$$
w_{I}(c(J))
=|I|^{-1}(|I|+|c(J)-c(I)|)
\geq \frac{\diam(I\cup J)}{|I|}=\rdist(I,J)
$$
we get 
\begin{equation}\label{psitilde}
|\tilde{\psi}(x)| \leq C^{2}|I|^{-\frac{1}{2}}\rdist(I,J)^{-N}\varphi_{J}(x)
\end{equation}


On the other hand, we have that $\tilde{\psi}$ is supported in $\frac{4}{32}\lambda_{2}J$ with mean zero.
Then, by the special cancellation condition $T(1)=0$, the explicit error of Lemma \ref{definecmo}
and the decay of $\tilde{\psi}$ just calculated, 
we can estimate the contribution of the term subtracted in (\ref{subtraction1}) by
$$
|\langle T({\cal T}_{c(J)}{\mathcal D}_{\frac{4}{32}\lambda_{2}|\tilde{I}|}\Phi),\tilde{\psi}\rangle |
=|\langle T({\cal T}_{c(J)}{\mathcal D}_{\frac{|I|}{|J|}|\frac{4}{32}\lambda_{2}J|}\Phi ),\tilde{\psi}\rangle-\langle T(1),\tilde{\psi}\rangle |
$$
$$
\leq \left(\frac{|I|}{|J|}\right)^{-\delta}\hspace{-.3cm}F_{K}(\lambda_{2}\tilde{I})
\| \tilde{\psi}\|_{L^{1}(\mathbb R)}
\leq C^{2}\left(\frac{|J|}{|I|}\right)^{\frac{1}{2}+\delta }
\rdist(I,J)^{-N}F_{K}(\lambda_{2}\tilde{I})
$$
which is better than the stated bound. 

For the remaining subtracted terms, we intend to  
apply instead the weak compactness condition. As in \eqref{psitilde}, we have 
$$
|\tilde{\psi}_{k}(x)|
\leq C^{2}|I|^{-(\frac{1}{2}+k)}\rdist(I,J)^{-N}\varphi_{J}(x)
$$
with $\varphi_{J}$ an $L^{2}$-normalized bump function adapted to $J$ as before. 
Since $\psi_{in,in}$ is adapted to $I\times J$ with constant $C^{2}$, 
we have 
\begin{align*}
|\tilde{\psi}_{k}^{(j)}(x)|&=|\partial_{2}^{j}\partial_{1}^{k}\psi_{in,in}(c(J),x)|
\\
&\leq C^{2}|I|^{-(\frac{1}{2}+k)}\rdist(I,J)^{-N}|J|^{-(\frac{1}{2}+j)}w_{J} (x)^{-N}
\end{align*}
showing that $\tilde{\psi}_{k}$
is adapted to $J$ with constant 
$C^{2}|I|^{-(\frac{1}{2}+k)}\rdist(I,J)^{-N}$.

Moreover, by Lemma \ref{bump6} we have that 
$|J|^{-1/2}({\mathcal T}_{c(J)}{\mathcal D}_{|J|}\Phi)(t)(t-c(J))^{k}$ is adapted to $J$
with constant $2^{3k}k!|J|^{k}$ and order $k$. 
Then, by the weak compactness condition we have that for every $\epsilon >0$, 
$$
|\langle T(({\mathcal T}_{c(J)}{\mathcal D}_{|J|}\Phi)(\cdot -c(J))^{k}),
\tilde{\psi}_{k}\rangle |
$$
$$
\leq C^{2}2^{3k}k!\Big(\frac{|J|}{|I|}\Big)^{\frac{1}{2}+k}
\rdist(I,J)^{-N}(F_{W}(J;M_{T,\epsilon })+\epsilon )
$$
This way, 
$$
\Big|\sum_{k=1}^{N}\frac{1}{k!}\langle T(({\mathcal T}_{c(J)}{\mathcal D}_{|J|}\Phi)(\cdot -c(J))^{k}),
\tilde{\psi}_{k}\rangle \Big|
$$
$$
\lesssim C^{2}2^{3N}\left(\frac{|J|}{|I|}\right)^{\frac{3}{2}}\rdist(I,J)^{-N}(F_{W}(J;M_{T,\epsilon } )+\epsilon )
$$
which is, again, no larger than the stated bound.


This ends the justification of the assumption 
$\partial_{1}^{k}\psi_{in,in}(c(J),x)=0$ for any $x$ and any $0\leq k\leq N$.
Now, we further decompose $\psi_{in,in}$:
\begin{align*}
&\psi_{in,in}=\psi_{in,in, out}+\psi_{in,in,in}
\\
&\psi_{in,in,in}(t,x)= \psi_{in,in}(t,x)\Phi_{\frac{8}{32}\lambda_2J}(t)
\end{align*}

%

b1) We first prove that $\psi_{in,in,in}$ is adapted to $\lambda_{2}J\times \lambda_{2}J$ with constant
$$
C\left(\frac{|J|}{|I|}\right)^{3/4}\rdist(I,J)^{-7}
$$ 
and order $N_{2}=[N^{1/2}]$.
First, 
\begin{equation}\label{derivofpsiininin}
\partial_{1}^{n}\partial_{2}^{j}\psi_{in,in,in}(t,x)=
\sum_{k=0}^{n}\Big(\begin{array}{c}n\\k\end{array}\Big)
\Phi_{\frac{8}{32}\lambda_{2} J}^{(n-k)}(t)\partial_{1}^{k}\partial_{2}^{j}\psi_{in,in}(t,x)
\end{equation}

On the one side, 
we have
\begin{equation}\label{firstfactor}
|\Phi_{\frac{8}{32}\lambda_{2} J}^{(n-k)}(t)|\lesssim |\lambda_{2}J|^{-(n-k)}\Phi_{\lambda_{2} J}(t)
\lesssim \frac{|\lambda_{2}J|^{\frac{1}{2}}}{|\lambda_{2}J|^{\frac{1}{2}+(n-k)}}w_{\lambda_{2}J}(t)^{-N_{2}}
\end{equation}
which shows that $\Phi_{\frac{8}{32}\lambda_{2} J}$ is $L^{2}$-adapted to $\lambda_{2}J$ with constant  
$|\lambda_{2}J|^{1/2}$. 

On the other side, since the support of $\psi_{in,in,in}$ in the variable $t$  is in $\lambda_2J$, 
for all  $t\in \lambda_2J$ and all $x\in \mathbb R$, we have 
by the extra assumption, 
\begin{align}\label{derivofpsiinin}
\nonumber
|&\partial_{1}^{k}\partial_{2}^{j}\psi_{in,in}(t,x)|
=\Big| \int_{c(J)}^{t} \partial_1^{k+1} \partial_{2}^{j}\psi_{in,in}(r,x)\, dr\Big| 
\\
&\leq |t-c(J)|\| \partial_1^{k+1}\partial_{2}^{j}\psi_{in,in}(\cdot ,x)\|_{\infty }
\lesssim \lambda_{2}|J|\| \partial_1^{k+1}\partial_{2}^{j}\psi_{in,in}(\cdot ,x)\|_{\infty }
\end{align}
By the definition of a bump function we have
$$
| \partial_1^{k+1}\partial_{2}^{j}\psi_{in,in}(r,x)|
\leq C^{2}|I|^{-(\frac{3}{2}+k)}\Big(1+\frac{|r-c(I)|}{|I|}\Big)^{-N}|\varphi_{J}^{(j)} (x)|
$$
where $\varphi_{J}$ is a bump function adapted to $J$, because 
$\psi_{in,in} (r,\cdot )$ is 
adapted to $J$. 
Now, from the choice of $\Phi_{\frac{8}{32}\lambda_{2} J}$, we have for all $r\in \lambda_{2}J$,
$$
|r-c(J)|
\leq 1/2\lambda_{2}|J|
= 1/2|J|^{1-\theta}\diam(I\cup J)^{\theta}\leq 1/2\diam(I\cup J)
$$
since $|J|\leq \diam(I\cup J)$. Then,  
$$
1+|I|^{-1}|r-c(I)|\geq 1+|I|^{-1}|c(I)-c(J)|-|I|^{-1}|r-c(J)|
$$
$$
\geq |I|^{-1}\diam(I\cup J)-1/2|I|^{-1}\diam(I\cup J)
=1/2\rdist(I,J)
$$
Therefore, 
$$
\| \partial_1^{k+1}\partial_{2}^{j}\psi_{in,in}(\cdot ,x)\|_{\infty }\lesssim C|I|^{-(\frac{3}{2}+k)}\rdist(I,J)^{-N}
|\varphi_{J}^{(j)} (x)|
$$
Moreover, as in Lemma \ref{bump4}, we have that 
$$
|\varphi_{J}^{(j)} (x)|
\leq C\Big(\frac{|\lambda_{2}J|}{|J|}\Big)^{\frac{1}{2}+j}|\lambda_{2}J|^{-(\frac{1}{2}+j)}
\Big(1+\frac{|x-c(\lambda_{2}J)|}{|\lambda_{2}J|}\Big)^{-N_{2}}
$$

With the two previous 
inequalities,
we continue the estimate of  (\ref{derivofpsiinin}):
\begin{equation*}\label{derivofpsiinin2}
|\partial_{1}^{k}\partial_{2}^{j}\psi_{in,in}(t,x)|
\lesssim C^{2}\lambda_{2}\frac{|J|}{|I|^{\frac{3}{2}+k}}\rdist(I,J)^{-N}
\lambda_{2}^{\frac{1}{2}+j}|\lambda_{2}J|^{-(\frac{1}{2}+j)}w_{\lambda_{2}J} (x)^{-N_{2}}
\end{equation*}
\begin{equation}\label{secondfactor}
= C^{2}\left(\frac{|J|}{|I|}\right)^{\frac{3}{2}+k}\hspace{-.3cm}|\lambda_{2}J|^{-(\frac{1}{2}+k)}\rdist(I,J)^{-N}
\lambda_{2}^{2+k+j}|\lambda_{2}J|^{-(\frac{1}{2}+j)}w_{\lambda_{2}J} (x)^{-N_{2}}
\end{equation}

Therefore, with inequalities (\ref{firstfactor}) and (\ref{secondfactor}), we can bound (\ref{derivofpsiininin}), for all $0\leq n,j\leq N_{2}$  
$$
|\partial_{1}^{n}\partial_{2}^{j}\psi_{in,in,in}(t,x)|
\leq C^{2}\sum_{k=0}^{n}\Big(\begin{array}{c}n\\k\end{array}\Big)
|\lambda_{2}J|^{\frac{1}{2}}|\lambda_{2}J|^{-\frac{1}{2}-(n-k)}w_{\lambda_{2}J} (t)^{-N_{2}}
$$
$$
\left(\frac{|J|}{|I|}\right)^{\frac{3}{2}+k}|\lambda_{2}J|^{-(\frac{1}{2}+k)}\rdist(I,J)^{-N}
\lambda_{2}^{2+k+j}|\lambda_{2}J|^{-(1/2+j)}w_{\lambda_{2}J} (x)^{-N_{2}}
$$
$$
=C^{2}\left(\frac{|J|}{|I|}\right)^{\frac{3}{2}}\rdist(I,J)^{-N}\lambda_{2}^{2+n+j}
\sum_{k=0}^{n}\Big(\begin{array}{c}n\\k\end{array}\Big)
\left(\frac{|J|}{|I|}\right)^{k}\lambda_{2}^{-(n-k)}
$$
$$
|\lambda_{2}J|^{-(\frac{1}{2}+n)}w_{\lambda_{2}J} (t)^{-N_{2}}
|\lambda_{2}J|^{-(\frac{1}{2}+j)}w_{\lambda_{2}J} (x)^{-N_{2}}
$$
$$
\leq C^{2}\left(\frac{|J|}{|I|}\right)^{\frac{3}{2}}\rdist(I,J)^{-N}\lambda_{2}^{2+n+j}
\Big(\frac{|J|}{|I|}+\lambda_{2}^{-1}\Big)^{n}
\varphi^{n}_{\lambda_{2}J} (t)
\varphi^{j}_{\lambda_{2}J} (x)
$$
where we denoted $\varphi^{i}_{\lambda_{2}J}=|\lambda_{2}J|^{-(1/2+i)}w_{\lambda_{2}J}^{-N_{2}}$. 
Now, 
$$
\lambda_{2}^{2+n+j}\leq \lambda_{2}^{2+2N_{2}}
=\Big(\frac{\diam(I\cup J)}{|J|}\Big)^{2\theta (1+N_{2})}
\hspace{-.3cm}\leq \Big(\frac{|I|}{|J|}\Big)^{2\theta (1+N_{2})}\hspace{-.3cm}\rdist(I,J)^{2\theta (1+N_{2})}
$$
while 
$$
\Big(\frac{|J|}{|I|}+\lambda_{2}^{-1}\Big)^{n}
=\bigg(\frac{|J|}{|I|}+\Big(\frac{|J|}{|I|}\Big)^{\theta }\Big(\frac{\diam(I\cup J)}{|I|}\Big)^{-\theta }\bigg)^{n}
$$
$$
\leq \Big(\frac{|J|}{|I|}\Big)^{\theta n}(1+\rdist(I,J)^{-\theta })^{n}
\leq 
\Big(\frac{|J|}{|I|}\Big)^{\theta n}2^{n}\leq 2^{N_{2}}
$$

Both estimates together give us 
that $\psi_{in,in,in}$ is adapted to $\lambda_{2}J\times \lambda_{2}J$ with order $N_{2}$ and constant 
bounded by
$$
C^{2}\left(\frac{|J|}{|I|}\right)^{\frac{3}{2}-2\theta (1+N_{2})}\hspace{-.5cm}
\rdist(I,J)^{-(N-2\theta (1+N_{2}))}
\leq C^{2}\left(\frac{|J|}{|I|}\right)^{\frac{3}{4}}\rdist(I,J)^{-7}
$$
because, from the choice of $\theta$ and $N$, we have 
$\theta^{2} N<\frac{1}{16}$ and so, $\theta N_{2}<\frac{1}{4}$. This implies 
$2\theta (1+N_{2})<2\theta +1/2<3/4$.

Then, by the weak compactness property of $T$ we get
$$
|\Lambda(\psi_{in,in,in})|
\leq C \left(\frac{|J|}{|I|}\right)^{\frac{3}{4}}\rdist(I,J)^{-7}(F_{W}(\lambda_{2}J;M_{T,\epsilon })+\epsilon )
$$

b2) We now work with $\psi_{in,in, out}$. When $\frac{8}{32}\lambda_2|J|>2\lambda_{1}|I|$, we have that 
$\psi_{in,in,out}$ is the zero function and so, the estimate for $|\Lambda (\psi_{in,in,out})|$ holds
trivially. Hence, we only need to work the case $\frac{8}{32}\lambda_2|J|\leq 2\lambda_{1}|I|$.
%
%

In this case, by the extra assumption again, we have 
\begin{align}\label{uno}
\nonumber
|\psi_{in,in,out}(t,x)|&\leq |\psi_{in,in}(t,x)|
=\Big| \int_{c(J)}^{t} \partial_1 \psi_{in,in}(r,x)\, dr\Big| 
\\
&\leq C|t-c(J)||I|^{-\frac{3}{2}}\rdist(I,J)^{-N}|\varphi_{J} (x)|
\end{align}
where, as before, $\varphi_{J} $ is a adapted to $J$ with constant $C^{2}$. 

On the support of $\psi_{in,in}$, we have $|t-c(J)|\leq 2\lambda_{1}|\tilde{I}|=2\lambda_{1}|I|$. 
Moreover, on the support of $\psi_{in,in, out}$ we have  
$|t-c(J)|\geq \frac{8}{32}\lambda_2|J|$ and
$|x-c(J)|\leq \frac{2}{32}\lambda_2|J|$. 
With the last two inequalities we obtain $2|x-c(J)|<|t-c(J)|$     
and so,
we use the Calder\'on-Zygmund kernel representation
and the mean zero of $\psi_{in, in, out}$ in the variable $x$ to write
$$
\Lambda (\psi_{in, in, out})
=\int \psi_{in, in, out}(t,x) (K(t,x)-K(t,c(J)))\, dtdx
$$

Now, we use
the smoothness property of a compact Calder\'on-Zygmund kernel.
We recall that, as described in Lemma \ref{compactCZ}, we have for some $0 < \delta < 1$ that
\begin{multline*}
|K(t,x)-K(x',t)|
\lesssim  
\frac{|x-x'|^{\delta}}{|t-x|^{1+\delta}}
L(|t-x|)S(|x-x'|)
D\Big(1+\frac{|x|}{1+|t-x|}\Big)
\end{multline*}
Note also the last remarks at the end of 
section \ref{kernelandadmissible}.
From this estimate and the bound for $\psi_{in, in, out}$ in (\ref{uno}), we get
\begin{align*}
|\Lambda (\psi_{in, in, out})|&\lesssim C\int |t-c(J)||I|^{-\frac{3}{2}}\rdist(I,J)^{-N}|\varphi_{J} (x)| \frac{|x-c(J)|^\delta}{|t-c(J)|^{1+\delta}}
\\
&L(|t-c(J)|)
S(|x-c(J)|)D\Big(1+\frac{|c(J)|}{1+|t-c(J)|}\Big)\, dtdx
\end{align*}

By using $2\lambda_{1}|I|>|t-c(J)|>\frac{8}{32}\lambda_2|J|$,
$|x-c(J)|<\frac{2}{32}\lambda_2|J|$, $c(\lambda_{2}J)=c(J)$, 
the monotonicity of 
$L$, $S$ and $D$,  and the bound of $\varphi_{J}$, we get
\begin{align*}
C^{2}&|I|^{-\frac{3}{2}}\rdist(I,J)^{-N}|J|^{-\frac{1}{2}}
L(\lambda_{2}|J|)S(\lambda_{2}|J|)D\Big(1+\frac{|c(\lambda_{1}\tilde{I})|}{1+\lambda_1|\tilde{I}|}\Big)
\\
&\int_{|x-c(J)|<2\lambda_{2}|J|}\hspace{-.1cm} |x-c(J)|^\delta dx 
\int_{\frac{8}{32}\lambda_{2}|J|<|t-c(J)|<2\lambda_{1}|I|}\hspace{-.1cm}|t-c(J)|^{-\delta}dt
\end{align*}

Since $\delta <1$, the product of these integrals can be bounded by 
\begin{align*}
(2\lambda_{2}|J|)^{1+\delta }&((2\lambda_{1}|I|)^{1-\delta}-(\frac{8}{32}\lambda_{2}|J|)^{1-\delta})
\lesssim (\lambda_{2}|J|)^{1+\delta }
(\lambda_{1}|I|)^{1-\delta}
\\
&=(|J|^{(1-\theta)}\diam(I\cup J))^{\theta(1+\delta )}\diam(I\cup J)^{1-\delta }
\\
&=|J|^{1+\delta'}|I|^{1-\delta' }\rdist(I,J)^{1-\delta'}
\end{align*}
since $\delta'=\delta-\theta(1+\delta )$. 
This way, we can bound 
\begin{align*}
|\Lambda (\psi_{in, in, out})|
&\lesssim C^{2}|I|^{-\frac{3}{2}}\rdist(I,J)^{-N}|J|^{-\frac{1}{2}}
\\
&\hspace{1cm}|J|^{1+\delta' }|I|^{1-\delta'}\rdist(I,J)^{1-\delta'}
F_{K}(\lambda_{2}J,\lambda_{1}\tilde{I})
\\
&\leq  C^{2}\Big(\frac{|J|}{|I|}\Big)^{\frac{1}{2}+\delta' } \rdist(I,J)^{-7}
F_{K}(\lambda_{2}J,\lambda_{1}\tilde{I})
\end{align*}


\vskip10pt
{\bf c)} It only remains to discuss $\psi_{in,out}$. On the support of $\psi_{in,out}$, we have  
$|t-c(J)|>\lambda_1|I|={\rm diam}(I\cup J)$ and 
$|x-c(J)|\leq 2/32\lambda_2|J|=2/32|J|^{1-\theta }{\rm diam}(I\cup J)^{\theta}$. Since
$1\leq |J|^{-1}{\rm diam}(I\cup J)$ implies $|J|^{1-\theta}{\rm diam}(I\cup J)^{\theta}<{\rm diam}(I\cup J)$, we get $2|x-c(J)|<|t-c(J)|$. Then,
the support of $\psi_{in,out}$ is disjoint with the diagonal and 
we can use the Calder\'on-Zygmund kernel representation 
and the zero mean of $\psi_{in,out}$ in the variable $x$
to write
$$
\Lambda (\psi_{in,out})
=\int \psi_{in,out}(t,x) (K(t,x)-K(t,c(J)))\, dtdx
$$
Moreover, we use the smoothness property of a compact Calder\'on-Zygmund kernel to estimate 
$$
|\Lambda (\psi_{in,out})|\lesssim  \int |\psi_{in,out}(t,x)| \frac{|x-c(J)|^\delta}{|t-c(J)|^{1+\delta}}L(|t-c(J)|)
$$
$$
S(|x-c(J)|)D\Big(1+\frac{|c(J)|}{1+|t-c(J)|}\Big)\, dtdx
$$
\vskip5pt
Now, 
using $|t-c(J)|>{\rm diam}(I\cup J)$, 
$|x-c(J)|<2/32|J|^{1-\theta}{\rm diam}(I\cup J)^{\theta}<\diam(I\cup J)$ 
and the monotonicity of 
$L$ and $S$, we bound by
\begin{align}\label{43}
C^{2}&|J|^{(1-\theta)\delta }\frac{{\rm diam}(I\cup J)^{\theta \delta}}{{\rm \diam}(I\cup J)^{1+\delta}}
L(\diam(I\cup J) )S(\diam(I\cup J))
\\
\nonumber
&\int_{|t-c(J)|>{\rm diam}(I\cup J)} |\psi_{in,out}(t,x)|
D\Big(1+\frac{|c(J)|}{1+|t-c(J)|}\Big)\, dtdx
\end{align}
We denote $d=\diam(I\cup J)=|\langle I,J\rangle |$ and 
$\Delta_{k}=\{t\in \mathbb R: 2^{k}d\leq |t-c(J)|<2^{k+1}d\}$. By monotonicity of $D$, we bound the integral by
\begin{equation}\label{integral}
\sum_{k\geq 0}\int_{\Delta_{k}} |\psi_{in,out}(t,x)|
dtdx \, 
D\Big(1+\frac{|c(J)|}{1+2^{k+1}d}\Big)
\end{equation}
Now, since $|c(I)|-|c(J)|\leq |c(I)-c(J)|\leq \diam(I\cup J)=d$
and $\Big|c(\langle I\cup J\rangle )-\frac{|c(I)+c(J)|}{2}\Big| \leq d/2$, we have
$$
1+2^{k+1}d+|c(J)|
\geq 1+2^{k}d+d/2+|c(J)|
$$
$$
\geq 1+2^{k}d+(|c(I)|+|c(J)|)/2
\geq 1+2^{k}d+|c(\langle I, J\rangle )|
-d/2
$$
and so, 
$$
1+\frac{|c(J)|}{1+2^{k+1}d}
\geq \frac{1}{1+2^{k+1}d}\Big(1+2^{k}d+|c(\langle I, J\rangle )|-\frac{d}{2}\Big)
$$
$$
\geq \frac{1}{2}\Big(1+\frac{|c(\langle I, J\rangle )|}{1+2^{k}d}-\frac{1}{2}\Big)
\geq \frac{1}{4}\Big(1+\frac{|c(2^{k}\langle I, J\rangle )|}{1+2^{k}|\langle I,J\rangle |}\Big)
=\frac{1}{4}\rdist(2^{k}\langle I,J\rangle ,\mathbb B)
$$

This way, by  the monotonicity of $D$ again, (\ref{integral}) can be bounded by 
\begin{equation}\label{integral2}
\sum_{k\geq 0}\int_{\Delta_{k}} |\psi_{in,out}(t,x)|
\, dtdx \, 
D(\rdist(2^{k}\langle I,J\rangle ,\mathbb B))
\end{equation}
We work now to bound the integral. First we note that, since 
$
2^{k}d\leq |t-c(J)|\leq |t-c(I)|+|c(I)-c(J)|\leq |t-c(I)|+d
$, 
we also have $(2^{k}-1)d\leq |t-c(I)|<2^{k+1}d$.

Then, since $\psi_{in,out}$ is adapted to $I\times J$, we can bound as follows:
$$
\int_{\Delta_{k}} |\psi_{in,out}(t,x)|
\, dtdx 
\lesssim |I|^{-\frac{1}{2}}\int_{(2^{k}-1)d\leq |t-c(I)|} w_{I}(t)^{-N}dt \, |J|^{\frac{1}{2}}
$$ 
$$
\lesssim |I|^{-\frac{1}{2}}(1+|I|^{-1}(2^k-1)d)^{-N}|I|\, |J|^{\frac{1}{2}}
\leq  |I|^{\frac{1}{2}}|J|^{\frac{1}{2}}2^{-Nk}
$$ 
since $|I|^{-1}d=\rdist(I,J)\geq 1$. Then, (\ref{integral2}) can be bounded by
$$
|I|^{\frac{1}{2}}|J|^{\frac{1}{2}}\sum_{k\geq 0}2^{-Nk}
D(\rdist(2^{k}\langle I,J\rangle ,\mathbb B))
\leq |I|^{\frac{1}{2}}|J|^{\frac{1}{2}}\tilde{D}(\langle I,J\rangle )
$$
with $\displaystyle{\tilde{D}(K)=\sum_{k\geq 0}2^{-k}
D(\rdist(2^{k}K,\mathbb B))}$. 
Then, we bound \eqref{43} by 
\begin{align*}
|\Lambda (\psi_{in,out})|&\lesssim C^{2}|J|^{(1-\theta)\delta } {\rm diam}(I\cup J)^{-1-\delta+\theta \delta}
|I|^{\frac{1}{2}}|J|^{\frac{1}{2}}
\\
&\hspace{1cm}L(\diam(I\cup J))S(\diam(I\cup J))\tilde{D}(\langle I,J\rangle )
\\
&\leq  C^{2 }\Big(\frac{|J|}{|I|}\Big)^{\frac{1}{2}+\delta' } \rdist(I,J)^{-(1+\delta' )}
\tilde{F}_{K}(\langle I, J\rangle)
\end{align*}
since $(1-\theta )\delta >\delta -\theta (1+\delta)=\delta'$. 
This shows the stated bound for $\Lambda(\psi_{in,out})$
and completes the proof.


\end{proof}

\section{$L^{p}$ compactness}\label{L2}

We start by describing the way 
to choose a wavelet basis of $L^{2}(\mathbb R)$ and how we use this basis to decompose the operators under study. For this, we use the results contained in the books 
\cite{Chui} and \cite{HerWeiss}.

Given a function $\psi $ 
and a dyadic interval $I=2^{-j}[k,k+1]$, $j,k\in \mathbb Z$, we denote $l(I)=\min\{ x: x\in I\}$ and
$$
\psi_{I}(x)
={\mathcal T}_{l(I)}\mathcal D_{|I|}^{2}\psi (x)
=2^{j/2}\psi(2^{j}x-k)
$$

\begin{theorem}
Let $\psi \in L^{2}(\mathbb R)$ with $\| \psi\|_{L^{2}(\mathbb R)}=1$. Then, 
$\{ \psi_{I}\}_{I\in {\mathcal D}}$ is an orthonormal wavelets basis of $L^{2}(\mathbb R)$
if and only if 
$$
\sum_{k\in \mathbb Z}|\hat{\psi}(\xi +k)|^{2}=1
\hskip20pt, \hskip25pt
\sum_{k\in \mathbb Z}\hat{\psi}(2^{j}(\xi +k))\overline{\hat{\psi}(\xi +k)}=0
$$
for all $\xi \in \mathbb R$ and all $j\geq 1$.
\end{theorem}

%

For our purposes, we take $\psi $ satisfying the hypotheses of previous theorems with the additional condition that 
$\psi \in C^{N}(\mathbb R)$ and it is adapted to $[-\frac{1}{2},\frac{1}{2}]$ with constant $C>0$ and order $N$. 
We remark the crucial fact that for every interval $I\in {\mathcal D}$, every wavelet function 
$\psi_{I}$ is a bump function adapted to $I$ with the same constant and order. Several examples of constructions of systems of wavelets with any required order of differentiability can also be found in \cite{HerWeiss}.

In this setting, the continuity of $T$ with respect the topology of ${\mathcal S}_{N}(\mathbb R)$, 
allows to write for every $f,g\in {\mathcal S}(\mathbb R)$, 
$$
\langle T(f),g\rangle 
=\sum_{I,J\in {\cal D}}\langle f,\psi_{I}\rangle \langle g, \psi_{J}\rangle \langle T(\psi_I),\psi_J\rangle 
$$
where the sums run over the family of all dyadic intervals of $\mathbb R$ and convergence is understood in 
the topology of  ${\mathcal S}_{N}(\mathbb R)$.
Furthermore, 
\begin{equation}\label{orthopro0}
\langle P_{M}^{\perp}(T(f)),g\rangle 
=\sum_{I\in {\cal D}}\sum_{J\in {\cal D}_{M}^{c}}\langle f,\psi_{I}\rangle \langle g, \psi_{J} \rangle 
\langle T(\psi_I),\psi_J\rangle 
\end{equation}
where the summation is performed as in equation \eqref{ortho}.

\vskip10pt
We prove now our main result about compactness on $L^{2}(\mathbb R)$ of Calder\'on-Zygmund operators
under the special cancellation conditions.

\begin{theorem}\label{L2bounds}
Let
$T$ 
be a continuous linear operator 
with a compact Calder\'on-Zygmund kernel such that $T$ satisfies 
the weak compactness condition
and the special cancellation conditions
$T(1)=T^{*}(1)=0$.

Then, $T$ can be extended to a compact operator on $L^2(\mathbb R)$.
\end{theorem}
\proof

By Theorem \ref{charofcompact}, to prove compactness of $T$, we 
need to check that
$P_{M}^{\perp}(T_{b})$ converges to zero in the operator norm 
$\| \cdot \|_{L^{2}\rightarrow L^{2}}$ when $M$ tends to infinity.
For this, it is enough to prove that 
$
\langle P_{M}^{\perp}(T(f)),g\rangle 
$
tends to zero uniformly for all $f,g \in {\mathcal S}(\mathbb R)$ in the unit ball of 
$L^{2}(\mathbb R)$.

We show first that for every 
 $\epsilon >0$ there is $M_{0}\in \mathbb N$ 
such that for any $M>M_{0}$, 
we have $F(I_{1},\ldots ,I_{6};M_{T,\epsilon })\lesssim \epsilon $ for $I_{i}\in {\cal I}_{M}^{c}$. This will follow once we prove
the inequality $F(I;M_{T,\epsilon })\lesssim \epsilon $ for every $I\in {\cal D}_{M}^{c}$. We note that the implicit constants only depend on the admissible functions. 

For $\epsilon >0$, let that $M_{T,\epsilon }>0$ 
be the parameter appearing in the Definition \ref{WB} of the weak compactness condition.
Once fixed $M_{T,\epsilon }$,  
there is
$M_{0}\in \mathbb N$, depending on $\epsilon, M_{T,\epsilon }$, with $M_{0}>M_{T,\epsilon }$,  
such that for any $M>M_{0}$, 
\begin{align*}
L_{K}(&2^{M})+S_{K}(2^{-M})+ D_{K}(M)<\epsilon
\\
L_{W}(&2^{M-M_{T,\epsilon }})+S_{W}(2^{-(M-M_{T,\epsilon })})+D_{W}(M/M_{T,\epsilon })<\epsilon
\end{align*} 

Let $I\in {\cal I}_{M}^{c}$. We prove the claim by considering the following cases: 
\vskip10pt
\noindent 1) If $|I|>2^{M}$ then, since $L_{K}$ and $L_{W}$ are non-increasing, we have  
$$
F(I;M_{T,\epsilon})\lesssim L_{K}(|I|)+L_{W}\Big(\frac{|I|}{2^{M_{T,\epsilon }}}\Big)
\leq L_{K}(2^{M})+L_{W}(\frac{2^{M}}{2^{M_{T,\epsilon }}})<\epsilon 
$$


\noindent 2) If $|I|<2^{-M}$ then, since $S_{K}$ and $S_{W}$ are non-decreasing, we get 
$$
F(I;M_{T,\epsilon})\lesssim S_{K}(|I|)+S_{W}(2^{M_{T,\epsilon }}|I|)
\leq S_{K}(2^{-M})+S_{W}\Big(\frac{2^{M_{T,\epsilon }}}{2^{M}}\Big)
<\epsilon 
$$

\noindent 3) If $2^{-M}\leq |I|\leq 2^{M}$ with $\rdist (I,\mathbb B_{2^{M}})>M$ then, as we saw in the remark after 
Definition \ref{Imdef}, 
$|c(I)|>(M-1)2^{M}$. Therefore,
$$
M_{T,\epsilon }^{-1}\rdist (I,\mathbb B_{2^{M_{T,\epsilon }}})
\geq M_{T,\epsilon }^{-1}\Big(1+\frac{|c(I)|}{\max(|I|,2^{M_{T,\epsilon }})}\Big)
>M/M_{T,\epsilon }
$$
We can apply a similar reasoning to show also $\rdist (I,\mathbb B)>M$. 
Then, since $D_{W}$ is non-increasing, we have 
$$
F(I;M_{T,\epsilon})\lesssim  D_{K}(\rdist (I,\mathbb B))
+D_{W}(M_{T,\epsilon }^{-1}\rdist (I,\mathbb B_{2^{M_{T,\epsilon }}}))
$$
$$
\leq D_{K}(M)+D_{W}(M/M_{T,\epsilon })<\epsilon 
$$

%


\vskip10pt
Now, for every fixed $\epsilon >0$ and chosen $M_{0}\in \mathbb N$, we are going to prove that for all 
$M>M_{0}$ such that $M2^{-M\delta }+M^{-\delta }<\epsilon $, we have 
$$
|\langle P_{2M}^{\perp}(T(f)),g\rangle |\lesssim \epsilon
$$
with the implicit constant depending on $\delta >0$ and the wavelets basis.

By equation (\ref{orthopro0}), we have that  
\begin{equation}\label{orthopro}
\langle P_{2M}^{\perp}(T(f)),g\rangle 
=\sum_{I\in {\cal D}}\sum_{J\in {\cal D}_{2M}^{c}}\langle f,\psi_{I}\rangle \langle g, \psi_{J}\rangle 
\langle T(\psi_I),\psi_J\rangle 
\end{equation}
where $(\psi_{I})_{I}$ is the wavelet basis given in the remarks preceding the statement of the theorem. Note that there exist a fixed order $N$ and constant $C$, such that each $\psi_{I}$ is adapted to $I$ with constant $C$ and order $N$. Furthermore, we are free to choose this $N$ so as to ensure that Proposition \ref{symmetricspecialcancellation} applies.

From now on, we work to obtain bounds of (\ref{orthopro}) when the
sum runs over finite families of dyadic intervals in such way that the bounds are independent of the chosen families.

In view of Proposition \ref{symmetricspecialcancellation},
we parametrize the sums according to eccentricities and relative distances of the intervals:
\begin{equation}\label{2Mpara}
\langle P_{2M}^{\perp}(T(f)),g\rangle=\sum_{e\in \mathbb Z}\sum_{n\in \mathbb N}
\sum_{\tiny \begin{array}{c}J{\in \cal D}_{2M}^{c}\end{array}}
\sum_{I\in J_{e,n}}
\langle f,\psi_{I}\rangle \langle g, \psi_{J}\rangle \langle T(\psi_I),\psi_J\rangle 
\end{equation}
where for fixed eccentricity $e\in \mathbb Z$, relative distance $n\in \mathbb N$ and every given interval $J$,
we define the family
$$
J_{e,n}=\{ I\in {\mathcal D}:|I|=2^{e}|J|, n\leq \rdist(I,J)< n+1 \}
$$
Notice that by symmetry the family $\{ (I,J): I\in J_{e,n}\}$ can be also parameterized as 
$\{ (I,J): J\in I_{-e,n}\}$. 

By Proposition \ref{symmetricspecialcancellation}, with $\delta$ denoting $\delta'$, we have for $\epsilon>0$, $M_{T,\epsilon}\in \mathbb N$
$$
|\langle T(\psi_I),\psi_J\rangle |
\lesssim 2^{-|e|(\frac{1}{2}+\delta )}n^{-(1+\delta )}\big( F(I_{1},\ldots ,I_{6}; M_{T,\epsilon })+\epsilon \big)
$$
where $I_{1}=I$, $I_{2}=J$, $I_{3}=\langle I,J\rangle$ , $I_{4}=\lambda_{1}\tilde{K}_{M}$, 
$I_{5}=\lambda_{2}\tilde{K}_{M}$ and  $I_{6}=\lambda_{2}K_{m}$ with parameters $\lambda_{1},\lambda_{2}\geq 1$ explicitly 
provided by the mentioned corollary.  
To simplify notation, we will simply write $F(I_{i})$. 
We also note that the implicit constant might depend on $\delta $ and the wavelet basis, but it is universal otherwise. 

Therefore, we can bound \eqref{2Mpara} as follows:
\begin{align}\label{moduloin2}
|\langle P_{2M}^{\perp}(T(f)),g\rangle |&\lesssim
\sum_{e\in \mathbb Z}\sum_{n\in \mathbb N}2^{-|e|(\frac{1}{2}+\delta )}n^{-(1+\delta )}
\\
\nonumber
&\sum_{\tiny \begin{array}{c}J{\in \cal D}_{2M}^{c}\end{array}}\sum_{I\in J_{e,n}}
\big( F(I_{i})+\epsilon \big)
|\langle f,\psi_{I}\rangle | |\langle g,\psi_{J}\rangle |
\end{align}
Now, in order to estimate this last quantity, we divide the study into six cases:
\vskip5pt
\noindent
\hspace{-.5cm}
\begin{minipage}{7cm}
\begin{enumerate}
\item $I_{i}\notin {\cal D}_{M}$ for all $i=1,\ldots ,6$
\item $I\in {\cal D}_{M}$
\item $\langle I, J\rangle \in {\cal I}_{M}$
\end{enumerate}
\end{minipage}
\hspace{-.5cm}
\begin{minipage}{7cm}
\begin{enumerate}
\item[(4)] $I\notin {\cal D}_{M}$ but $\lambda_{1}\tilde{K}_{max}\in {\cal I}_{M}$
\item[(5)] $I\notin {\cal D}_{M}$ but $\lambda_{2}\tilde{K}_{max}\in {\cal I}_{M}$
\item[(6)] $I\notin {\cal D}_{M}$ but $\lambda_{2}K_{min} \in {\cal I}_{M}$
\end{enumerate}
\end{minipage}

\vskip10pt
1) In the first case, we have 
$F(I_{i})
<\epsilon $. Then, by Cauchy inequality, we bound the second line of (\ref{moduloin2}) 
corresponding to this case by
$$
2\epsilon 
\Big( \sum_{I\in {\mathcal D}}\hspace{-.1cm} \sum_{\tiny \begin{array}{c}J\! \in \! I_{-e,n}\end{array}}|\langle f,\psi_{I}\rangle |^2\Big)^{\frac{1}{2}}
\Big( \sum_{J\in {\mathcal D}}\hspace{-.1cm} 
\sum_{\tiny \begin{array}{c}I\! \in \! J_{e,n}\end{array}}|\langle g,\psi_{J}\rangle |^2\Big)^{\frac{1}{2}}
$$

Now, for fixed $J$ and $n\in \mathbb N$ there are $2^{-\min(e,0)}$ dyadic intervals $I$ such that
$|I|=2^{e}|J|$ and  $n\leq \rdist(I, J)<n+1$. This implies that the cardinality of
$J_{e,n}$ is $2^{-\min(e,0)}$
and so, the cardinality of $I_{-e,n}$ is $2^{-\min(-e,0)}=2^{\max(e,0)}$.
Therefore, previous expression coincides with
\begin{eqnarray*}
\epsilon\Big( 2^{\max(e,0)}\sum_{I\in {\mathcal D}}|\langle f,\psi_{I}\rangle |^2\Big)^{\frac{1}{2}}
\Big( 2^{-\min(e,0)}\sum_{J\in {\mathcal D}}|\langle g,\psi_{J}\rangle |^2\Big)^{\frac{1}{2}}
\lesssim \epsilon 2^{\frac{|e|}{2}}\| f\|_2\| g\|_2
\end{eqnarray*}
since $2^{\max(e,0)}2^{-\min(e,0)}=2^{|e|}$. 
Therefore, the correspoding the terms in (\ref{moduloin2}) can be bounded by a constant times
\begin{equation*}
\epsilon \Big( \sum_{e\in \mathbb Z}2^{-|e|\delta }\sum_{n\in \mathbb N}n^{-(1+\delta )}\Big)\| f\|_2\| g\|_2
\lesssim \epsilon \| f\|_2\| g\|_2
\end{equation*}
This finishes the first case.


\vskip10pt 
In the remaining cases, we will not use the smallness of $F$. Instead, we will employ  
the geometrical features of the intervals $I$ and $J$, which make either their eccentricity or their relative distance very extreme. 

2) We deal with the case
$I\in {\cal D}_{M}$, that is,  
$2^{-M}\leq |I|\leq 2^{M}$ and $\rdist(I,\mathbb B_{2^{M}})\leq M$.
Notice that, since $F$ is bounded, we can estimate $F(I_{i})+\epsilon \lesssim 1$. 

Since $J\in {\mathcal D}_{2M}^{c}$, we separate into the three usual cases: $|J|>2^{2M}$, $|J|<2^{-2M}$
and $2^{-2M}\leq |J|\leq 2^{2M}$ with $\rdist(J,\mathbb B_{2^{2M}})>2M$.

2.1) When $|J|>2^{2M}$, 
we have 
$2^{e}|J|=|I|\leq 2^{M}$ and so, $2^{e}\leq 2^{M}|J|^{-1}\leq 2^{-M}$, that is, 
$e\leq -M$. 
Then, we bound the terms in (\ref{moduloin2}) corresponding to this case by 
$$
\sum_{\tiny \begin{array}{c}e\leq -M\end{array}}
\sum_{n\geq 1}
2^{-|e|(1/2+\delta )}n^{-(1+\delta )}
\sum_{J\in {\cal D}_{2M}^{c}} \sum_{I\in J_{e,n}}|\langle f,\psi_{I}\rangle ||\langle g,\psi_{J}\rangle |
$$
$$
\lesssim \Big(\sum_{\tiny \begin{array}{c}e\leq -M\end{array}}
2^{-|e|\delta }\sum_{n\geq 1}
n^{-(1+\delta )}\Big)\| f\|_2\| g\|_2
\lesssim 2^{-M\delta }\| f\|_2\| g\|_{2}<\epsilon \| f\|_2\| g\|_2
$$
by the choice of $M$.

2.2) The case $|J|<2^{-2M}$ is completely symmetrical and amounts to changing $e\leq -M$ by $e\geq M$
in the previous case. 

2.3) In the case $2^{-2M}\leq |J|\leq 2^{2M}$  and $\rdist(J,\mathbb B_{2^{2M}})\geq 2M$, we parametrize
by size $|J|=2^{k}$ with $-2M\leq k\leq 2M$. Moreover, by the remarks after Definition 
\ref{Imdef}, we have $|c(J)|\geq (2M-1)2^{2M}$,
and, since $I\in {\mathcal D_{M}}$, we also get
$|c(I)|\leq (M-1/2)2^{M}$. This implies 
$$
|c(I)-c(J)|\geq |c(J)|-|c(I)|
\geq M2^{2M}
$$ 
and, since $\max(|I|,|J|)\leq 2^{2M}$, we get
$$
n+1> \rdist(I,J)
\geq \frac{|c(I)-c(J)|}{\max(|I|,|J|)}
\geq M
$$

This way, we can bound the relevant terms in (\ref{moduloin2}) by 
$$
\sum_{\tiny \begin{array}{c}e\in \mathbb Z\end{array}}
\sum_{n\geq M-1}
2^{-|e|(1/2+\delta )}n^{-(1+\delta )}
\sum_{J\in {\cal D}_{2M}^{c}} \sum_{I\in J_{e,n}}|\langle f,\psi_{I}\rangle ||\langle g,\psi_{J}\rangle |
$$
$$
\lesssim \Big(\sum_{\tiny \begin{array}{c}e\in \mathbb Z\end{array}}
2^{-|e|\delta }\sum_{n\geq M-1}
n^{-(1+\delta )}\Big)\| f\|_2\| g\|_2
\lesssim M^{-\delta }\| f\|_2\| g\|_2<\epsilon \| f\|_2\| g\|_2
$$
again by the choice of $M$.

3) Now, we deal with the case when $\langle I, J\rangle \in {\mathcal I}_{M}$, that is, when 
$2^{-M}\leq |\langle I, J\rangle |\leq 2^{M}$ and $\rdist (\langle I,J\rangle ,\mathbb B_{2^{M}})\leq M$. The last inequality implies that 
$|c(\langle I, J\rangle )|\leq M2^{M}$.

Since $|c(\langle I, J\rangle )-\frac{c(I)+c(J)}{2}|\leq \diam(I\cup J)/2=|\langle I,J\rangle |/2$, we have
\begin{equation}\label{centers}
|c(I)+c(J)|\leq 2|c(\langle I, J\rangle )|+|\langle I,J\rangle |
\leq 2(M+1)2^{M}
\end{equation}
Now, we divide into three cases:

3.1) When $|J|> 2^{2M}$ we have 
that $|\langle I, J\rangle|\geq |J|>2^{2M}$ implies $\langle I, J\rangle \notin {\cal D}_{M}$ and so, we do not
need to consider this case. 

3.2) When $2^{-2M}\leq |J|\leq 2^{2M}$ with $\rdist(J,\mathbb B_{2^{2M}})\geq 2M$, we have that
$|c(J)|>(2M-1)2^{2M}>M2^{M}$ . 

If $\sign \, c(I)=-\sign \, c(J)$ we get 
$$
|\langle I, J\rangle|
\geq |c(I)-c(J)|=|c(I)|+|c(J)|>|c(J)|
>M2^{2M}
$$
which is contradictory with $\langle I, J\rangle \in {\mathcal I}_{M}$. 

Otherwise, if $\sign \, c(I)=\sign \, c(J)$ we have 
$$
|c(I)+c(J)|= |c(I)|+|c(J)|>M2^{2M}
$$
which is also contradictory with \eqref{centers}. 
So, we do not consider this case. 


3.3) The only remaining case is when $|J|< 2^{-2M}$.
When $e\leq 0$, 
$$
n+1> \rdist(I,J)=|J|^{-1}\diam(I\cup J)\geq 2^{2M}2^{-M}=2^{M}
$$
Meanwhile, when $e\geq 0$, 
$$
n+1> |I|^{-1}\diam(I\cup J)=2^{-e}|J|^{-1}|\langle I\cup J\rangle |\geq 2^{-e}2^{2M}2^{-M}=2^{M-e}
$$

Therefore, we bound the relevant part of  (\ref{moduloin2}) by a constant times
$$
\sum_{\tiny \begin{array}{c}e\leq 0\end{array}}
\sum_{n\geq 2^{M}-1}
2^{-|e|(\frac{1}{2}+\delta )}n^{-(1+\delta )}
\sum_{J\in {\cal D}_{2M}^{c}} \sum_{I\in J_{e,n}}|\langle f,\psi_{I}\rangle ||\langle g,\psi_{J}\rangle |
$$
$$
+\sum_{\tiny \begin{array}{c}e\geq 0\end{array}}
\sum_{n\geq \max(2^{M-e}-1,1)}
2^{-|e|(\frac{1}{2}+\delta )}n^{-(1+\delta )}
\sum_{J\in {\cal D}_{2M}^{c}} \sum_{I\in J_{e,n}}|\langle f,\psi_{I}\rangle ||\langle g,\psi_{J}\rangle |
$$
$$
\leq \Big(\sum_{\tiny \begin{array}{c}e\leq 0\end{array}}
2^{-|e|\delta }\sum_{n\geq 2^{M-1}}
n^{-(1+\delta )}\Big)\| f\|_2\| g\|_2
$$
$$
+ \Big(\hspace{-.5cm} \sum_{\tiny \begin{array}{c}0\leq e\leq M-1\end{array}}
\hspace{-.5cm}2^{-|e|\delta }\sum_{n\geq 2^{M-e}-1}
n^{-(1+\delta )}
+\hspace{-.3cm}\sum_{\tiny \begin{array}{c}M\leq e\end{array}}
2^{-|e|\delta }\sum_{n\geq 1}
n^{-(1+\delta )}\Big)\| f\|_2\| g\|_2
$$
$$
\lesssim \Big( 2^{-M\delta }
+\sum_{\tiny \begin{array}{c}0\leq e\leq M-1\end{array}}
2^{-e\delta }2^{-(M-e)\delta }
+\sum_{\tiny \begin{array}{c}M\leq e\end{array}}
2^{-e\delta }\Big)\| f\|_2\| g\|_2
$$
$$
\lesssim 2^{-M\delta }\| f\|_2\| g\|_2
+(M2^{-M\delta }
+2^{-M\delta })\| f\|_2\| g\|_2
<\epsilon \| f\|_2\| g\|_2
$$

6) We deal now with the case $\lambda_{2}K_{min}\in {\cal I}_{M}$.

6.1) When $|J|> 2^{2M}$, we have two cases. Whenever $e>0$
then, $K_{min}=J$ and so, 
$|\lambda_{2}J|\geq |J|\geq 2^{2M}$ which is contradictory with $\lambda_{2}J\in {\cal I}_{M}$.

On the other hand, when $e\leq 0$ we have $K_{min}=I$ and  $|I|\leq |\lambda_{2}I|\leq 2^{M}$. 
Then, $2^{e}=|I|/|J|\leq 2^{-M}$ and so, $e\leq -M$. Therefore, the arguments of the case 2.1) show that the corresponding part 
of (\ref{moduloin2}) can be bounded by $\epsilon \| f\|_2\| g\|_2$.

6.2) When $2^{-2M}\leq |J|\leq 2^{2M}$ with $\rdist(J,\mathbb B_{2^{2M}})\geq 2M$, we have  
$|c(J)|>(2M-1)2^{2M}$. Now, we divide into the same two cases. 

When $e\geq 0$, we know $K_{min}=J$ and then, 
$2^{-M}\leq |\lambda_{2}J|\leq 2^{M}$ with $\rdist(\lambda_{2}J,\mathbb B_{2^{M}})\leq M$. This leads to  
the following contradiction:
$$
M\geq \rdist(\lambda_{2}J,\mathbb B_{2^{M}})
> 
2^{-M}|c(J)|
\geq (2M-1)2^{M}
$$
On the other hand, when $e\leq0 $ we have $K_{min}=I$ and then, 
$|c(I)|=|c(\lambda_{2}I)|\leq (M-1)2^{M}$. This implies $|c(I)-c(J)|> M2^{2M}$ and
$$
n+1> \rdist(I,J)
\geq \frac{|c(I)-c(J)|}{|J|}
\geq M
$$
Then, the same arguments of the case 2.3) provide the bound $\epsilon \| f\|_2\| g\|_2$.

6.3) When $|J|< 2^{-2M}$, we study as follows. 
If $e\geq 0$, we have $K_{\min}=J$ and so, 
$|\lambda_{2}J|\geq 2^{-M}$. This implies $\lambda_{2}\geq 2^{-M}|J|^{-1}> 2^{M}$ and
$$
2^{M}<\lambda_{2}=\Big(\frac{\diam(I\cup J)}{|J|}\Big)^{\theta}
\leq \Big(\frac{|I|}{|J|}\Big)^{\theta }\rdist(I,J)^{\theta}
<2^{e\theta }(n+1)
$$
Meanwhile, if $e\leq 0$, we have $K_{\min}=I$ and then, $|\lambda_{2}I|\geq 2^{-M}$. We also have 
$|I|\leq |J|\leq 2^{-2M}$. 
All this implies $\lambda_{2}\geq 2^{-M}|I|^{-1}> 2^{M}$ and
$$
2^{M}<\lambda_{2}
= \Big(\frac{\diam(I\cup J)}{|I|}\Big)^{\theta}
\leq \Big(\frac{|J|}{|I|}\Big)^{\theta }\rdist(I,J)^{\theta}
<2^{-e\theta }(n+1)
$$

In any case we get $n>2^{-|e|\theta}2^{M}$ and, since $\theta <1$, 
previous arguments show that  the relevant part 
of (\ref{moduloin2}) can be bounded by
$$
\Big(\sum_{\tiny \begin{array}{c}e\in \mathbb Z\end{array}}
2^{-|e|\delta }\sum_{n\geq 2^{-|e|\theta}2^{M}}
n^{-(1+\delta )}\Big)\| f\|_2\| g\|_2
$$
$$
\lesssim \Big(\sum_{\tiny \begin{array}{c}e\in \mathbb Z\end{array}}
2^{-|e|\delta }2^{|e|\theta \delta }2^{-M\delta }
\Big)\| f\|_2\| g\|_2
\lesssim 2^{-M\delta }\| f\|_2\| g\|_2
\leq \epsilon \| f\|_2\| g\|_2
$$

\vskip10pt
Finally, we note that similar type of calculations are enough to deal with the other two remaining cases 4) and 5). This finishes the proof of compactness on $L^{2}(\mathbb R)$. 

\vskip20pt
Since $T$ is bounded on $L^{2}(\mathbb R)$, from the classical theory, we know that the operator $T$ is bounded on $L^{p}(\mathbb R)$ for all $1<p<\infty $. Then, since we have proved that $T$ is 
compact on $L^{2}(\mathbb R)$, we can use classical interpolation techniques to also obtain 
compactness of $T$ on $L^{p}(\mathbb R)$. 

We refer to the following theorem, whose proof,
in a more general setting, can be found in \cite{Kra} and also in \cite{Cot}.

\begin{theorem}
Let $1\leq p_{1}, r_{1}, p_{2}, r_{2} \leq \infty$ be a set of indices with $r_{1}<\infty$. Let $T$ be a given linear operator which is continuous simultaneously as a mapping from $L^{p_{1}}$ to $L^{r_{1}}$ and from $L^{p_{2}}$ to $L^{r_{2}}$. Assume in addition that $T$ is compact as a mapping from $L^{p_{1}}$ to $L^{r_{1}}$. 
Then $T$ is compact as a mapping from $L^{p}$ to $L^{r}$, where $1/p=t/p_{1}+(1-t)/p_{2}$, 
$1/r=t/r_{1}+(1-t)/r_{2}$, $0<t<1$. 
\end{theorem}

\section{Compact Paraproducts}

To prove compactness on $L^{p}(\mathbb R)$ in
the general case, that is, without the special cancellation conditions,  
we follow the classical scheme.
When $b_1=T(1)$ and $b_2=T^{*}(1)$ are arbitrary functions in $\CMO(\mathbb R)$, we construct compact paraproducts $T_{b}$ with compact 
Calder\'on-Zymund kernels such that
$T_{b_1}(1)=b_1$, $T_{b_1}^{*}(1)=0$.
Then, the operator 
$$
\tilde{T}=T-T_{b_1}-T_{b_2}^{*}
$$ 
satisfies the hypotheses of Theorem \ref{L2bounds} and so, it is compact on $L^{p}(\mathbb R)$. Furthermore, since the operators  $T_{b_1}$ and $T_{b_2}^{*}$ are compact by construction, we deduce that   $T$ is also compact on $L^{p}(\mathbb R)$. 

We remark that, as we will later see in full detail, the appropriate paraproducts are exactly the same ones as in the classical setting, with the only difference that the parameter functions $b_{i}$ belong to the space $\CMO(\mathbb R)$ instead of 
$\BMO(\mathbb R)$. 

\begin{proposition}\label{paraproducts1}
Given a function $b$ in $\CMO(\mathbb R)$, there exists a linear operator $T_b$ associated with a compact Calder\'on-Zygmund kernel 
such that $T_b$  and $T_b^{*}$ are compact on $L^{p}(\mathbb R)$ for all $1<p<\infty $ and it satisfies
$
\langle T_b(1),g\rangle =\langle b,g\rangle
$
and
$
\langle T_b(f),1\rangle =0
$, 
for all $f,g\in {\mathcal S}(\mathbb R)$.
\end{proposition}

\begin{remark}
Once the proposition is proven, the required operators to finish the program are $T_{b_{1}}^{1}=T_{b_{1}}$ and 
$T_{b_{2}}^{2}=T_{b_{2}}^{*}$.
\end{remark}
\proof

Let $(\psi_I)_{I\in {\mathcal D}}$ be a wavelet basis of $L^2(\mathbb R)$ such that $\psi_{I}$ is
$L^{2}$ normalized, supported and adapted to $I$ with constant $C$ and order $N$. 

We denote by $\phi$ a positive bump function supported and adapted to $[-1/2,1/2]$ with order $N$ and integral one. Then, we have that 
$0\leq \phi (x)\leq C(1+|x|)^{-N}$ and $|\phi'(x)|\leq C(1+|x|)^{-N}$. 
Let $(\phi_{I})_{I\in {\mathcal D}}$  be
the family of bump functions defined by $\phi_I={\mathcal T}_{c(I)}{\mathcal D}_{|I|}^{1}\phi$. Then, it is clear that each 
$\phi_{I}$ is an $L^{1}$-normalized bump function adapted to $I$ and so, satisfying
 $\phi_{I}(x)\leq C|I|^{-1}w_{I}(x)^{-N}$  and $|\phi_{I}'(x)|\leq C|I|^{-2}w_{I}(x)^{-N}$.


We define the linear operator $T_b$ by 
$$
\langle T_b(f),g\rangle =\sum_{I\in\mathcal D} \langle b, \psi_{I}\rangle \langle f, \phi_{I}\rangle \langle g,\psi_{I}\rangle
$$
Since $b\in\CMO(\mathbb R)\subset \BMO(\mathbb R)$, we know $T_b$ is bounded on $L^{p}(\mathbb R)$ and, at least formally, it satisfies
$
\langle T_b(1),g\rangle =\langle b,g\rangle
$
and
$
\langle T_b(f),1\rangle =0
$. 

We now prove that $T_b$ is compact on $L^p(\mathbb R)$ for all $1<p<\infty $. Compactness of $T_{b}^{*}$ on $L^{p}(\mathbb R)$ follows by duality. By interpolation, it suffices to verify that 
$\langle P_{M}^{\perp}(T_b(f)),g\rangle $
tends to zero uniformly for all $f,g\in {\mathcal S}(\mathbb R)$ in the unit ball of $L^{2}(\mathbb R)$. 

We know that 
$
P_{M}^{\perp }(g)=\sum_{I\in {\mathcal D}_{M}^{c}}\langle g, \psi_{I}\rangle \psi_{I}
$.
Moreover, 
since $(\psi_{I})_{I\in \cal D}$ can be chosen so that it is also a wavelet basis on $\CMO(\mathbb R)$ (see the comments in Lemma \ref{lem:cmochar}), we have 
$
P_{M}^{\perp }(b)=\sum_{I\in {\mathcal D}_{M}^{c}}\langle b, \psi_{I}\rangle \psi_{I}
$.
Then,
\begin{align*}\label{projofparaproduct}
\langle P_{M}^{\perp}(T_b(f)),g\rangle &
=\langle T_b(f),P_{M}^{\perp}(g)\rangle 
=\sum_{I\in {\mathcal D}_{M}^{c}} \langle b, \psi_{I}\rangle \langle f, \phi_{I}\rangle \langle g,\psi_{I}\rangle
\\
&
=\sum_{I\in {\mathcal D}} \langle P_{M}^{\perp}(b), \psi_{I}\rangle \langle f, \phi_{I}\rangle \langle g,\psi_{I}\rangle
=\langle T_{P_{M}^{\perp}(b)}(f),g\rangle 
\end{align*}

Now, by boundedness of $T_{P_{M}^{\perp}(b)}$, we have
$$
|\langle P_{M}^{\perp}(T_b(f)),g\rangle|
\lesssim \| P_{M}^{\perp }(b)\|_{\BMO(\mathbb R)}
\| f\|_{L^2(\mathbb R)}\| g\|_{L^{2}(\mathbb R)}
\leq \| P_{M}^{\perp }(b)\|_{\BMO(\mathbb R)}
$$
which tends to zero when $M$ tends to infinity. 


\vskip10pt
To finish the proof, we still need to show that $T_b$ belongs to the class of operators for which the theory applies.
For this, we must show that
it satisfies the integral representation of Definition \ref{intrep} with
a kernel satisfying the Definition \ref{prodCZ} of a compact Calder\'on-Zygmund kernel. 
For any $f,g\in {\mathcal S}(\mathbb R)$ with disjoint support, we have 
$$
\langle T_b(f),g\rangle 
=\int_{\mathbb R}\int_{\mathbb R} f(t)g(x)
\sum_{I\in {\mathcal D}} \langle b, \psi_{I}\rangle
\phi_{I}(t)\psi_{I}(x)dtdx
$$
As we will see, the disjointness of the supports of $f$ and $g$ implies the convergence of the infinite sum.
The kernel of $T_b$ is hence given by
$$
K(t,x)=\Big\langle b, \sum_{I\in {\mathcal D}}\phi_{I}(t)\psi_{I}(x)\psi_{I}\Big\rangle
$$
We now check that $K$ satisfies Definition \ref{prodCZ} of a compact Calder\'on-Zygmund kernel:
whenever $2|t-t'|<|t-x|$ we have
$$
|K(t,x)-K(t',x)|\lesssim \frac{|t-t'|}{|t-x|^{2}}L(|t-x|)S(|t-x|)D(|t+x|)
$$
The analogous statement for $|K(t,x)-K(t,x')|$
follows similarly.
Actually, we will prove that $|K(t,x)-K(t',x)|$ is dominated by $|t-t'|/|t-x|^{2}$ times a bounded function that tends to zero when $|t-x|$ tends to zero or to infinity and when $|t+x|$ tends to infinity (cf. Section \ref{kernelandadmissible}). 

First of all, we have that
$$
K(t,x)-K(t',x)=
\sum_{I\in {\mathcal D}} \langle b, \psi_{I}\rangle (\phi_{I}(t)-\phi_{I}(t'))\psi_{I}(x)
$$

Moreover, we notice that 
$(\phi_{I}(t)-\phi_{I}(t'))\psi_{I}(x)\neq 0$  
implies that 
$x,t,t'\in I$ for all intervals in the sum. By symmetry, we can assume $|t-x'|\leq |t-x|$. Let $I_{t,x}$ be the smallest dyadic interval containing $t$, $x$ and $t'$. Notice that 
$|I_{t,x}|/2\leq |t-x|\leq |I_{t,x}|$ and that
\begin{equation}\label{K-K}
K(t,x)-K(t',x)=
\sum_{\tiny \begin{array}{c}I\in {\mathcal D}\\ I_{t,x}\subset I\end{array}} 
\langle b, \psi_{I}\rangle (\phi_{I}(t)-\phi_{I}(t'))\psi_{I}(x)
\end{equation}

Since $b\in \CMO(\mathbb R)$, for every $\epsilon >0$ there is $M\in \mathbb N$ such that
$\| P_{M}^{\perp}(b)\|_{\BMO(\mathbb R)}+2^{-M}<\epsilon $. 
We are going to prove that 
$$
|K(t,x)-K(t',x)|\lesssim \frac{|t-t'|}{|t-x|^{2}}\epsilon 
$$
when $|t-x|>2^{M}$, $|t+x|>M2^{M+1}$ or $|t-x|<2^{-3M/2}(1+\| b\|_{\BMO(\mathbb R)})^{-1/2}$.

\vskip10pt
1) When $|t-x|>2^{M}$ we have that all intervals $I$ in the sum satisfy $|I|>|t-x|>2^{M}$ and so, 
$I\in {\mathcal D}_{M}^{c}$. We can rewrite equation (\ref{K-K}) as 
\begin{align*}
K(t,x)-K(t',x)&=
\sum_{I\in {\mathcal D}_{M}^{c}} \langle b, \psi_{I}\rangle (\phi_{I}(t)-\phi_{I}(t'))\psi_{I}(x)
\\
&=\Big\langle P_{M}^{\perp}(b), \sum_{I\in {\mathcal D}} (\phi_{I}(t)-\phi_{I}(t'))\psi_{I}(x)\psi_{I}
\Big\rangle
\end{align*}
Then, 
$$
|K(t,x)-K(t',x)|
\leq \| P_{M}^{\perp}(b)\|_{\BMO(\mathbb R)}
\Big\| \sum_{I\in {\mathcal D}} (\phi_{I}(t)-\phi_{I}(t'))\psi_{I}(x)
\psi_{I}\Big\|_{H^1(\mathbb R)}
$$
With the use of the square function $S$, the $H^1$-norm is equivalent to
$$
\Big\| S\Big(\sum_{I\in {\mathcal D}} (\phi_{I}(t)-\phi_{I}(t'))\psi_{I}(x)\psi_{I}\Big)\Big\|_{L^1(\mathbb R)}
$$
$$
= \int_{\mathbb R} \Big(\sum_{\tiny \begin{array}{c}I\in {\mathcal D}\\ I_{t,x}\subset I\end{array}} 
|\phi_{I}(t)-\phi_{I}(t')|^2
|\psi_{I}(x)|^2\frac{\chi_{I}(y)}{|I|}\Big)^{1/2}dy
$$
Let $(I_{k})_{k\in \mathbb N}$ be the family of dyadic intervals such that
$I_{x,t}\subset I_k$ with 
$|I_{k}|=2^{k}|I_{x,t}|$ and $I_{-1}=\emptyset $. Then, we rewrite the previous integral as
$$
\sum_{k\geq 0} \int_{I_{k}\backslash I_{k-1}}\Big(\sum_{j\geq k} |\phi_{I_{j}}(t)-\phi_{I_{j}}(t')|^2
|\psi_{I_{j}}(x)|^2\frac{\chi_{I_{j}}(y)}{|I_{j}|}\Big)^{1/2}dy
$$

Now, since the integrand is constant on every interval $I_{k}\backslash I_{k-1}$, 
the integral can be bounded by 
\begin{equation}\label{K-Kpara}
\sum_{k\geq 0} \Big(\sum_{j\geq k} |\phi_{I_{j}}(t)-\phi_{I_{j}}(t')|^2
|\psi_{I_{j}}(x)|^2\frac{1}{|I_{j}|}\Big)^{1/2} |I_{k}\backslash I_{k-1}|
\end{equation}
Since
$
|\phi_{I}(t)-\phi_{I}(t')|\leq \|\partial_{t}\phi_{I}\|_{L^{\infty }(\mathbb R)}|t-t'|
\leq C|I|^{-2}
|t-t'|
$
and $|\psi_{I}(x)|\leq C|I|^{-1/2}$,  we bound (\ref{K-Kpara}) by a constant times
$$
\sum_{k\geq 0} \Big(\sum_{j\geq k} \frac{1}{|I_{j}|^4}
\frac{1}{|I_{j}|}\frac{1}{|I_{j}|}\Big)^{\frac{1}{2}}|I_{k}||t-t'|
$$
$$
=\sum_{k\geq 0} \Big(\sum_{j\geq k} 
\frac{1}{2^{6j}}\Big)^{\frac{1}{2}}2^{k}\frac{|t-t'|}{|I_{t,x}|^2}
\lesssim \sum_{k\geq 0} \frac{1}{2^{2k}}\frac{|t-t'|}{|t-x|^2}
\lesssim \frac{|t-t'|}{|t-x|^2}
$$
Therefore, we obtain
$$
|K(t,x)-K(t',x)|
\lesssim \| P_{M}^{\perp}(b)\|_{\BMO(\mathbb R)}
\frac{|t-t'|}{|t-x|^2}
\leq \epsilon \frac{|t-t'|}{|t-x|^2}
$$
by the choice of $M$.

2) We work now the case $|t+x|>M2^{M+1}$. Since
every interval $I$ in the sum satisfies $I_{t,x}\subset I$, we have $(t+x)/2\in I$ and so, 
$|c(I)-(t+x)/2|<|I|/2$. 

Now, if $|I|>2^{M}$ we have  $I\in {\mathcal D}_{M}^{c}$.
On the other hand, if $|I|\leq 2^{M}$ 
$$
\rdist(I,{\mathbb B}_{2^{M}})
\geq \frac{|c(I)|+2^{M-1}+|I|/2}{2^{M}}
\geq \frac{|t+x|/2+2^{M-1}}{2^{M}}
>M
$$
and also $I\in {\mathcal D}_{M}^{c}$. Therefore, from the previous case we can conclude
$$
|K(t,x)-K(t',x)|
\lesssim \| P_{M}^{\perp}(b)\|_{\BMO(\mathbb R)}
\frac{|t-t'|}{|t-x|^2}
\leq \epsilon \frac{|t-t'|}{|t-x|^2}
$$

3) The last case, when $|t-x|<2^{-3M/2}(1+\| b\|_{\BMO(\mathbb R)})^{-1/2}$, is a more involved. 
Those intervals in the sum such that 
$|I|<2^{-M}$ or $|I|>2^{M}$ satisfy that $I\in {\mathcal D}_{M}^{c}$ and may be taken care of as in the previous cases.

However, those intervals such that $2^{-M}\leq |I|\leq 2^{M}$ may belong to ${\mathcal D}_{M}$ and the previous argument can not be used. Instead, we reason as follows. The terms under consideration 
are given by those intervals $I\in {\mathcal D}$ such that $I_{t,x}\subset  I$ and  $2^{-M}\leq |I|\leq 2^{M}$ and so, they can be parametrized by size as $|I_{k}|=2^{k}|I_{t,x}|$ with 
$-M-\log|I_{t,x}|\leq k \leq M-\log|I_{t,x}|$. 
Notice that, since $|I_{t,x}|\leq 2|t-x|<2^{-3M/2+1}$, we have $-M-\log|I_{t,x}|>M/2-1>0$.
We are left with the sum
$$
\sum_{-M-\log|I_{t,x}|\leq k \leq M-\log|I_{t,x}|} \langle b, \psi_{I_{k}}\rangle (\phi_{I_{k}}(t)-\phi_{I_{k}}(t'))\psi_{I_{k}}(x)
$$
We use $|\langle b, \psi_{I}\rangle |\lesssim \| b\|_{\BMO(\mathbb R)}|I|^{1/2}$, 
$|\phi_{I}(t)-\phi_{I}(t')|\lesssim |I|^{-2}
|t-t'|$ and 
$|\psi_{I}(x)|\lesssim |I|^{-1/2}
$, to bound the sum by
$$
\| b\|_{\BMO(\mathbb R)}\hspace{-.3cm}\sum_{-M-\log|I_{t,x}|\leq k } 
\frac{|t-t'|}{|I_{k}|^{2}}
=\| b\|_{\BMO(\mathbb R)}\hspace{-.3cm}
\sum_{-M-\log|I_{t,x}|\leq k }2^{-2k}  \frac{|t-t'|}{|I_{t,x}|^{2}}
$$
$$
\lesssim \| b\|_{\BMO(\mathbb R)}2^{2M}|I_{t,x}|^{2}\frac{|t-t'|}{|I_{t,x}|^{2}}
\lesssim  2^{-M}\frac{|t-t'|}{|I_{t,x}|^{2}}\lesssim \epsilon \frac{|t-t'|}{|t-x|^2}
$$
since $|I_{t,x}|\leq 2|t-x|<2^{-3M+2}(1+\| b\|_{\BMO(\mathbb R)})^{-\frac{1}{2}}$ and the choice of $M$.

We note that in dealing with $|K(t,x)-K(t,x')|$ we use the analogue inequalities
$|\phi_{I}(t)|\lesssim |I|^{-1}
$
and $|\psi_{I}(x)-\psi_{I}(x')|\lesssim |I|^{-3/2}
|x-x'|$.

\bibliographystyle{plain}

\begin{thebibliography}{99}



\bibitem{Chui}
C.~K. Chui.
\newblock {\em {An introduction to wavelets}}.
\newblock Academic Press Professional, Inc., San Diego, CA, USA, 1992.

\bibitem{Cot}
M.~Cotlar.
\newblock {\em Continuity conditions for potential and {H}ilbert operators},
  volume {F}asc. 2.
\newblock {Departamento de Matematica, Facultad de Ciencias Exactas y
  Naturales, Universidad Nacional de Buenos Aires}, 1959.

\bibitem{DJ}
G.~David and J.~L. Journ\'e.
\newblock {A boundedness criterion for generalized Calder\'on-Zygmund
  operators}.
\newblock {\em Ann. of Math.}, 120:371--397, 1984.

\bibitem{Fabes}
E.~Fabes, M.~Jodeit, and N.~Rivi\`{e}re.
\newblock Potential techniques for boundary value problems on {C}1-domains.
\newblock {\em Acta Math.}, 141(3--4):165--186, 1978.

\bibitem{Fab}
M.~Fabian, P.~Habala, V.~Montesinos, and V.~Zizler.
\newblock {\em Banach space theory. {T}he basis for linear and nonlinear
  analysis}.
\newblock {CMS} Books in Mathematics/Ouvrages de Math{\'e}matiques de la SMC.
  Springer, New York, 2011.

\bibitem{HerWeiss}
Eugenio Hernandez and Guido Weiss.
\newblock {\em A first course on wavelets}.
\newblock CRC Press, Inc., Boca Raton, Florida, 1996.

\bibitem{hmtimb}
S.~Hofmann, M.~Mitrea, and M.~Taylor.
\newblock Singular integrals and elliptic boundary problems on regular
  {S}emmes-{K}enig-{T}oro domains.
\newblock {\em Int. Math. Res. Not. IMRN}, 2010(14):2567--2865, 2010.

\bibitem{Kra}
M.~A. Krasnoselski.
\newblock On a theorem of {M}. {R}iesz.
\newblock {\em Dokl. Akad. Nauk}, 131:246--248, 1959.

\bibitem{LTW}
M.~Lacey, E.~Terwilleger, and B.~Wick.
\newblock {Remarks on Product VMO}.
\newblock {\em Proc. Amer. Math. Soc}, 2:465--474, 2006.

\bibitem{OV}
J.~F. Olsen and P.~Villarroya.
\newblock {Endpoint estimates for compact Calder\'on-Zygmund operators}.
\newblock {\em Isr. J. Math.}, To appear.

\bibitem{Roch}
R.~Rochberg.
\newblock Toeplitz and {H}ankel operators on the {P}aley-{W}iener space.
\newblock {\em Integral Equations Operator Theory}, 10(2):187--235, 1987.

\bibitem{TLec}
C.~Thiele.
\newblock {\em Wave packet analysis}, volume 105 of {\em {CBMS} Regional
  Conference Series in Mathematics}.
\newblock Amer. Math. Soc., Providence RI, 2006.

\bibitem{Uchi}
A.~Uchiyama.
\newblock On the compactness of operators of {H}ankel type.
\newblock {\em T\^ohoku Math. J. (2)}, 30(1):163--171, 1978.

\end{thebibliography}

\end{document}